\documentclass{article}
\usepackage{amssymb}
\usepackage{amsmath}
\usepackage{amsfonts}

\newtheorem{theorem}{Theorem}

\newtheorem{definition}[theorem]{Definition}
\newtheorem{example}[theorem]{Example}

\newtheorem{lemma}[theorem]{Lemma}
\newtheorem{notation}[theorem]{Notation}
\newtheorem{problem}[theorem]{Problem}
\newtheorem{proposition}[theorem]{Proposition}
\newtheorem{remark}[theorem]{Remark}

\newenvironment{proof}[1][Proof]{\noindent\textbf{#1.} }{\ \rule{0.5em}{0.5em}}
\input{tcilatex}
\begin{document}

\title{Notes on area operator, geometric $2$-rough paths and Young integral
when $p^{-1}+q^{-1}=1$}
\author{Danyu Yang\thanks{%
University of Oxford and Oxford-man Institute, Email: yangd@maths.ox.ac.uk}}
\maketitle

\begin{abstract}
When the norm on continuous bounded variation paths weakened to $2$%
-variation, the area operator is not continuous nor bounded, but is closable
in 2-rough norm, and paths in the closure (i.e. paths which admits an
enhancement into a geometric $2$-rough path) is not linear.

For path $\gamma $ with vanishing $2$-variation, the Riemann-Stieltjes
integral $2^{-1}\tiint_{s<u_{1}<u_{2}<t}\left[ d\gamma \left( u_{1}\right)
,d\gamma \left( u_{2}\right) \right] $ is the only possible candidate to
enhance $\gamma $ into a geometric $2$-rough path, but the integral may not
exist, so not every path with vanishing $2$-variation admits an enhancement.

Young integral is extended to the case $p^{-1}+q^{-1}=1$ by assuming a finer
scale continuity. As a consequence, when $p=q=2$, by adding a $\log $ term
(and $\log \log $ term, etc.) in the modulus of continuity, there exists a
sequence of nested spaces of enhancible paths.
\end{abstract}

\section{Definitions and notations}

\noindent Firstly, we define $p$-variation seminorm on the space of
continuous paths, which is important in rough path theory (see \cite{T. J.
Lyons notes}, \cite{T. J. Lyons and Z. Qian} and \cite{P. Friz}).

\begin{definition}
A finite set of points $D=\left\{ t_{j}\right\} _{j=0}^{n}$ is said to be a
finite partition of interval $\left[ 0,T\right] $, if $0=t_{0}<t_{1}<\dots
<t_{n}=T$. \ 
\end{definition}

\begin{notation}
Suppose $D=\left\{ t_{j}\right\} _{j=0}^{n}$ is a finite partition of $\left[
0,T\right] $. Denote $\left\vert D\right\vert :=\max_{0\leq j\leq
n-1}\left\{ \left\vert t_{j+1}-t_{j}\right\vert \right\} $ as the mesh of $D$%
.
\end{notation}

\begin{notation}
Denote $\mathcal{V}$ as a Banach space with norm $\left\Vert \cdot
\right\Vert $.
\end{notation}

\begin{notation}
For $T>0$, denote $C\left( \left[ 0,T\right] ,\mathcal{V}\right) :=\left\{
\gamma |\gamma :\left[ 0,T\right] \rightarrow \mathcal{V}\text{ is continuous%
}\right\} $; denote $\bigtriangleup _{\left[ 0,T\right] }:=\left\{ \left(
s,t\right) |0\leq s\leq t\leq T\right\} $ and%
\begin{equation*}
C\left( \bigtriangleup _{\left[ 0,T\right] },\mathcal{V}\right) :=\left\{
\alpha |\alpha :\bigtriangleup _{\left[ 0,T\right] }\rightarrow \mathcal{V}%
\text{ is continuous, }\alpha \left( t,t\right) =0,\forall t\in \left[ 0,T%
\right] \right\} .
\end{equation*}
\end{notation}

\begin{definition}
Suppose $\alpha \in C\left( \bigtriangleup _{\left[ 0,T\right] },\mathcal{V}%
\right) $. For $p>0$, define the $p$-variation of $\alpha $ on $\left[ 0,T%
\right] $ as%
\begin{equation}
\left\Vert \alpha \right\Vert _{p-var,\left[ 0,T\right] }:=\left(
\sup_{D\subset \left[ 0,T\right] }\sum_{j,t_{j}\in D}\left\Vert \alpha
\left( t_{j},t_{j+1}\right) \right\Vert ^{p}\right) ^{\frac{1}{p}},
\label{Definition p-variation}
\end{equation}%
where the supremum is over all finite partitions of $\left[ 0,T\right] $.

\noindent When $p=\infty $, define $\left\Vert \alpha \right\Vert _{\infty
-var,\left[ 0,T\right] }:=\sup_{0\leq s<t\leq T}\left\Vert \alpha \left(
s,t\right) \right\Vert $.
\end{definition}

For any fixed $\alpha \in C\left( \bigtriangleup _{\left[ 0,T\right] },%
\mathcal{V}\right) $, the function $p\mapsto \left\Vert \alpha \right\Vert
_{p-var,\left[ 0,T\right] }$ on $p\in (0,\infty ]$ is non-increasing and
continuous where it is finite.

\begin{definition}
Suppose $\alpha \in C\left( \bigtriangleup _{\left[ 0,T\right] },\mathcal{V}%
\right) $. Then $\alpha $ is said to be of vanishing $p$-variation for some $%
p>0$, if%
\begin{equation}
\lim_{\delta \rightarrow 0}\omega _{p}\left( \alpha ,\delta \right)
:=\lim_{\delta \rightarrow 0}\left( \sup_{D\subset \left[ 0,T\right]
,\left\vert D\right\vert \leq \delta }\sum_{j,t_{j}\in D}\left\Vert \alpha
\left( t_{j},t_{j+1}\right) \right\Vert ^{p}\right) ^{\frac{1}{p}}=0\text{.}
\label{Definition vanishing p-variation}
\end{equation}
\end{definition}

\begin{definition}
Suppose $\gamma \in C\left( \left[ 0,T\right] ,\mathcal{V}\right) $. Define $%
\widetilde{\gamma }\in C\left( \bigtriangleup _{\left[ 0,T\right] },\mathcal{%
V}\right) $ by setting%
\begin{equation}
\widetilde{\gamma }\left( s,t\right) :=\gamma \left( t\right) -\gamma \left(
s\right) ,\forall 0\leq s\leq t\leq T.
\label{Definition of a function on triangle from a path}
\end{equation}%
Then define $\left\Vert \gamma \right\Vert _{p-var,\left[ 0,T\right]
}:=\left\Vert \widetilde{\gamma }\right\Vert _{p-var,\left[ 0,T\right] }$, $%
\omega _{p}\left( \gamma ,\delta \right) :=\omega _{p}\left( \widetilde{%
\gamma },\delta \right) $ and that $\gamma $ is said to be of vanishing $p$%
-variation if $\lim_{\delta \rightarrow 0}\omega _{p}\left( \gamma ,\delta
\right) =0$.
\end{definition}

Both $p$-variation norm and being of vanishing $p$-variation are invariant
under reparametrisation (i.e. continuous non-decreasing $\varphi :\left[ 0,T%
\right] \rightarrow \overline{%
\mathbb{R}
^{+}}$, continuity preserves compactness and being non-decreasing preserves
the order).

\begin{notation}
For $p>0$, denote $C^{p-var}\left( \bigtriangleup _{\left[ 0,T\right] },%
\mathcal{V}\right) $\ and $C^{0,p-var}\left( \bigtriangleup _{\left[ 0,T%
\right] },\mathcal{V}\right) $ as subspaces of $C\left( \bigtriangleup _{%
\left[ 0,T\right] },\mathcal{V}\right) $:%
\begin{eqnarray*}
C^{p-var}\left( \bigtriangleup _{\left[ 0,T\right] },\mathcal{V}\right)
&:&=\left\{ \alpha \in C\left( \bigtriangleup _{\left[ 0,T\right] },\mathcal{%
V}\right) |\left\Vert \alpha \right\Vert _{p-var,\left[ 0,T\right] }<\infty
\right\} , \\
C^{0,p-var}\left( \bigtriangleup _{\left[ 0,T\right] },\mathcal{V}\right)
&:&=\left\{ \alpha \in C\left( \bigtriangleup _{\left[ 0,T\right] },\mathcal{%
V}\right) |\lim_{\delta \rightarrow 0}\omega _{p}\left( \alpha ,\delta
\right) =0\right\} .
\end{eqnarray*}
Similarly, for $p\geq 1$, denote $C^{p-var}\left( \left[ 0,T\right] ,%
\mathcal{V}\right) $\ and $C^{0,p-var}\left( \left[ 0,T\right] ,\mathcal{V}%
\right) $ as subspaces of continuous paths $C\left( \left[ 0,T\right] ,%
\mathcal{V}\right) $:%
\begin{eqnarray*}
C^{p-var}\left( \left[ 0,T\right] ,\mathcal{V}\right) &:&=\left\{ \gamma \in
C\left( \left[ 0,T\right] ,\mathcal{V}\right) |\left\Vert \gamma \right\Vert
_{p-var,\left[ 0,T\right] }<\infty \right\} , \\
C^{0,p-var}\left( \left[ 0,T\right] ,\mathcal{V}\right) &:&=\left\{ \gamma
\in C\left( \left[ 0,T\right] ,\mathcal{V}\right) |\lim_{\delta \rightarrow
0}\omega _{p}\left( \gamma ,\delta \right) =0\right\} .
\end{eqnarray*}
\end{notation}

Then (based on Proposition5.6 \cite{P. Friz}), $C^{0,p-var}\left(
\bigtriangleup _{\left[ 0,T\right] },\mathcal{V}\right) \subset
C^{p-var}\left( \bigtriangleup _{\left[ 0,T\right] },\mathcal{V}\right) $
and $C^{0,p-var}\left( \left[ 0,T\right] ,\mathcal{V}\right) \subset
C^{p-var}\left( \left[ 0,T\right] ,\mathcal{V}\right) $. Moreover, $%
C^{p-var}\left( \left[ 0,T\right] ,\mathcal{V}\right) $ can be treated as a
subspace of $C^{p-var}\left( \bigtriangleup _{\left[ 0,T\right] },\mathcal{V}%
\right) $ in which functions on $\bigtriangleup _{\left[ 0,T\right] }$\ are
generated from paths (by $\left( \ref{Definition of a function on triangle
from a path}\right) $). Similarly, $C^{0,p-var}\left( \left[ 0,T\right] ,%
\mathcal{V}\right) $ can be treated as a subspace of $C^{0,p-var}\left(
\bigtriangleup _{\left[ 0,T\right] },\mathcal{V}\right) $. Therefore, we
have the inclusions of spaces:

\begin{equation*}
\begin{array}{ccc}
C^{0,p-var}\left( \bigtriangleup _{\left[ 0,T\right] },\mathcal{V}\right) & 
\subset & C^{p-var}\left( \bigtriangleup _{\left[ 0,T\right] },\mathcal{V}%
\right) \\ 
\cup &  & \cup \\ 
C^{0,p-var}\left( \left[ 0,T\right] ,\mathcal{V}\right) & \subset & 
C^{p-var}\left( \left[ 0,T\right] ,\mathcal{V}\right)%
\end{array}%
\end{equation*}

For paths in $C^{0,p-var}\left( \left[ 0,T\right] ,\mathcal{V}\right) $ an
explicit characterization is available.

\begin{notation}
Suppose $\gamma :\left[ 0,T\right] \rightarrow \mathcal{V}$ is a continuous
path, and $D=\left\{ t_{j}\right\} _{j}$ a finite partition of $\left[ 0,T%
\right] $. Denote $\gamma ^{D}$ as the piecewise linear path which coincides
with $\gamma $ on points in $D$, i.e.%
\begin{equation}
\gamma ^{D}\left( t\right) =\frac{t-t_{j}}{t_{j+1}-t_{j}}\gamma \left(
t_{j+1}\right) +\frac{t_{j+1}-t}{t_{j+1}-t_{j}}\gamma \left( t_{j}\right) 
\text{, \ }t\in \left[ t_{j},t_{j+1}\right] \text{.}
\label{Definition gammaD}
\end{equation}
\end{notation}

Then when $1<p<\infty $, for $\gamma \in C^{p-var}\left( \left[ 0,T\right] ,%
\mathcal{V}\right) $, the following three statements are equivalent
(Wiener's characterization, Thm5.31 \cite{P. Friz}): 
\begin{eqnarray}
&&\gamma \in C^{0,p-var}\left( \left[ 0,T\right] ,\mathcal{V}\right) 
\label{equivalent relation for vanishing 2-var path} \\
&\Leftrightarrow &\ \exists \left\{ \gamma _{n}\right\} _{n=0}^{\infty }\in
C^{1-var}\left( \left[ 0,T\right] ,\mathcal{V}\right) \text{ s.t.}\
\lim_{n\rightarrow \infty }\left\Vert \gamma _{n}-\gamma \right\Vert _{p-var,%
\left[ 0,T\right] }=0  \notag \\
&\Leftrightarrow &\lim_{\left\vert D\right\vert \rightarrow 0}\left\Vert
\gamma ^{D}-\gamma \right\Vert _{p-var.\left[ 0,T\right] }=0.  \notag
\end{eqnarray}%
(In Thm5.31 \cite{P. Friz}, the equivalency is identified for paths taking
value in $%
\mathbb{R}
^{d}$, but can be extended to paths taking value in Banach space $\mathcal{V}
$.) When $p=1$, the latter two are equivalent to the absolutely continuity
of $\gamma $ (Proposition1.32 \cite{P. Friz}), while $\gamma $ is of
vanishing $1$-variation if and only if it is a constant.

\begin{notation}
Denote $\otimes $ as tensor product. Suppose $\left( \mathcal{V},\left\Vert
\cdot \right\Vert _{\mathcal{V}}\right) $ and $\left( \mathcal{U},\left\Vert
\cdot \right\Vert _{\mathcal{U}}\right) $ are two Banach spaces. Denote $%
\left( \mathcal{V}\otimes \mathcal{U},\left\Vert \cdot \right\Vert _{%
\mathcal{V}\otimes \mathcal{U}}\right) $ is the Banach space defined as the
completion of $\left\{ \sum_{i=1}^{n}v_{i}\otimes u_{i},\text{ }v_{i}\in 
\mathcal{V},u_{i}\in \mathcal{U},\text{ }n\geq 1\text{ }\right\} $ w.r.t. $%
\left\Vert \cdot \right\Vert _{\mathcal{V}\otimes \mathcal{U}}$. 
\end{notation}

\begin{notation}
For Banach space  $\left( \mathcal{V},\left\Vert \cdot \right\Vert _{%
\mathcal{V}}\right) $ and $v_{1},v_{2}\in \mathcal{V}$, denote $\left[
v_{1},v_{2}\right] :=v_{1}\otimes v_{2}-v_{2}\otimes v_{1}$. Denote $\left( %
\left[ \mathcal{V},\mathcal{V}\right] ,\left\Vert \cdot \right\Vert _{%
\mathcal{V}\otimes \mathcal{V}}\right) $ as the Banach space defined as the
completion of $\left\{ \sum_{i=1}^{n}\left[ v_{1}^{i},v_{2}^{i}\right] ,%
\text{ }v_{1}^{i},v_{2}^{i}\in \mathcal{V},\text{\ }n\geq 1\right\} $ w.r.t. 
$\left\Vert \cdot \right\Vert _{\mathcal{V}\otimes \mathcal{V}}$.
\end{notation}

In this manuscript, we assume  $\left\Vert v\otimes u\right\Vert _{\mathcal{V%
}\otimes \mathcal{U}}\leq \left\Vert v\right\Vert _{\mathcal{V}}\left\Vert
u\right\Vert _{\mathcal{U}}$, $\forall v\in \mathcal{V}$, $\forall u\in 
\mathcal{U}$.

\begin{definition}
\label{Definition area of two paths}Suppose $\mathcal{V}$ and $\mathcal{U}$
are two Banach spaces, and $\gamma _{1}\in C^{1-var}\left( \left[ 0,T\right]
,\mathcal{V}\right) $, $\gamma _{2}\in C^{1-var}\left( \left[ 0,T\right] ,%
\mathcal{U}\right) $.

Define the iterated integral of $\gamma _{1}$ and $\gamma _{2}$, $I\left(
\gamma _{1},\gamma _{2}\right) \in C\left( \bigtriangleup _{\left[ 0,T\right]
},\mathcal{V}\otimes \mathcal{U}\right) $, as%
\begin{equation*}
I\left( \gamma _{1},\gamma _{2}\right) \left( s,t\right)
=\diint\nolimits_{s<u_{1}<u_{2}<t}d\gamma _{1}\left( u_{1}\right) \otimes
d\gamma _{2}\left( u_{2}\right) \text{, }\forall 0\leq s\leq t\leq T
\end{equation*}

When $\mathcal{U}=\mathcal{V}$ (so $\gamma _{i}\in C^{1-var}\left( \left[ 0,T%
\right] ,\mathcal{V}\right) $, $i=1,2$), define $A\left( \gamma _{1},\gamma
_{2}\right) \in C\left( \bigtriangleup _{\left[ 0,T\right] },\left[ \mathcal{%
V},\mathcal{V}\right] \right) $ as 
\begin{equation*}
A\left( \gamma _{1},\gamma _{2}\right) \left( s,t\right) =\frac{1}{2}%
\diint\nolimits_{s<u_{1}<u_{2}<t}\left[ d\gamma _{1}\left( u_{1}\right)
,d\gamma _{2}\left( u_{2}\right) \right] \text{, }\forall 0\leq s\leq t\leq
T.
\end{equation*}
\end{definition}

The notation $I\left( \gamma _{1},\gamma _{2}\right) $ is used in the proof
of extension of Young integral, $A\left( \gamma _{1},\gamma _{2}\right) $ is
used to estimate $A\left( \gamma \right) $ when $\gamma =\gamma _{1}+\gamma
_{2}$.

\begin{definition}[area]
\label{Definition of area}Suppose $\gamma \in C^{1-var}\left( \left[ 0,T%
\right] ,\mathcal{V}\right) $. Define the area of $\gamma $, $A\left( \gamma
\right) \in C\left( \bigtriangleup _{\left[ 0,T\right] },\left[ \mathcal{V},%
\mathcal{V}\right] \right) :=A\left( \gamma ,\gamma \right) $.
\end{definition}

\begin{definition}[area operator]
The area operator is the operator defined on the set of continuous bounded
variation paths which sends $\gamma \ $to $A\left( \gamma \right) $.
\end{definition}

The area operator can be extended where the Riemann-Stieltjes integral $%
A\left( \gamma \right) $ is well-defined (e.g. $\mathcal{G}_{2}\left( 
\mathcal{V}\right) $ defined below).

When $\gamma \in C^{1-var}\left( \left[ 0,T\right] ,\mathcal{V}\right) $,
based on Young integral (i.e. $\left( \ref{Reimann area inequality}\right) $
below), 
\begin{equation*}
A\left( \gamma \right) \in C^{\frac{1}{2}-var}\left( \bigtriangleup _{\left[
0,T\right] },\left[ \mathcal{V},\mathcal{V}\right] \right) \subseteq
C^{0,1-var}\left( \bigtriangleup _{\left[ 0,T\right] },\left[ \mathcal{V},%
\mathcal{V}\right] \right) .
\end{equation*}%
On the other hand, because $\omega _{1}\left( \alpha ,\delta \right) \leq
\left\Vert \alpha -\alpha _{n}\right\Vert _{1-var}+\omega _{1}\left( \alpha
_{n},\delta \right) $ ($\omega _{1}$ defined at $\left( \ref{Definition
vanishing p-variation}\right) $), $C^{0,1-var}\left( \bigtriangleup _{\left[
0,T\right] },\left[ \mathcal{V},\mathcal{V}\right] \right) $ is closed under 
$1$-variation. Thus,%
\begin{equation}
\overline{\{A\left( \gamma \right) |\gamma \in C^{1-var}\left( \left[ 0,T%
\right] ,\mathcal{V}\right) \}}^{1-var}\hspace{-0.02in}\subseteq \hspace{%
-0.02in}C^{0,1-var}\left( \bigtriangleup _{\left[ 0,T\right] },\left[ 
\mathcal{V},\mathcal{V}\right] \right) .  \label{Area is of vanishing 1-var}
\end{equation}

\begin{definition}[weak geometric $2$-rough path]
Suppose $\gamma \in C\left( \left[ 0,T\right] ,\mathcal{V}\right) $, $\alpha
\in C\left( \bigtriangleup _{\left[ 0,T\right] },\left[ \mathcal{V},\mathcal{%
V}\right] \right) $. Then $\Gamma :=\left( \gamma ,\alpha \right) \in
C\left( \bigtriangleup _{\left[ 0,T\right] },\mathcal{V\oplus }\left[ 
\mathcal{V},\mathcal{V}\right] \right) $ is called a weak geometric $2$%
-rough path, if for any $0\leq s\leq u\leq t\leq T$,%
\begin{gather}
\alpha \left( s,t\right) =\alpha \left( s,u\right) +\alpha \left( u,t\right)
+\frac{1}{2}\left[ \gamma \left( u\right) -\gamma \left( s\right) ,\gamma
\left( t\right) -\gamma \left( u\right) \right] \text{,}
\label{property of multiplicativity} \\
\text{and }\left\Vert \Gamma \right\Vert _{G^{\left( 2\right) }}:=\left(
\left\Vert \gamma \right\Vert _{2-var}^{2}+\left\Vert \alpha \right\Vert
_{1-var}\right) ^{\frac{1}{2}}<\infty \text{.}  \notag
\end{gather}
\end{definition}

Property at $\left( \ref{property of multiplicativity}\right) $ is called
multiplicativity. $\left\Vert \cdot \right\Vert _{G^{\left( 2\right) }}$ is
called $2$-rough norm.

\begin{definition}[geometric $2$-rough path]
\label{Definition of geometric 2-rough path}$\Gamma :=\left( \gamma ,\alpha
\right) \in C\left( \bigtriangleup _{\left[ 0,T\right] },\mathcal{V\oplus }%
\left[ \mathcal{V},\mathcal{V}\right] \right) $ is a geometric $2$-rough
path, if there exist $\left\{ \gamma _{n}\right\} _{n}\subset
C^{1-var}\left( \left[ 0,T\right] ,\mathcal{V}\right) $ such that%
\begin{equation*}
\lim_{n\rightarrow \infty }\left\Vert \Gamma -\left( \gamma _{n},A\left(
\gamma _{n}\right) \right) \right\Vert _{G^{\left( 2\right) }}=0\text{.}
\end{equation*}
\end{definition}

One can verify that a geometric $2$-rough path is a weak geometric $2$-rough
path.

Thus, if $\left( \gamma ,\alpha \right) $ is a geometric $2$-rough path,
then $\gamma $ is of vanishing $2$-variation (because of $\left( \ref%
{equivalent relation for vanishing 2-var path}\right) $) and $\alpha $ is of
vanishing $1$-variation (because of $\left( \ref{Area is of vanishing 1-var}%
\right) $). Suppose $\gamma \in C^{0,2-var}\left( \left[ 0,T\right] ,%
\mathcal{V}\right) $, then we say $\gamma $ can be enhanced into a geometric 
$2$-rough path (or enhancible), if there exists $\alpha \in
C^{0,1-var}\left( \bigtriangleup _{\left[ 0,T\right] },\left[ \mathcal{V},%
\mathcal{V}\right] \right) $ such that $\left( \gamma ,\alpha \right) $ is a
geometric $2$-rough path.

\begin{notation}
\label{Notation G_2}Denote $\mathcal{G}_{2}\left( \mathcal{V}\right)
\subseteq C^{0,2-var}\left( \left[ 0,T\right] ,\mathcal{V}\right) $ as the
set of paths which admits an enhancement into a geometric $2$-rough path.
\end{notation}

$\mathcal{G}_{2}\left( \mathcal{V}\right) $ is invariant under
reparametrisation and contains, e.g. $C^{1-var}\left( \left[ 0,T\right] ,%
\mathcal{V}\right) $.

\section{Questions, answers and results}

Suppose $\gamma _{1}$ and $\gamma _{2}$ are continuous paths on $\left[ 0,T%
\right] $, consider the Riemann-Stieltjes integrals (whenever they exist):%
\begin{gather}
\alpha \left( s,t\right) =\diint\nolimits_{s<u_{1}<u_{2}<t}d\gamma
_{1}\left( u_{1}\right) \otimes d\gamma _{2}\left( u_{2}\right) ,\left(
s,t\right) \in \bigtriangleup _{\left[ 0,T\right] }
\label{Reimann integral area} \\
i\left( t\right) =\int_{0}^{t}\gamma _{1}\left( u\right) \otimes d\gamma
_{2}\left( u\right) ,t\in \left[ 0,T\right] .  \notag
\end{gather}%
If $\gamma _{1}$ is continuous and $\gamma _{2}$ of bounded variation, then $%
\alpha $ and $i$ are of bounded variation, and%
\begin{equation*}
\left\Vert \alpha \right\Vert _{1-var,\left[ 0,T\right] }\vee \left\Vert
i\right\Vert _{1-var,\left[ 0,T\right] }\leq \left\Vert \gamma
_{1}\right\Vert _{\infty -var,\left[ 0,T\right] }\left\Vert \gamma
_{2}\right\Vert _{1-var,\left[ 0,T\right] }.
\end{equation*}%
Young \cite{L. C. Young} demonstrated that, if $\gamma _{1}$ is of finite $p$%
-variation, $\gamma _{2}$ of finite $q$-variation, and $p>1$, $q>1$, $%
p^{-1}+q^{-1}>1$, then $\alpha $ and $i$ are still well-defined, and (based
on Thm 1.16 in \cite{T. J. Lyons notes}) 
\begin{gather}
\left\Vert \alpha \right\Vert _{\left( p^{-1}+q^{-1}\right) ^{-1}-var,\left[
0,T\right] }\leq C_{p,q}\left\Vert \gamma _{1}\right\Vert _{p-var,\left[ 0,T%
\right] }\left\Vert \gamma _{2}\right\Vert _{q-var,\left[ 0,T\right] }\text{,%
}  \label{Reimann area inequality} \\
\left\Vert i\right\Vert _{p-var,\left[ 0,T\right] }\leq \left(
C_{p,q}+1\right) \left\Vert \gamma _{1}\right\Vert _{p-var,\left[ 0,T\right]
}\left\Vert \gamma _{2}\right\Vert _{q-var,\left[ 0,T\right] }.  \notag
\end{gather}%
($\alpha $ is of finite $\left( p^{-1}+q^{-1}\right) ^{-1}$-variation, $%
\left( p^{-1}+q^{-1}\right) ^{-1}<1$; $i$ is of finite $q$-variation, $q>1$,
the same as $\gamma _{2}$.) However, the existence of integral is
problematic when $p^{-1}+q^{-1}=1$. In the special case $\gamma _{1}=\gamma
_{2}:=\gamma $, the definition of $\int \gamma \otimes d\gamma $ is
problematic when $\gamma $ is of (vanishing) $2$-variation.

While according to rough path theory, if a vanishing $2$-variation path $%
\gamma $ can be enhanced into a geometric $2$-rough path, then one can give
meaning to differential equation driven by enhanced $\gamma $, and the
solution exists and is unique under certain regularity assumptions on the
vector field (see \cite{T. J. Lyons notes}, \cite{T. J. Lyons and Z. Qian}, 
\cite{P. Friz}).

In this manuscript, we study the properties of the area operator and of $%
\mathcal{G}_{2}\left( \mathcal{V}\right) $, through several questions. (This
manuscript is intended to be some notes about area and geometric $2$-rough
paths, and main results are as listed in the abstract.)

\begin{problem}
\label{Problem existence of Reimann Stieltjes integral}Suppose $\mathcal{V}\ 
$is a Banach spaces, and $\gamma \in C^{0,2-var}\left( \left[ 0,T\right] ,%
\mathcal{V}\right) $. Does the Riemann-Stieltjes integration $%
\int_{0}^{T}\gamma \otimes d\gamma $ exist; if it exists, what is the
regularity of path $t\mapsto \int_{0}^{t}\gamma \otimes d\gamma $.
\end{problem}

In 2009, P. L. Lions \cite{Pierre} sketched a proof of the statement that:
if $\gamma _{1}$ and $\gamma _{2}$ are of vanishing $2$-variation, then $%
\int_{0}^{\cdot }\gamma _{1}\otimes d\gamma _{2}$ can be defined through
Riemann sums and is of vanishing $2$-variation. His statement, however, is
incorrect: first of all, the Riemann-Stieltjes integral may not exist
(Example \ref{Example non-existence of Reimann integral}); secondly, (when
restricted to continuous bounded variation paths equipped with $2$%
-variation) the path$\mapsto $area operator is not bounded (even when area
equipped with uniform norm).

In \cite{P. Friz}(p194), the authors give an example of possible divergence
of Riemann sums (w.r.t. finite partition $D$) as $\left\vert D\right\vert
\rightarrow 0$. Here we modify the example and get non-existence.

For Riemann-Stieltjes integral $\int \gamma \otimes d\gamma $, selecting
different representative points only produces a negligible error when $%
\gamma \in C^{0,2-var}\left( \left[ 0,T\right] ,\mathcal{V}\right) $.
Actually, suppose $\gamma $ is a path defined on $\left[ 0,T\right] $ of
vanishing $2$-variation, and $D=\left\{ t_{j}\right\} $ is a finite
partition satisfying $\left\vert D\right\vert \leq \delta $. Then for any $%
\left\{ \eta _{j},\xi _{j}\right\} _{j}$ satisfying $t_{j}\leq \eta _{j},\xi
_{j}\leq t_{j+1}$, we have (assume $\left\Vert u\otimes v\right\Vert \leq
\left\Vert u\right\Vert \left\Vert v\right\Vert $):%
\begin{multline*}
\left\Vert \sum_{j}\left( \gamma \left( \eta _{j}\right) -\gamma \left( \xi
_{j}\right) \right) \otimes \left( \gamma \left( t_{j+1}\right) -\gamma
\left( t_{j}\right) \right) \right\Vert \\
\leq \left( \sum_{j}\left\Vert \gamma \left( \eta _{j}\right) -\gamma \left(
\xi _{j}\right) \right\Vert ^{2}\right) ^{\frac{1}{2}}\left(
\sum_{j}\left\Vert \gamma \left( t_{j+1}\right) -\gamma \left( t_{j}\right)
\right\Vert ^{2}\right) ^{\frac{1}{2}}.
\end{multline*}%
Since $\left\{ \eta _{j},\xi _{j}\right\} _{j}$ can be treated as points in
another finite partition whose mesh is less or equal $2\delta $, so%
\begin{multline*}
\lim_{\delta \rightarrow 0}\sup_{D,\left\vert D\right\vert \leq \delta
}\left\Vert \sum_{j,t_{j}\in D}\left( \gamma \left( \eta _{j}\right) -\gamma
\left( \xi _{j}\right) \right) \otimes \left( \gamma \left( t_{j+1}\right)
-\gamma \left( t_{j}\right) \right) \right\Vert \\
\leq \lim_{\delta \rightarrow 0}\sup_{D,\left\vert D\right\vert \leq 2\delta
}\sum_{j,t_{j}\in D}\left\Vert \gamma \left( t_{j+1}\right) -\gamma \left(
t_{j}\right) \right\Vert ^{2}=0\text{.}
\end{multline*}%
However, problem may occur when one keeps on inserting partition
points---the area generated by the added points could be infinite. In
Example \ref{Example non-existence of Reimann integral}, we give a path $%
f\in C^{0,2-var}\left( \left[ 0,1\right] ,%
\mathbb{C}
\right) $: 
\begin{gather}
f\left( t\right) =\sum_{n=1}^{\infty }\sum_{k=l_{n}}^{l_{n+1}-1}\frac{1}{k^{%
\frac{1}{2}}2^{k}}\exp \left( 2\pi i\left( -1\right) ^{n}2^{2k}t\right) 
\text{, \ }t\in \left[ 0,1\right] \text{,}  \label{Definition of f} \\
\text{where }c>\pi \text{, \ }c^{n}\leq \sum_{k=l_{n}}^{l_{n+1}}k^{-1}\leq
c^{n}+1\text{, }\forall n\geq 1\text{, }  \notag
\end{gather}%
which satisfies that, for any $a\in \left[ -\infty ,\infty \right] $, there
exists a sequence of finite partitions $\left\{ D_{n}^{a}\right\} _{n}$ of $%
\left[ 0,1\right] $ ($x:=\func{Re}f$, $y:=\func{Im}f$),%
\begin{equation}
\lim_{n\rightarrow \infty }\left\vert D_{n}^{a}\right\vert =0\text{ but }%
\lim_{n\rightarrow \infty }\sum_{k,t_{k}\in D_{n}^{a}}x\left( t_{k}\right)
y\left( t_{k+1}\right) -x\left( t_{k+1}\right) y\left( t_{k}\right) =a.
\label{inner expression divergence1}
\end{equation}%
As a result, since the Riemann sum w.r.t. finite partition $D$ is%
\begin{eqnarray*}
&&\sum_{k,t_{k}\in D}\frac{1}{2}\left( f\left( t_{k}\right) +f\left(
t_{k+1}\right) \right) \otimes \left( f\left( t_{k+1}\right) -f\left(
t_{k}\right) \right) \\
&=&\frac{1}{2}\sum_{k,t_{k}\in D}\left[ f\left( t_{k}\right) ,f\left(
t_{k+1}\right) \right] +\frac{1}{2}f\left( T\right) ^{\otimes 2}-\frac{1}{2}%
f\left( 0\right) ^{\otimes 2},
\end{eqnarray*}%
which does not have a limit as $\left\vert D\right\vert \rightarrow 0$
because of $\left( \ref{inner expression divergence1}\right) $, so the
Riemann-Stieltjes integral $\int_{0}^{1}f\otimes df$ does not exist.

$f$ at $\left( \ref{Definition of f}\right) $ is in $C^{0,2-var}\left( \left[
0,T\right] ,%
\mathbb{C}
\right) $. Similar argument can be applied to $C^{0,2-var}\left( \left[ 0,T%
\right] ,\mathcal{V}\right) $ when $\dim \left( \mathcal{V}\right) \geq 2$.
Select $e_{1}$, $e_{2}\in \mathcal{V}$, s.t. $\left[ e_{1},e_{2}\right] \neq
0$. With $f$ at $\left( \ref{Definition of f}\right) $, define $\widetilde{f}%
=\left( \func{Re}f\right) e_{1}+\left( \func{Im}f\right) e_{2}$. Then
following similar reasoning, the Riemann-Stieltjes integral $\int_{0}^{1}%
\widetilde{f}\otimes d\widetilde{f}$ does not exist, and for any $a\in \left[
-\infty ,\infty \right] $, there exists a sequence of finite partitions $%
\left\{ D_{n}^{a}\right\} _{n}$ of $\left[ 0,1\right] $, s.t.%
\begin{equation}
\lim_{n\rightarrow \infty }\left\vert D_{n}^{a}\right\vert =0\text{ but }%
\lim_{n\rightarrow \infty }\left\Vert \sum_{k,t_{k}\in D_{n}^{a}}\left[ 
\widetilde{f}\left( t_{k}\right) ,\widetilde{f}\left( t_{k+1}\right) \right]
\right\Vert =a.  \label{inner expression divergence2}
\end{equation}

When $\dim \left( \mathcal{V}\right) =1$, the Riemann-Stieltjes integral $%
\int_{0}^{T}\gamma d\gamma $ does exist for any $\gamma \in
C^{0,2-var}\left( \left[ 0,T\right] ,\mathcal{V}\right) $ and equals to $%
2^{-1}\left( \gamma ^{2}\left( T\right) -\gamma ^{2}\left( 0\right) \right) $%
, because the vector field is commutative in one-dimensional case, so the
Lie bracket vanishes. Thus, any one-dimensional vanishing $2$-variation path
is in $\mathcal{G}_{2}\left( \mathcal{V}\right) $, and 
\begin{equation}
\mathcal{G}_{2}\left( \mathcal{V}\right) =C^{0,2-var}\left( \left[ 0,T\right]
,\mathcal{V}\right) \text{ when }\dim \left( \mathcal{V}\right) =1.
\label{G2 equals to C0-2var when dim(V)=1}
\end{equation}

\begin{problem}
When equipping $C^{1-var}\left( \left[ 0,T\right] ,\mathcal{V}\right) $ with 
$2$-variation norm, is the area operator continuous, or bounded?
\end{problem}

When $\dim \left( \mathcal{V}\right) =1$, area vanishes, so the area
operator is trivial. In that case it is continuous and bounded. When $\dim
\left( \mathcal{V}\right) \geq 2$, as a consequence of possible
non-existence of the Riemann-Stieltjes integral $\left( \ref{inner
expression divergence2}\right) $, the area operator is not continuous nor
bounded.

Actually, suppose $\dim \left( \mathcal{V}\right) \geq 2$, $\gamma \in
C^{0,2-var}\left( \left[ 0,T\right] ,\mathcal{V}\right) $, and $\gamma ^{D}$
the piecewise linear paths coincides with $\gamma $ on points in $D$ (as
defined at $\left( \ref{Definition gammaD}\right) $). Then after direct
computation, the Riemann sum of $\int \gamma \otimes d\gamma $ w.r.t. $D$
equals to $A\left( \gamma ^{D}\right) \left( 0,T\right) $ plus a constant:%
\begin{eqnarray}
&&\sum_{k,t_{k}\in D}\frac{1}{2}\left( \gamma \left( t_{k}\right) +\gamma
\left( t_{k+1}\right) \right) \otimes \left( \gamma \left( t_{k+1}\right)
-\gamma \left( t_{k}\right) \right)
\label{Reimann sums and area of piecewise linear path} \\
&=&\frac{1}{2}\sum_{k,t_{k}\in D}\left[ \gamma \left( t_{k}\right) ,\gamma
\left( t_{k+1}\right) \right] +\frac{1}{2}\gamma ^{\otimes 2}\left( T\right)
-\frac{1}{2}\gamma ^{\otimes 2}\left( 0\right)  \notag \\
&=&A\left( \gamma ^{D}\right) \left( 0,T\right) +\frac{1}{2}\left( \gamma
\left( T\right) +\gamma \left( 0\right) \right) \otimes \left( \gamma \left(
T\right) -\gamma \left( 0\right) \right) .  \notag
\end{eqnarray}%
Thus, based on $\left( \ref{inner expression divergence2}\right) $, there
exists a path $f:\left[ 0,1\right] \rightarrow \mathcal{V}$ of vanishing $2$%
-variation, such that for any $a\in \left[ -\infty ,\infty \right] $, there
exists a sequence of finite partitions $\left\{ D_{n}^{a}\right\} $ of $%
\left[ 0,1\right] $, satisfying $\lim_{n\rightarrow \infty }\left\vert
D_{n}^{a}\right\vert =0$ but $\lim_{n\rightarrow \infty }\left\Vert A\left(
f^{D_{n}^{a}}\right) \left( 0,1\right) \right\Vert =a$. While $f^{D_{n}^{a}}$
converges to $f$ in~$2$-variation when $n$ tends to infinity (based on $%
\left( \ref{equivalent relation for vanishing 2-var path}\right) $). Thus,
the area operator is not continuous and not bounded, at least when area is
equipped with uniform norm. Thus, there is No universal constant $C$, s.t. $%
\left\Vert A\left( \gamma \right) \right\Vert _{\infty -var}\leq C\left\Vert
\gamma \right\Vert _{2-var}^{2}$ for all $\gamma \in C^{1-var}\left( \left[
0,T\right] ,\mathcal{V}\right) $. Compare with Young integral: for any $p\in
\lbrack 1,2)$, there exists $C_{p}$, s.t. for any $\gamma \in
C^{p-var}\left( \left[ 0,T\right] ,\mathcal{V}\right) $, $\left\Vert A\left(
\gamma \right) \right\Vert _{\frac{p}{2}-var}\leq C_{p}\left\Vert \gamma
\right\Vert _{p-var}^{2}$ (i.e. $\left( \ref{Reimann area inequality}\right) 
$).

Moreover, by modifying our example, we get a sequence of continuous bounded
variation paths (Example \ref{Example divergence in p-variation,p>1} at p%
\pageref{Example divergence in p-variation,p>1}) converging to zero in $2$%
-variation, but their area diverge at any non-trivial point: $\left(
s,t\right) \in \bigtriangleup _{\left[ 0,T\right] }$, $s<t$. Therefore, when
equipping bounded variation paths with $2$-variation, the area operator is
not continuous nor bounded, even in the sense of at some single point. (The
paths in Example \ref{Example divergence in p-variation,p>1} are in $%
C^{0,2-var}\left( \left[ 0,1\right] ,%
\mathbb{C}
\right) $, but can be generalized to $C^{0,2-var}\left( \left[ 0,1\right] ,%
\mathcal{V}\right) $ whenever $\dim \left( \mathcal{V}\right) \geq 2$.)

\begin{problem}
When $C^{1-var}\left( \left[ 0,T\right] ,\mathcal{V}\right) $ is equipped
with $2$-variation norm, is the path$\mapsto $area operator\ closable in $p$%
-variation? In other words, if $\left\{ \gamma _{n}\right\} _{n}$\ and $%
\left\{ \gamma _{m}\right\} _{m}$\ are two sequence of paths in $%
C^{1-var}\left( \left[ 0,T\right] ,\mathcal{V}\right) $ converging in $2$%
-variation to the same limit, and $\left\{ A\left( \gamma _{n}\right)
\right\} _{n}$\ and $\left\{ A\left( \gamma _{m}\right) \right\} _{m}$\
converge in $p$-variation respectively. Then is that true that $\left\{
A\left( \gamma _{n}\right) \right\} _{n}$\ and $\left\{ A\left( \gamma
_{m}\right) \right\} _{m}$\ converge to the same limit?
\end{problem}

When $p>1$, not true. When $p=1$, is true. (We assume $\dim \left( \mathcal{V%
}\right) \geq 2$, because area vanishes for one-dimensional paths.)

For $p>1$, an illustrative example is $r_{n}\left( t\right) =\left( \frac{%
\cos nt}{\sqrt{n}},\frac{\sin nt}{\sqrt{n}}\right) $, $t\in \left[ 0,2\pi %
\right] $, $n\geq 1$. $r_{n}$ converges to $0$ in $q$-variation for any $q>2$%
, but their area converge to $t-s$ in $p$-variation for any $p>1$:%
\begin{equation*}
\int_{s}^{t}\left( \frac{\cos nu}{\sqrt{n}}-\frac{\cos ns}{\sqrt{n}}\right) d%
\frac{\sin nu}{\sqrt{n}}-\left( \frac{\sin nu}{\sqrt{n}}-\frac{\sin ns}{%
\sqrt{n}}\right) d\frac{\cos nu}{\sqrt{n}}=t-s-\frac{\sin n\left( t-s\right) 
}{n}.
\end{equation*}%
and 
\begin{equation*}
\left\Vert \frac{1}{\sqrt{n}}\exp \left( int\right) \right\Vert
_{q-var}\lesssim \frac{1}{n^{\frac{1}{2}-\frac{1}{q}}}\text{, }\left\Vert 
\frac{\sin n\left( t-s\right) }{n}\right\Vert _{p-var}\lesssim \frac{1}{n^{1-%
\frac{1}{p}}}.
\end{equation*}%
Thus, $\left( 0,0\right) $ and $\left( 0,t-s\right) $ are two geometric $q$%
-rough paths with the same first level path for any $q\in \left( 2,3\right) $%
. (Geometric $q$-rough paths $q\in \left( 2,3\right) $ are elements in the
closure of $\left\{ \left( \gamma ,A\left( \gamma \right) \right) |\gamma
\in C^{1-var}\left( \left[ 0,T\right] ,\mathcal{V}\right) \right\} $ under
the metric%
\begin{equation*}
\left. d\left( \left( \gamma _{1},A\left( \gamma _{1}\right) \right) ,\left(
\gamma _{2},A\left( \gamma _{2}\right) \right) \right) :=\left( \left\Vert
\gamma _{1}-\gamma _{2}\right\Vert _{q-var}^{q}+\left\Vert A\left( \gamma
_{1}\right) -A\left( \gamma _{2}\right) \right\Vert _{\frac{q}{2}-var}^{%
\frac{q}{2}}\right) ^{\frac{1}{q}}.\right)
\end{equation*}

However, $r_{n}\ $are uniformly bounded in $2$-variation, but do not
converge in $2$-variation ($||n^{-\frac{1}{2}}\cos \left( nt\right) -\left(
2n\right) ^{-\frac{1}{2}}\cos \left( 2nt\right) ||_{2-var}\geq 2$, $\forall n
$). To construct our example, we add in a decay factor, sum finitely of them
together to compensate the decaying effect on $t-s$, and end up with
functions $\left\{ g_{n}\right\} _{n}\subset C^{0,2-var}\left( \left[ 0,1%
\right] ,%
\mathbb{C}
\right) $ (Example \ref{Example, non-closability} at p\pageref{Example,
non-closability})%
\begin{gather}
g_{n}\left( t\right) =\left( \pi \sum_{k=l_{n}}^{l_{n+1}-1}\frac{1}{k}%
\right) ^{-\frac{1}{2}}\sum_{k=l_{n}}^{l_{n+1}-1}\frac{1}{k^{\frac{1}{2}%
}2^{k}}\exp \left( 2\pi i2^{2k}t\right) \text{, }t\in \left[ 0,1\right] ,
\label{Definition of g_n} \\
\text{where }\sum_{k=l_{n}}^{l_{n+1}-1}k^{-1}\geq 1\text{, }\forall n\geq 1%
\text{.}  \notag
\end{gather}%
We prove that $g_{n}$ converge in $2$-variation to zero as $n$ tends to
infinity, but their area converge to $t-s$ in $p$-variation, for any $p>1$.

For Banach space $\mathcal{V}$, $\dim \left( \mathcal{V}\right) \geq 2$,
select $e_{1},e_{2}\in \mathcal{V}$, s.t. $\left[ e_{1},e_{2}\right] \neq 0$%
. With $g_{n}$ defined at $\left( \ref{Definition of g_n}\right) $, define $%
\widetilde{g_{n}}:=\left( \func{Re}g_{n}\right) e_{1}+\left( \func{Im}%
g_{n}\right) e_{2}$. Then $\left\{ \widetilde{g_{n}}\right\} _{n}\subset
C^{1-var}\left( \left[ 0,1\right] ,\mathcal{V}\right) $, $\lim_{n\rightarrow
\infty }\left\Vert \widetilde{g_{n}}\right\Vert _{2-var}=0\,$\ and 
\begin{equation*}
\lim_{n\rightarrow \infty }\left\Vert A\left( \widetilde{g_{n}}\right)
\left( s,t\right) -\left( t-s\right) \left[ e_{1},e_{2}\right] \right\Vert
_{p-var}=0\text{ for any }p>1\text{.}
\end{equation*}

When $p=1$, if $\left( \gamma ,\alpha _{1}\right) $ and $\left( \gamma
,\alpha _{2}\right) $ are two geometric $2$-rough paths, then $\alpha
_{1}-\alpha _{2}:=\varphi $ is additive thus a path. Moreover, based on $%
\left( \ref{Area is of vanishing 1-var}\right) $, $\alpha _{1}$ and $\alpha
_{2}$ are in $C^{0,1-var}\left( \bigtriangleup _{\left[ 0,T\right] },\left[ 
\mathcal{V},\mathcal{V}\right] \right) $ of vanishing $1$-variation, then $%
\varphi $ is of vanishing $1$-variation. While a path of vanishing $1$%
-variation is constant, so $\alpha _{1}=\alpha _{2}$.

For the same reason we have: the projection of a geometric $n$-rough path to
its first $n-1$ level elements is injective for any $n\in 
\mathbb{N}
$, $n\geq 2$. While in Remark 9.13 (case ii b2) in \cite{P. Friz}, the
authors commented that the projection is not a injection without providing a
proof.

\begin{problem}
\label{Problem non-existence of enhancement}Is that true that every path in $%
C^{0,2-var}\left( \left[ 0,T\right] ,\mathcal{V}\right) $ admits an
enhancement into a (weak) geometric $2$-rough path? (i.e. is the inclusion $%
\mathcal{G}_{2}\left( \mathcal{V}\right) \subseteq C^{0,2-var}\left( \left[
0,T\right] ,\mathcal{V}\right) $ strict?)
\end{problem}

When $\dim \left( \mathcal{V}\right) =1$, $\mathcal{G}_{2}\left( \mathcal{V}%
\right) =C^{0,2-var}\left( \left[ 0,T\right] ,\mathcal{V}\right) $ (see $%
\left( \ref{G2 equals to C0-2var when dim(V)=1}\right) $).

When $\dim \left( \mathcal{V}\right) \geq 2$, $\mathcal{G}_{2}\left( 
\mathcal{V}\right) \subsetneqq C^{0,2-var}\left( \left[ 0,T\right] ,\mathcal{%
V}\right) $, and an example is given in Thm 9.12 \cite{P. Friz}. Actually,
following the same reasoning as in Thm 9.12 \cite{P. Friz}, we use $f$
defined at $\left( \ref{Definition of f}\right) $ to prove that $\mathcal{G}%
_{2}\left( \mathcal{V}\right) \subsetneqq C^{0,2-var}\left( \left[ 0,T\right]
,\mathcal{V}\right) $ when $\dim \left( \mathcal{V}\right) \geq 2$. Select $%
e_{1}$, $e_{2}\in \mathcal{V}$, s.t. $\left[ e_{1},e_{2}\right] \neq 0$.
With $f\ $at $\left( \ref{Definition of f}\right) $, denote $\widetilde{f}%
:=\left( \func{Re}f\right) e_{1}+\left( \func{Im}f\right) e_{2}$, so $%
\widetilde{f}\in C^{0,2-var}\left( \left[ 0,T\right] ,\mathcal{V}\right) $.
Assume that $(\widetilde{f},\alpha )$ is a weak geometric $2$-rough path.
Then using multiplicativity of $(\widetilde{f},\alpha )$ (i.e. $\left( \ref%
{property of multiplicativity}\right) $), for any finite partitions $D$, we
have 
\begin{align*}
& \left\Vert \sum_{j,t_{j}\in D}\left[ \widetilde{f}\left( t_{j}\right) ,%
\widetilde{f}\left( t_{j+1}\right) \right] -\left[ \widetilde{f}\left(
0\right) ,\widetilde{f}\left( 1\right) \right] \right\Vert =2\left\Vert
\alpha \left( 0,1\right) -\sum_{j,t_{j}\in D}\alpha \left(
t_{j},t_{j+1}\right) \right\Vert \\
& \leq 4\left\Vert \alpha \right\Vert _{1-var}<\infty .
\end{align*}%
Then contradiction is established, if $\sum_{j,t_{j}\in D}\left[ \widetilde{f%
}\left( t_{j}\right) ,\widetilde{f}\left( t_{j+1}\right) \right] $ are not
uniformly bounded for all finite partitions, which is true because of $%
\left( \ref{inner expression divergence2}\right) $.

Then a natural question arises:

\begin{problem}
What is the condition for vanishing $2$-variation paths to be enhancible
(i.e. in $\mathcal{G}_{2}\left( \mathcal{V}\right) $)?
\end{problem}

We prove that:

\begin{theorem}
\label{Theorem condition to be enhancible}Suppose $\gamma \in
C^{0,2-var}\left( \left[ 0,T\right] ,\mathcal{V}\right) $. Then $\gamma \in 
\mathcal{G}_{2}\left( \mathcal{V}\right) $ if and only if $A\left( \gamma
^{D}\right) $ converges in $1$-variation as $\left\vert D\right\vert
\rightarrow 0$.
\end{theorem}

The proof is given in page \pageref{Proof of Theorem condition to be
enhancible}.

In Thm 8.22 \cite{P. Friz}, the authors proved that, when $\mathcal{V=%
\mathbb{R}
}^{d}$, if $\left( \gamma ,\alpha \right) $ is a geometric $2$-rough path,
then there exists a sequence of continuous bounded variation paths $\left\{
\gamma _{n}\right\} _{n}$, s.t. $\left( \gamma _{n},A\left( \gamma
_{n}\right) \right) $ converge to $\left( \gamma ,\alpha \right) $ in $2$%
-rough norm $\left\Vert \cdot \right\Vert _{G^{\left( 2\right) }}$. However,
their construction of $\left\{ \gamma _{n}\right\} $ depends on $\alpha $
(i.e. Chow-Rashevskii connectivity theorem), while not $\gamma ^{D}$ in
general.

For any $0\leq s\leq t\leq T$ and any finite partition $D$ of $\left[ s,t%
\right] $, the Riemann sums of $2^{-1}\int_{s}^{t}\left[ \gamma \left(
u\right) -\gamma \left( s\right) ,d\gamma \left( u\right) \right] $ w.r.t. $%
D\subset \left[ s,t\right] $ is%
\begin{eqnarray*}
&&2^{-1}\sum_{k,t_{k}\in D}\frac{1}{2}\left[ \gamma \left( t_{k}\right)
+\gamma \left( t_{k+1}\right) ,\gamma \left( t_{k+1}\right) -\gamma \left(
t_{k}\right) \right] -2^{-1}\left[ \gamma \left( s\right) ,\gamma \left(
t\right) \right] \\
&=&2^{-1}\sum_{k,t_{k}\in D}\left[ \gamma \left( t_{k}\right) ,\gamma \left(
t_{k+1}\right) \right] -2^{-1}\left[ \gamma \left( s\right) ,\gamma \left(
t\right) \right] .
\end{eqnarray*}%
On the other hand, direct computation gives us 
\begin{equation*}
A\left( \gamma ^{D}\right) \left( s,t\right) =2^{-1}\sum_{k,t_{k}\in
D\subset \left[ s,t\right] }\left[ \gamma \left( t_{k}\right) ,\gamma \left(
t_{k+1}\right) \right] -2^{-1}\left[ \gamma \left( s\right) ,\gamma \left(
t\right) \right] .
\end{equation*}%
Thus, the Riemann-Stieltjes integral $2^{-1}\int_{s}^{t}\left[ \gamma \left(
u\right) -\gamma \left( s\right) ,d\gamma \left( u\right) \right] $ is the
pointwise limit of $A\left( \gamma ^{D}\right) $ as $\left\vert D\right\vert
\rightarrow 0$. Hence, if $\gamma $ is in $\mathcal{G}_{2}\left( \mathcal{V}%
\right) $, then $A\left( \gamma ^{D}\right) $ converge in $1$-variation
(Theorem \ref{Theorem condition to be enhancible}), so converge pointwisely,
to $2^{-1}\int_{s}^{t}\left[ \gamma \left( u\right) -\gamma \left( s\right)
,d\gamma \left( u\right) \right] $.

Therefore, the Riemann-Stieltjes integral $2^{-1}\int_{s}^{t}\left[ \gamma
\left( u\right) -\gamma \left( s\right) ,d\gamma \left( u\right) \right] $
is the only possible candidate to enhance $\gamma $: If the integral does
not exist, or $(\gamma ,$ $2^{-1}\int_{s}^{t}\left[ \gamma \left( u\right)
-\gamma \left( s\right) ,d\gamma \left( u\right) \right] )$ is not a
geometric $2$-rough path, then $\gamma $ can not be enhanced into a
geometric $2$-rough path.

While when $p>2$, the convergence of $A\left( f^{D}\right) $ as $\left\vert
D\right\vert \rightarrow 0$ is not necessary to enhance a path in $%
C^{0,p-var}\left( \left[ 0,T\right] ,\mathcal{V}\right) $. Our path $f$ at $%
\left( \ref{Definition of f}\right) $ is in $C^{0,2-var}\left( \left[ 0,T%
\right] ,\mathcal{V}\right) \subset C^{2-var}\left( \left[ 0,T\right] ,%
\mathcal{V}\right) \subset \cap _{p>2}C^{0,p-var}\left( \left[ 0,T\right] ,%
\mathcal{V}\right) $. Based on \cite{T.J.Lyons and Victoir}, finite $p$%
-variation paths can be enhanced into a geometric $q$-rough path for any $%
q>p $, so $f\ $can be enhanced into a geometric $p$-rough path for any $p>2$%
. While $\sup_{D\subset \left[ 0,1\right] }A\left( f^{D}\right) \left(
0,1\right) $ is not bounded, so $A\left( f^{D}\right) $ do not converge in $%
p $-variation, for any $p\in \left[ 1,\infty \right] $.

\medskip

Similar to Theorem \ref{Theorem condition to be enhancible}, we proved that:

\begin{theorem}
\label{Theorem condition to be enhancible into a 2-rough path}Suppose $%
\gamma \in C^{2-var}\left( \left[ 0,T\right] ,\mathcal{V}\right) $. Then $%
\gamma $ can be enhanced into a weak geometric $2$-rough path if and only if%
\begin{equation*}
\sup_{D}\left\Vert A\left( \gamma ^{D}\right) \right\Vert _{1-var,\left[ 0,T%
\right] }<\infty \text{ and }\left\{ A\left( \gamma ^{D}\right) \right\} _{D}%
\text{ are equicontinuous.}
\end{equation*}
\end{theorem}

The proof is given at page \pageref{Proof of Theorem condition to be
enhancible into a 2-rough path}.

\begin{problem}
\label{Problem not compose a space}Is $\mathcal{G}_{2}\left( \mathcal{V}%
\right) $ a linear space?
\end{problem}

$\mathcal{G}_{2}\left( \mathcal{V}\right) $ is linear when $\dim \left( 
\mathcal{V}\right) =1$; is not linear when $\dim \left( \mathcal{V}\right)
\geq 2$.

Based on $\left( \ref{G2 equals to C0-2var when dim(V)=1}\right) $ we got at
the end of Problem \ref{Problem existence of Reimann Stieltjes integral},
when $\dim \left( \mathcal{V}\right) =1$, $\mathcal{G}_{2}\left( \mathcal{V}%
\right) =C^{0,2-var}\left( \left[ 0,T\right] ,\mathcal{V}\right) $ thus a
space. When $\dim \left( \mathcal{V}\right) \geq 2$, based on the reasoning
in Problem \ref{Problem non-existence of enhancement}, $\mathcal{G}%
_{2}\left( \mathcal{V}\right) $ is not a space.

\medskip

The non-linearity of $\mathcal{G}_{2}\left( \mathcal{V}\right) $ is
inherited from the non-linearity of the area operator.

\begin{proposition}
When $\dim \left( \mathcal{V}\right) \geq 2$, both $\mathcal{G}_{2}\left( 
\mathcal{V}\right) $ and $C^{0,2-var}\left( \left[ 0,T\right] ,\mathcal{V}%
\right) /\mathcal{G}_{2}\left( \mathcal{V}\right) $ are dense in $%
C^{0,2-var}\left( \left[ 0,T\right] ,\mathcal{V}\right) $ under $2$%
-variation norm.
\end{proposition}

\begin{proof}
$\mathcal{G}_{2}\left( \mathcal{V}\right) $ is dense in $C^{0,2-var}\left( %
\left[ 0,T\right] ,\mathcal{V}\right) $, because (based on $\left( \ref%
{equivalent relation for vanishing 2-var path}\right) $) 
\begin{equation*}
C^{1-var}\left( \left[ 0,T\right] ,\mathcal{V}\right) \subseteq \mathcal{G}%
_{2}\left( \mathcal{V}\right) \subseteq C^{0,2-var}\left( \left[ 0,T\right] ,%
\mathcal{V}\right) =:\overline{C^{1-var}\left( \left[ 0,T\right] ,\mathcal{V}%
\right) }^{2-var}.
\end{equation*}

On the other hand, when $\dim \left( \mathcal{V}\right) \geq 2$, suppose $%
\gamma \in C^{0,2-var}\left( \left[ 0,T\right] ,\mathcal{V}\right) $. We
want to find a non-enhancible path $\widetilde{\gamma }$ in the $2$%
-variation neighborhood of $\gamma $. Based on the definition of $f$ at $%
\left( \ref{Definition of f}\right) $, define%
\begin{equation*}
f_{N}\left( t\right) :=\sum_{n=N}^{\infty }\sum_{k=l_{n}}^{l_{n+1}-1}\frac{1%
}{k^{\frac{1}{2}}2^{k}}\exp \left( 2\pi i\left( -1\right) ^{n}2^{2k}t\right)
,t\in \left[ 0,1\right] \text{.}
\end{equation*}%
Then (based on Lemma \ref{Lemma, decay of 2-var} below, which is used in the
proof of the non-enhancibility of $f$), $\sup_{N}l_{N}^{\frac{1}{2}%
}\left\Vert f_{N}\right\Vert _{2-var}:=C<\infty $. On the other hand, for
each fixed $N$, $\sup_{D}\left\vert A(\left( f_{N}\right) ^{D})\left(
0,1\right) \right\vert =\infty $ (because $\sup_{D}\left\vert A\left(
f^{D}\right) \left( 0,1\right) \right\vert =\infty $ and $f-f_{N}$ is
smooth). Select $e_{1}$, $e_{2}\in \mathcal{V}$, s.t. $\left[ e_{1},e_{2}%
\right] \neq 0$. For any $\epsilon >0$, choose integer $K$, s.t. $2^{-2K}<T$%
, $\left\Vert \gamma \right\Vert _{2-var,[0,2^{-2K}]}<\epsilon $ and $\left(
\left\Vert e_{1}\right\Vert +\left\Vert e_{2}\right\Vert \right) (Cl_{K+1}^{-%
\frac{1}{2}})<\epsilon $. Define $g\in C^{0,2-var}\left( \left[ 0,1\right] ,%
\mathcal{V}\right) $ by%
\begin{equation*}
g\left( t\right) =\left( \func{Re}\left( f_{K+1}\left( t\right)
-f_{K+1}\left( 1\right) \right) \right) e_{1}+\left( \func{Im}\left(
f_{K+1}\left( t\right) -f_{K+1}\left( 1\right) \right) \right) e_{2},t\in %
\left[ 0,1\right] .
\end{equation*}%
Then $g\left( 1\right) =0$ and 
\begin{eqnarray*}
\left\Vert g\right\Vert _{2-var,\left[ 0,1\right] } &\leq &\left( \left\Vert
e_{1}\right\Vert +\left\Vert e_{2}\right\Vert \right) \left\Vert
f_{K+1}\right\Vert _{2-var,\left[ 0,1\right] } \\
&\leq &\left( \left\Vert e_{1}\right\Vert +\left\Vert e_{2}\right\Vert
\right) (Cl_{K+1}^{-\frac{1}{2}})<\epsilon , \\
\sup_{D}\left\Vert A(g^{D})\left( 0,1\right) \right\Vert
&=&\sup_{D}\left\vert A(\left( f_{K+1}\right) ^{D})\left( 0,1\right)
\right\vert \left\Vert \left[ e_{1},e_{2}\right] \right\Vert =\infty .
\end{eqnarray*}%
Define%
\begin{equation*}
\widetilde{\gamma }\left( t\right) =\left\{ 
\begin{array}{cc}
g\left( 2^{2\left( K+1\right) }t\right) +\gamma (\frac{1}{2^{2\left(
K+1\right) }}), & t\in \lbrack 0,\frac{1}{2^{2\left( K+1\right) }}] \\ 
\text{linear,} & t\in \lbrack \frac{1}{2^{2\left( K+1\right) }},\frac{1}{%
2^{2K}}] \\ 
\gamma \left( t\right) , & t\in \lbrack \frac{1}{2^{2K}},T]%
\end{array}%
\right. .
\end{equation*}%
Then $\widetilde{\gamma }$ is continuous and%
\begin{equation*}
\left\Vert \gamma -\widetilde{\gamma }\right\Vert _{2-var}\leq 2\left\Vert
\gamma \right\Vert _{2-var,[0,2^{-2K}]}+\left\Vert g\right\Vert _{2-var,%
\left[ 0,1\right] }<3\epsilon .
\end{equation*}%
On the other hand, 
\begin{gather*}
\sup_{D\subset \lbrack 0,1]}\left\Vert A(\widetilde{\gamma }^{D})\right\Vert
_{1-var}\geq \sup_{D\subset \lbrack 0,1]}\left\Vert A(\widetilde{\gamma }%
^{D})\right\Vert _{\infty -var} \\
\geq \sup_{D\subset \lbrack 0,\frac{1}{2^{2\left( K+1\right) }}]}\left\Vert
A(\widetilde{\gamma }^{D})(0,\frac{1}{2^{2\left( K+1\right) }})\right\Vert
=\sup_{D\subset \left[ 0,1\right] }\left\Vert A(g^{D})\left( 0,1\right)
\right\Vert =\infty .
\end{gather*}%
Thus $A(\widetilde{\gamma }^{D})$ do not converge in $1$-variation as $%
\left\vert D\right\vert \rightarrow 0$, and based on Theorem \ref{Theorem
condition to be enhancible}, $\widetilde{\gamma }$ is not enhancible.
\end{proof}

When $\gamma $ is a path of finite $p$-variation, $p\in \lbrack 1,2)$, based
on Young integral and Theorem \ref{Theorem condition to be enhancible}, the
enhancement of $\gamma $ to geometric $2$-rough path exists uniquely in the
form of Riemann-Stieltjes integral. Thus $\cup _{1\leq p<2}C^{p-var}\left( %
\left[ 0,T\right] ,\mathcal{V}\right) \subseteq \mathcal{G}_{2}\left( 
\mathcal{V}\right) $.

\begin{problem}
Is the inclusion $\cup _{1\leq p<2}C^{p-var}\left( \left[ 0,T\right] ,%
\mathcal{V}\right) \subseteq \mathcal{G}_{2}\left( \mathcal{V}\right) $
strict?
\end{problem}

Yes, it is. When $\dim \left( \mathcal{V}\right) =1$, $\mathcal{G}_{2}\left( 
\mathcal{V}\right) =C^{0,2-var}\left( \left[ 0,1\right] ,\mathcal{V}\right) $
(based on $\ \left( \ref{G2 equals to C0-2var when dim(V)=1}\right) $).
Select $e\in \mathcal{V}$, $e\neq 0$, and define $h\left( t\right) =\left(
t^{\frac{1}{2}}\cos ^{2}\left( \frac{\pi }{t}\right) /\ln t\right) e$, $t\in %
\left[ 0,1\right] $. Then 
\begin{equation*}
h\in C^{0,2-var}\left( \left[ 0,T\right] ,\mathcal{V}\right) \backslash \cup
_{1\leq p<2}C^{p-var}\left( \left[ 0,T\right] ,\mathcal{V}\right) \text{ \
(Exer5.35\cite{P. Friz}).}
\end{equation*}
When $\dim \left( \mathcal{V}\right) \geq 2$, the inclusion is strict
because $\cup _{1\leq p<2}C^{p-var}\left( \left[ 0,T\right] ,\mathcal{V}%
\right) $ is a space, but $\mathcal{G}_{2}\left( \mathcal{V}\right) $ is not
(Problem \ref{Problem not compose a space}).

\medskip

Although $\mathcal{G}_{2}$ is not a space, it can be shifted in any of the
"Young" direction.

\begin{proposition}
$\mathcal{G}_{2}\left( \mathcal{V}\right) +\cup _{1\leq p<2}C^{p-var}\left( %
\left[ 0,T\right] ,\mathcal{V}\right) =\mathcal{G}_{2}\left( \mathcal{V}%
\right) $.
\end{proposition}

Suppose $\gamma _{1}\in \mathcal{G}_{2}\left( \mathcal{V}\right) $, then $%
\gamma _{1}$ is of finite $2$-variation. For any $\gamma _{2}$ of finite $p$%
-variation, $p\in \lbrack 1,2)$, according to Young integral (i.e.$\left( %
\ref{Reimann area inequality}\right) $), $A(\gamma _{1}^{D},\gamma _{2}^{D})$
converge in $\left( 2^{-1}+p^{-1}\right) ^{-1}$-variation as $\left\vert
D\right\vert \rightarrow 0$ ($p<2$, so converge in $1$-variation).
Similarly, $A(\gamma _{2}^{D},\gamma _{1}^{D})$ and $A(\gamma
_{2}^{D},\gamma _{2}^{D})$ converge in $1$-variation as $\left\vert
D\right\vert \rightarrow 0$. On the other hand, $\gamma _{1}\in \mathcal{G}%
_{2}\left( \mathcal{V}\right) $, so apply Theorem \ref{Theorem condition to
be enhancible}, $A(\gamma _{1}^{D}):=A\left( \gamma _{1}^{D},\gamma
_{1}^{D}\right) $ converge in $1$-variation. Therefore $A((\gamma
_{1}+\gamma _{2})^{D})=\sum_{i,j=1,2}A\left( \gamma _{i}^{D},\gamma
_{j}^{D}\right) $ converge in $1$-variation as $\left\vert D\right\vert
\rightarrow 0$ and $\gamma _{1}+\gamma _{2}$ is enhancible (Theorem \ref%
{Theorem condition to be enhancible}).

\medskip

In the way of exploring paths in $\mathcal{G}_{2}\left( \mathcal{V}\right) $%
, we get an extension to Young \cite{L. C. Young}.

\begin{theorem}
\label{Theorem generalized Young integral}Let $\mathcal{V}_{i}$, $i=1,2$, be
two Banach spaces and $\gamma _{i}:\left[ 0,1\right] \rightarrow \mathcal{V}%
_{i}$ be two continuous paths. If there exist $p>1$, $q>1$, $p^{-1}+q^{-1}=1$%
, and two non-decreasing functions $m_{i}:\left[ 0,1\right] \rightarrow 
\overline{%
\mathbb{R}
^{+}}$, $i=1,2$, satisfying%
\begin{equation*}
\lim_{t\rightarrow 0}m_{i}\left( t\right) =0\text{, }m_{i}\left( 1\right)
\leq 1\text{, and }\int_{0}^{1}\frac{m_{1}\left( t\right) m_{2}\left(
t\right) }{t}dt<\infty ,
\end{equation*}%
such that 
\begin{equation}
\hspace{-0.02in}\hspace{-0.02in}\sup_{0\leq s<t\leq 1}\frac{\left\Vert
\gamma _{1}\left( t\right) -\gamma _{1}\left( s\right) \right\Vert }{%
\left\vert t-s\right\vert ^{\frac{1}{p}}m_{1}\left( t-s\right) }\hspace{%
-0.02in}:=\hspace{-0.02in}C_{1}<\hspace{-0.02in}\infty \text{,}\sup_{0\leq
s<t\leq 1}\frac{\left\Vert \gamma _{2}\left( t\right) -\gamma _{2}\left(
s\right) \right\Vert }{\left\vert t-s\right\vert ^{\frac{1}{q}}m_{2}\left(
t-s\right) }\hspace{-0.02in}:=\hspace{-0.02in}C_{2}\hspace{-0.02in}<\hspace{%
-0.02in}\infty .  \label{Condition on gamma}
\end{equation}%
Then the Riemann-Stieltjes integral $\int_{0}^{t}\gamma _{1}\left( t\right)
\otimes d\gamma _{2}\left( t\right) $, $t\in \left[ 0,1\right] $ exists, and 
\begin{equation*}
\left\Vert \int_{0}^{\cdot }\gamma _{1}\left( t\right) \otimes d\gamma
_{2}\left( t\right) \right\Vert _{q-var}\leq 8C_{1}C_{2}\left( 2+\int_{0}^{1}%
\frac{m_{1}\left( t\right) m_{2}\left( t\right) }{t}dt\right) .
\end{equation*}
\end{theorem}

Theorem \ref{Theorem generalized Young integral} is proved in page \pageref%
{Proof of theorem generalized Young integral}.

\begin{remark}
When $m_{1}\left( x\right) =x^{a}$, $m_{2}\left( x\right) =x^{b}$, $a>0$, $%
b>0$, we get Young integral \cite{L. C. Young}.
\end{remark}

\begin{remark}
In the proof of Theorem \ref{Theorem generalized Young integral}, we get an
estimation of the iterated integral of $\gamma _{1}$ and $\gamma _{2}$
(Definition \ref{Definition area of two paths}):%
\begin{equation*}
\left\Vert I\left( \gamma _{1},\gamma _{2}\right) \right\Vert _{1-var}\leq
C_{1}C_{2}\left( 15+8\int_{0}^{1}\frac{m_{1}\left( t\right) m_{2}\left(
t\right) }{t}dt\right) .
\end{equation*}
\end{remark}

On the other hand, $\int_{0}^{1}\frac{m_{1}\left( t\right) m_{2}\left(
t\right) }{t}dt<\infty $ is necessary in the sense of the following example.

\begin{example}
\label{Example integration finite is necessary}Suppose $m_{i}:\left[ 0,1%
\right] \rightarrow \overline{%
\mathbb{R}
^{+}}$ are two non-decreasing functions, satisfying $\lim_{t\rightarrow
0}m_{i}\left( t\right) =0$, $\left\vert m_{i}\right\vert \leq 1$, $i=1,2$,
and $\int_{0}^{1}\frac{m_{1}\left( t\right) m_{2}\left( t\right) }{t}%
dt=\infty $. Then for any $p>1$, $q>1$, $p^{-1}+q^{-1}=1$, there exist two
continuous real-valued paths $\gamma _{i}:\left[ 0,1\right] \rightarrow 
\mathbb{R}
$, $i=1,2$, s.t. 
\begin{equation*}
\sup_{0\leq s<t\leq 1}\frac{\left\vert \gamma _{1}\left( t\right) -\gamma
_{1}\left( s\right) \right\vert }{\left\vert t-s\right\vert ^{\frac{1}{p}%
}m_{1}\left( t-s\right) }<\infty \text{, }\sup_{0\leq s<t\leq 1}\frac{%
\left\vert \gamma _{2}\left( t\right) -\gamma _{2}\left( s\right)
\right\vert }{\left\vert t-s\right\vert ^{\frac{1}{q}}m_{2}\left( t-s\right) 
}<\infty ,
\end{equation*}%
but the Riemann-Stieltjes integral $\int_{0}^{1}\gamma _{1}\left( t\right)
d\gamma _{2}\left( t\right) $ does not exist.
\end{example}

Proof of Example \ref{Example integration finite is necessary} is give in
page \pageref{Proof of Example integration finite is necessary}.

As a consequence of refined Young integral, we have a sufficient condition
for path to be in $\mathcal{G}_{2}\left( \mathcal{V}\right) $.

\begin{theorem}
\label{Theorem larger space of enhancible paths}Let $\gamma :\left[ 0,1%
\right] \rightarrow \mathcal{V}$ be a continuous paths. If there exists a
non-decreasing function $m:\left[ 0,1\right] \rightarrow \overline{%
\mathbb{R}
^{+}}$ satisfying%
\begin{equation*}
\lim_{t\rightarrow 0}m\left( t\right) =0\text{, }m\left( 1\right) \leq 1%
\text{, and }\int_{0}^{1}\frac{m^{2}\left( t\right) }{t}dt<\infty ,
\end{equation*}%
such that 
\begin{equation}
\sup_{0\leq s<t\leq 1}\frac{\left\Vert \gamma \left( t\right) -\gamma \left(
s\right) \right\Vert }{\left\vert t-s\right\vert ^{\frac{1}{2}}m\left(
t-s\right) }<\infty .  \label{Condition on gamam in G2}
\end{equation}%
Then $\gamma \in \mathcal{G}_{2}\left( \mathcal{V}\right) $.
\end{theorem}

Theorem \ref{Theorem larger space of enhancible paths} is proved in page %
\pageref{Proof of Theorem larger space of enhancible paths}.

\begin{remark}
In Theorem \ref{Theorem larger space of enhancible paths}, by adding a $\log 
$ term and $\log $-$\log $ term so on and so forth, one can get a sequence
of nested spaces in $\mathcal{G}_{2}\left( \mathcal{V}\right) $. Because of
inclusion, their union is still a space in $\mathcal{G}_{2}\left( \mathcal{V}%
\right) $.
\end{remark}

\begin{remark}
As a consequence of Example \ref{Example integration finite is necessary},
for any non-decreasing function $m:\left[ 0,1\right] \rightarrow \overline{%
\mathbb{R}
^{+}}$, $\lim_{t\rightarrow 0}m\left( t\right) =0$, $m\left( 1\right) \leq 1$
and $\int_{0}^{1}\frac{m^{2}\left( t\right) }{t}dt=\infty $, there exists $%
\gamma :\left[ 0,1\right] \rightarrow 
\mathbb{C}
$ satisfying $\left( \ref{Condition on gamam in G2}\right) $ but not in $%
\mathcal{G}_{2}\left( \mathcal{%
\mathbb{C}
}\right) $.
\end{remark}

\section{Proofs}

Recall $\bigtriangleup _{\left[ 0,1\right] }=\left\{ \left( s,t\right)
|0\leq s\leq t\leq 1\right\} $.

\begin{lemma}
\label{Lemma existence of Cp}For any $p>1$ and any $a>0$, there exists
constant $C_{a,p}>0$, such that for any integer $m\geq 1$,%
\begin{equation*}
\sum_{k=1}^{m}\frac{2^{2\left( 1-\frac{1}{p}\right) k}}{k^{a}}\leq C_{a,p}%
\frac{2^{2\left( 1-\frac{1}{p}\right) m}}{m^{a}}.
\end{equation*}
\end{lemma}

\begin{proof}
Fix $p>1$. Denote $b:=2^{2\left( 1-\frac{1}{p}\right) }$. Firstly, suppose $%
c>0$ is a constant, and $\sum_{k=1}^{m_{1}}k^{-a}b^{k}\leq c\left(
m_{1}\right) ^{-a}b^{m_{1}}$. Then $\sum_{k=1}^{m_{1}+1}k^{-a}b^{k}\leq
c\left( m_{1}+1\right) ^{-a}b^{m_{1}+1}$ would hold provided:%
\begin{equation*}
c\frac{b^{m_{1}}}{m_{1}^{a}}+\frac{b^{m_{1}+1}}{\left( m_{1}+1\right) ^{a}}%
\leq c\frac{b^{m_{1}+1}}{\left( m_{1}+1\right) ^{a}}\text{, i.e. }\left(
\left( c-1\right) ^{\frac{1}{a}}b^{\frac{1}{a}}-c^{\frac{1}{a}}\right)
m_{1}\geq c^{\frac{1}{a}}.
\end{equation*}%
Then we choose $C\ $in this way: Fix constant $C_{1}>\frac{b}{b-1}$, and let%
\begin{equation*}
C_{a,p}:=C_{1}\vee \max \left\{ \frac{m^{a}}{b^{m}}\sum_{k=1}^{m}\frac{b^{k}%
}{k^{a}},1\leq m\leq \left[ \frac{C_{1}^{\frac{1}{a}}}{\left( C_{1}-1\right)
^{\frac{1}{a}}b^{\frac{1}{a}}-C_{1}^{\frac{1}{a}}}\right] +1\right\} .
\end{equation*}
\end{proof}

The following lemma is in the form of Exercise 9.14 in \cite{P. Friz}, only
that we give an uniform estimates.

\begin{lemma}
\label{Lemma, decay of 2-var}Suppose $\mathcal{V}$ is a Banach space, $%
\varphi _{n}:\bigtriangleup _{\left[ 0,1\right] }\rightarrow \mathcal{V}$, $%
n\geq 1$, and there exists constant $M>0$ s.t.%
\begin{equation*}
\left\Vert \varphi _{n}\left( s,t\right) \right\Vert \leq M\left( 1\wedge
\left\vert t-s\right\vert \right) ,\forall \left( s,t\right) \in
\bigtriangleup _{\left[ 0,T\right] },\forall n\geq 1\text{.}
\end{equation*}

For $p\in \left( 1,\infty \right) $, $a\in \left( 0,\infty \right) $ and
integers $1\leq N_{1}\leq N_{2}\leq \infty $, define 
\begin{equation*}
g_{N_{1},N_{2}}^{a,p}\left( s,t\right) =\sum_{k=N_{1}}^{N_{2}}\frac{1}{%
k^{a}2^{\frac{2k}{p}}}\varphi _{k}\left( 2^{2k}s,2^{2k}t\right) \text{, }%
t\in \left[ 0,1\right] .
\end{equation*}%
Then%
\begin{equation}
(i)\text{ }\sup_{1\leq N_{1}\leq N_{2}\leq \infty }\sup_{0\leq s<t\leq 1}%
\frac{\left\Vert g_{N_{1},N_{2}}^{a,p}\left( s,t\right) \right\Vert }{%
\left\vert t-s\right\vert ^{\frac{1}{p}}\left( \ln \frac{2}{t-s}\right) ^{-a}%
}\leq C_{a,p.M}<\infty ;  \label{inequality holder continuity}
\end{equation}%
for any $\delta \in \left( 0,1\right) $ (recall $\omega _{p}\left( \gamma
,\delta \right) $ defined at $\left( \ref{Definition vanishing p-variation}%
\right) $), we have 
\begin{equation}
(ii)\sup_{1\leq N_{1}\leq N_{2}\leq \infty }\omega _{p}\left(
g_{N_{1},N_{2}}^{a,p},\delta \right) \leq C_{a,p,M}\left( \ln \frac{2}{%
\delta }\right) ^{-a};  \label{inequality vanishing 2-var}
\end{equation}%
and for any fixed $N_{1}\geq 1$,%
\begin{equation}
(iii)\sup_{N_{1}\leq N_{2}\leq \infty }\left\Vert
g_{N_{1},N_{2}}^{a,p}\right\Vert _{p-var,\left[ 0,1\right] }\leq \frac{%
\widetilde{C_{a,p,M}}}{N_{1}^{a}},  \label{inequality 2-var tends to zero}
\end{equation}%
where $C_{a,p,M}=\left( \ln 4\right) ^{a}2^{-\frac{1}{p}}M\left(
8C_{a,p}+\left( 2^{\frac{2}{p}}-1\right) ^{-1}\right) $ with $C_{a,p}$ from
Lemma \ref{Lemma existence of Cp}, and $\widetilde{C_{a,p,M}}=\left( \left(
\ln 4\right) ^{-ap}C_{a,p.M}^{p}+2M^{p}\left( 1-2^{-\frac{2}{p}}\right)
^{-p}\right) ^{\frac{1}{p}}$.
\end{lemma}

\begin{proof}
For $\left( \ref{inequality holder continuity}\right) $. Fix $0\leq s<t\leq
1 $. Denote $n:=\left[ \log _{4}\frac{8}{t-s}\right] $, then use $\left\Vert
\varphi _{k}\left( s,t\right) \right\Vert \leq M\left( 1\wedge \left\vert
t-s\right\vert \right) $, we get 
\begin{eqnarray*}
\left\Vert g_{N_{1},N_{2}}^{a,p}\left( s,t\right) \right\Vert &\leq
&\sum_{k=1}^{n}\frac{1}{k^{a}2^{\frac{2k}{p}}}\left\Vert \varphi _{k}\left(
2^{2k}s,2^{2k}t\right) \right\Vert +\sum_{k=n+1}^{\infty }\frac{1}{k^{a}2^{%
\frac{2k}{p}}}\left\Vert \varphi _{k}\left( 2^{2k}s,2^{2k}t\right)
\right\Vert \\
&\leq &M\sum_{k=1}^{n}\frac{2^{2\left( 1-\frac{1}{p}\right) k}}{k^{a}}%
\left\vert t-s\right\vert +\sum_{k=n+1}^{\infty }\frac{M}{k^{a}2^{\frac{2k}{p%
}}}.
\end{eqnarray*}%
Based on Lemma \ref{Lemma existence of Cp}, there exists $C_{a,p}$, s.t. for
any $m\geq 1$, $\sum_{k=1}^{m}k^{-a}2^{2\left( 1-\frac{1}{p}\right) k}\leq
C_{a,p}m^{-a}2^{2\left( 1-\frac{1}{p}\right) m}$. Thus ($n>\log _{4}\frac{2}{%
t-s}$ and $\frac{2}{t-s}<2^{2n}\leq \frac{8}{t-s}$), 
\begin{eqnarray*}
\left\Vert g_{N_{1},N_{2}}^{a,p}\left( s,t\right) \right\Vert &\leq &MC_{a,p}%
\frac{2^{2\left( 1-\frac{1}{p}\right) n}}{n^{a}}\left\vert t-s\right\vert +%
\frac{M}{2^{\frac{2}{p}}-1}\frac{1}{n^{a}2^{\frac{2n}{p}}} \\
&\leq &M\left( 8C_{a,p}+\frac{1}{2^{\frac{2}{p}}-1}\right) \frac{1}{n^{a}2^{%
\frac{2n}{p}}} \\
&\leq &\frac{\left( \ln 4\right) ^{a}M}{2^{\frac{1}{p}}}\left( 8C_{a,p}+%
\frac{1}{2^{\frac{2}{p}}-1}\right) \left\vert t-s\right\vert ^{\frac{1}{p}%
}\left( \ln \frac{2}{t-s}\right) ^{-a}\text{.}
\end{eqnarray*}%
Since our estimates holds for any $0\leq s<t\leq 1$ and any integers $1\leq
N_{1}\leq N_{2}\leq \infty $, $\left( \ref{inequality holder continuity}%
\right) $ is done.

Based on $\left( \ref{inequality holder continuity}\right) $, for any $%
\delta \in \left( 0,1\right) $, and any finite partition $D=\left\{
t_{j}\right\} $, $\left\vert D\right\vert \leq \delta $, we have%
\begin{equation*}
\sum_{j,t_{j}\in D}\left\Vert g_{N_{1},N_{2}}^{a,p}\left(
t_{j},t_{j+1}\right) \right\Vert ^{p}\leq C_{a,p.M}^{p}\left( \ln \frac{2}{%
\delta }\right) ^{-ap}\sum_{j,t_{j}\in D}\left\vert t_{j+1}-t_{j}\right\vert
=C_{a,p,M}^{p}\left( \ln \frac{2}{\delta }\right) ^{-ap}.
\end{equation*}%
It holds for any $D$, $\left\vert D\right\vert \leq \delta $, and any
integers $1\leq N_{1}\leq N_{2}\leq \infty $, so $\left( \ref{inequality
vanishing 2-var}\right) $ holds.

Then we prove $\left( \ref{inequality 2-var tends to zero}\right) $. Fix $%
N_{1}$. Finite partitions whose mesh less then $2^{-2N_{1}}$ is done in $%
\left( \ref{inequality vanishing 2-var}\right) $:%
\begin{equation}
\sup_{N_{1}\leq N_{2}\leq \infty }\sup_{\left\vert D\right\vert \leq
2^{-2N_{1}}}\sum_{j,t_{j}\in D}\left\Vert g_{N_{1},N_{2}}^{a,p}\left(
t_{j},t_{j+1}\right) \right\Vert ^{p}\leq \frac{C_{a,p.M}^{p}}{\left( \ln
4\right) ^{ap}}\frac{1}{N_{1}^{ap}{}}\text{.}  \label{inner3}
\end{equation}%
For finite partitions $D=\left\{ t_{j}\right\} $ satisfying $\left\vert
D\right\vert >2^{-2N_{1}}$, we denote $J_{N_{1}+}:=\left\{ j|\left\vert
t_{j+1}-t_{j}\right\vert >2^{-2N_{1}}\right\} $. Since there can not be more
than $2\times 2^{2N_{1}}$ many subintervals in $J_{N_{1}+}$ (and using $%
\left\vert \varphi _{n}\left( s,t\right) \right\vert \leq M$) 
\begin{equation*}
\sum_{t_{j}\in D,j\in J_{N_{1}+}}\left\Vert g_{N_{1},N_{2}}^{a,p}\left(
t_{j},t_{j+1}\right) \right\Vert ^{p}\leq 2^{2N_{1}+1}\left(
\sum_{k=N_{1}}^{\infty }\frac{M}{k^{a}2^{\frac{2k}{p}}}\right) ^{p}\leq
2\left( \frac{2^{\frac{2}{p}}M}{2^{\frac{2}{p}}-1}\right) ^{p}\frac{1}{%
N_{1}^{ap}}.
\end{equation*}%
The intervals in $D$ which are not in $J_{N_{1}+}$ can be treated as
subintervals in another finite partition $D^{\prime }$, $\left\vert
D^{\prime }\right\vert \leq 2^{-2N_{1}}$, so using $\left( \ref{inner3}%
\right) $ to bound them, we get 
\begin{equation*}
\sum_{t_{j}\in D}\left\Vert g_{N_{1},N_{2}}^{a,p}\left( t_{j},t_{j+1}\right)
\right\Vert ^{p}\leq \sum_{j\notin J_{N_{1}+}}+\sum_{j\in J_{N_{1}+}}\leq
\left( \frac{C_{a,p,M}^{p}}{\left( \ln 4\right) ^{ap}}+2\left( \frac{2^{%
\frac{2}{p}}M}{2^{\frac{2}{p}}-1}\right) ^{p}\right) \frac{1}{N_{1}^{ap}{}}.
\end{equation*}%
Our estimates hold for any finite partition $D$, and for any integer $%
N_{2}\geq N_{1}$, so $\left( \ref{inequality 2-var tends to zero}\right) $
holds.
\end{proof}

\begin{example}
\label{Example non-existence of Reimann integral}Suppose $c>\pi $ is a
constant, and $\left\{ l_{n}\right\} $ is a sequence of increasing integers,
satisfying 
\begin{equation}
c^{n}\leq \sum_{k=l_{n}}^{l_{n+1}-1}\frac{1}{k}\leq c^{n}+1\text{, }\forall
n\geq 1.  \label{condition on ln}
\end{equation}%
If define $f:\left[ 0,1\right] \rightarrow 
\mathbb{C}
$ as%
\begin{equation*}
f\left( t\right) =\sum_{n=1}^{\infty }\sum_{k=l_{n}}^{l_{n+1}-1}\frac{1}{k^{%
\frac{1}{2}}2^{k}}\exp \left( 2\pi i\left( -1\right) ^{n}2^{2k}t\right) 
\text{, \ }t\in \left[ 0,1\right] \text{.}
\end{equation*}%
Then $f$ is of vanishing $2$-variation, and for any $a\in \left[ -\infty
,\infty \right] $, there exists a sequence of finite partition $\left\{
D_{n}^{a}\right\} $ of $\left[ 0,1\right] $ satisfying (with $x:=\func{Re}f$%
, $y:=\func{Im}f$)%
\begin{equation}
\lim_{n\rightarrow \infty }\left\vert D_{n}^{a}\right\vert =0\text{ and }%
\lim_{n\rightarrow \infty }\sum_{l,t_{l}\in D_{n}^{a}}\left( x\left(
t_{l}\right) y\left( t_{l+1}\right) -y\left( t_{l}\right) x\left(
t_{l+1}\right) \right) =a.  \label{condition on D_n^a}
\end{equation}
\end{example}

The $\left( -1\right) ^{n}$ ensure that the limit oscillates. If without $%
\left( -1\right) ^{n}$ we only get divergence, while not non-existence.

\begin{proof}
$f$ of vanishing $2$-variation follows from $\left( \ref{inequality
vanishing 2-var}\right) $ in Lemma \ref{Lemma, decay of 2-var} (with $a=%
\frac{1}{2}$, $p=2$, $M=1$, $N_{1}=1$, $N_{2}=\infty $). Suppose $N\geq 1$
is an integer, denote%
\begin{equation}
D_{N}:=\left\{ l2^{-2N}\right\} _{l=0}^{2^{2N}}\text{, }t_{l}^{N}:=l2^{-2N}%
\text{, }l=0,1,\dots ,2^{2N}\text{,}  \label{inner definition Dn}
\end{equation}%
\begin{equation}
\text{and }\left\langle f,D_{N}\right\rangle :=\sum_{l=0}^{2^{2N}-1}\left(
x\left( t_{l}^{N}\right) y\left( t_{l+1}^{N}\right) -y\left(
t_{l}^{N}\right) x\left( t_{l+1}^{N}\right) \right) .
\label{inner definition of sigmaD_N}
\end{equation}%
We want to prove that for each $a\in \left[ -\infty ,\infty \right] $, there
exists a sequence of finite partitions $\left\{ D_{n}^{a}\right\}
_{n}\subset \left\{ D_{N}\right\} _{N}$ , satisfying $\lim_{n\rightarrow
\infty }\left\langle f,D_{n}^{a}\right\rangle =a$.

Denote 
\begin{equation*}
\epsilon _{k}=\left( -1\right) ^{n}\text{, }k=l_{n},\dots ,l_{n+1}-1\text{, }%
c_{k}^{N}=2\pi 2^{2k-2N}\epsilon _{k}\text{, \ }k=l_{1},\dots ,N-1.
\end{equation*}%
Then $2\pi \epsilon _{k}2^{2k}t_{l}=lc_{k}^{N}$, and%
\begin{eqnarray*}
&&x\left( t_{l}^{N}\right) y\left( t_{l+1}^{N}\right) -y\left(
t_{l}^{N}\right) x\left( t_{l+1}^{N}\right) \\
&=&\left( \sum_{j=l_{1}}^{N-1}\frac{1}{j^{\frac{1}{2}}2^{j}}\cos \left( 2\pi
\epsilon _{j}2^{2j}t_{l}^{N}\right) \right) \left( \sum_{k=l_{1}}^{N-1}\frac{%
1}{k^{\frac{1}{2}}2^{k}}\sin \left( 2\pi \epsilon
_{k}2^{2k}t_{l+1}^{N}\right) \right) \\
&&-\left( \sum_{j=l_{1}}^{N-1}\frac{1}{j^{\frac{1}{2}}2^{j}}\sin \left( 2\pi
\epsilon _{j}2^{2j}t_{l}^{N}\right) \right) \left( \sum_{k=l_{1}}^{N-1}\frac{%
1}{k^{\frac{1}{2}}2^{k}}\cos \left( 2\pi \epsilon
_{k}2^{2k}t_{l+1}^{N}\right) \right) \\
&=&\sum_{k,j=l_{1}}^{N-1}\frac{1}{k^{\frac{1}{2}}j^{\frac{1}{2}}2^{k+j}}\sin
\left( \left( l+1\right) c_{k}^{N}-lc_{j}^{N}\right) \\
&=&\sum_{k=l_{1}}^{N-1}\frac{1}{k2^{2k}}\sin \left( 2\pi \epsilon
_{k}2^{2k-2N}\right) \\
&&+\sum_{l_{1}\leq k<j\leq N-1}\frac{1}{k^{\frac{1}{2}}j^{\frac{1}{2}}2^{k+j}%
}\left( \sin \left( l\left( c_{k}^{N}-c_{j}^{N}\right) +c_{k}^{N}\right)
+\sin \left( l\left( c_{j}^{N}-c_{k}^{N}\right) +c_{j}^{N}\right) \right)
\end{eqnarray*}%
Sum $l$ from $0$ to $2^{2N}-1$, 
\begin{eqnarray*}
\left\langle f,D_{N}\right\rangle &=&\sum_{l=0}^{2^{2N}-1}x\left(
t_{l}^{N}\right) y\left( t_{l+1}^{N}\right) -y\left( t_{l}^{N}\right)
x\left( t_{l+1}^{N}\right) \\
&=&\sum_{k=l_{1}}^{N-1}\frac{1}{k2^{2k-2N}}\sin \left( 2\pi \epsilon
_{k}2^{2k-2N}\right) \\
&&+\sum_{l_{1}\leq k<j\leq N-1}\frac{1}{k^{\frac{1}{2}}j^{\frac{1}{2}}2^{k+j}%
}\sum_{l=0}^{2^{2N}-1}\left( \sin \left( l\left( c_{k}^{N}-c_{j}^{N}\right)
+c_{k}^{N}\right) +\sin \left( l\left( c_{j}^{N}-c_{k}^{N}\right)
+c_{j}^{N}\right) \right)
\end{eqnarray*}%
Since 
\begin{equation*}
\sum_{l=0}^{2^{2N}-1}\sin \left( l\left( c_{k}^{N}-c_{j}^{N}\right)
+c_{k}^{N}\right) =\sum_{l=0}^{2^{2N}-1}\sin \left( l\left(
c_{j}^{N}-c_{k}^{N}\right) +c_{j}^{N}\right) =0,
\end{equation*}%
so%
\begin{eqnarray*}
\left\langle f,D_{N}\right\rangle &=&\sum_{k=l_{1}}^{N-1}\frac{1}{k2^{2k-2N}}%
\sin \left( 2\pi \epsilon _{k}2^{2k-2N}\right) \\
&=&:\sum_{j=1}^{J-1}\left( -1\right) ^{j}s_{j}^{N}+\left( -1\right)
^{J}\sum_{k=l_{J}}^{N-1}\frac{1}{k2^{2k-2N}}\sin \left( 2\pi
2^{2k-2N}\right) .
\end{eqnarray*}%
where $l_{J}+1\leq N\leq l_{J+1}$, and%
\begin{equation*}
s_{j}^{N}:=\sum_{k=l_{j}}^{l_{j+1}-1}\frac{1}{k2^{2k-2N}}\sin \left( 2\pi
2^{2k-2N}\right) \text{, }1\leq j\leq J-1\text{.}
\end{equation*}%
Using $\frac{2}{\pi }\theta \leq \sin \theta \leq \theta $ when $\theta \in %
\left[ 0,\frac{\pi }{2}\right] $ and condition $\left( \ref{condition on ln}%
\right) $, we have, for any $j\geq 1$, and any $N\geq l_{j+1}$, 
\begin{equation*}
4\times c^{j}\leq s_{j}^{N}\leq 2\pi \times \left( c^{j}+1\right) .
\end{equation*}%
Thus using $s_{j}^{N}-s_{j-1}^{N}\geq \left( 4c-2\pi \right) c^{j-1}-2\pi $,
we estimate $\sum_{j=1}^{m-1}\left( -1\right) ^{j}s_{j}^{N}$. When $m$ is
even and $m\geq 4$, for any $N\geq l_{m}$,%
\begin{eqnarray}
\sum_{j=1}^{m-1}\left( -1\right) ^{j}s_{j}^{N} &=&-\left(
s_{m-1}^{N}-s_{m-2}^{N}\right) -\dots -s_{1}^{N}
\label{lower limit infinity} \\
&\leq &-\frac{4c-2\pi }{c^{2}-1}\left( c^{m}-c^{2}\right) +\pi \left(
m-2\right) -4c.  \notag
\end{eqnarray}%
Similarly, when $m$ is odd and $m\geq 5$, for any $N\geq l_{m}$, 
\begin{eqnarray}
\sum_{j=1}^{m-1}\left( -1\right) ^{j}s_{j}^{N} &=&\left(
s_{m-1}^{N}-s_{m-2}^{N}\right) +\dots +(s_{2}^{N}-s_{1}^{N})
\label{upper limit infinity} \\
&\geq &\frac{4c-2\pi }{c^{2}-1}\left( c^{m}-c\right) -\pi \left( m-1\right) ;
\notag
\end{eqnarray}%
and when $m$ is odd and $m\geq 5$, for any $N\geq l_{m}$, the upper bound:%
\begin{eqnarray}
\sum_{j=1}^{m-1}\left( -1\right) ^{j}s_{j}^{N} &=&s_{m-1}^{N}-\left(
s_{m-2}^{N}-s_{m-3}^{N}\right) -\dots -s_{1}^{N}  \label{lower upper limit}
\\
&\leq &2\pi \times \left( c^{m-1}+1\right) -\frac{4c-2\pi }{c^{2}-1}\left(
c^{m-1}-c^{2}\right) +\pi \left( m-3\right) -4c  \notag \\
&=&\left( \frac{2\pi }{c}-\frac{4c-2\pi }{c\left( c^{2}-1\right) }\right)
c^{m}+\pi \left( m-1\right) +\frac{4c-2\pi }{c^{2}-1}c^{2}-4c.  \notag
\end{eqnarray}

Since we assumed $c>\pi $, so in $\left( \ref{lower limit infinity}\right) $
and $\left( \ref{upper limit infinity}\right) $, $\frac{4c-2\pi }{c^{2}-1}>0$%
. On the other hand, since $\left\langle f,D_{l_{m}}\right\rangle
=\sum_{j=1}^{m-1}\left( -1\right) ^{j}s_{j}^{l_{m}}$, so based on $\left( %
\ref{lower limit infinity}\right) $ and $\left( \ref{upper limit infinity}%
\right) $, we have%
\begin{equation*}
\lim_{n\rightarrow \infty }\left\langle f,D_{l_{2n}}\right\rangle =-\infty 
\text{ and }\lim_{n\rightarrow \infty }\left\langle
f,D_{l_{2n+1}}\right\rangle =+\infty \text{.}
\end{equation*}%
Thus, if when $a=+\infty $ let $D_{n}^{a}:=D_{l_{2n+1}}$, when $a=-\infty $
let $D_{n}^{a}:=D_{l_{2n}}$, then when $a=+\infty $ or $-\infty $, we have $%
\lim_{n\rightarrow \infty }\left\vert D_{n}^{a}\right\vert =0$, and $%
\lim_{n\rightarrow \infty }\left\langle f,D_{n}^{a}\right\rangle =a$.

Fix $a\in \left( -\infty ,\infty \right) $.

Firstly, we assumed $c>\pi $, so 
\begin{equation*}
0<\frac{2\pi }{c}-\frac{4c-2\pi }{c\left( c^{2}-1\right) }<\frac{2\pi }{c}<2%
\text{.}
\end{equation*}%
For our fixed $c>\pi $, choose integer $M_{c}\geq 1$, s.t. for any $m\geq
M_{c}$,%
\begin{equation*}
\left( \frac{2\pi }{c}-\frac{4c-2\pi }{c\left( c^{2}-1\right) }\right)
c^{m}+\pi \left( m-1\right) +\frac{4c-2\pi }{c^{2}-1}c^{2}-4c\leq 2c^{m}.
\end{equation*}%
Thus, combined with $\left( \ref{lower upper limit}\right) $, when $m$ is
odd and $m\geq 5\vee M_{c}$, for any $N\geq l_{m}$, we have%
\begin{equation}
\sum_{j=1}^{m-1}\left( -1\right) ^{j}s_{j}^{N}\leq 2c^{m}\text{.}
\label{upper bound for s_j^N}
\end{equation}%
Then for our fixed $a\in (-\infty ,\infty )$, choose odd integer $M\left(
a\right) \geq 5\vee M_{c}$ such that, for any odd integer $m\geq M\left(
a\right) $, and any $N\geq l_{m}$, we have 
\begin{equation}
\sum_{j=1}^{m-1}\left( -1\right) ^{j}s_{j}^{N}>\left\vert a\right\vert
+10\pi \text{,}  \label{lower bound for s_j^N}
\end{equation}%
which is possible because of $\left( \ref{upper limit infinity}\right) $.

We prove that for any odd integer $m\geq M\left( a\right) $, there exists $%
N_{m}\left( a\right) $, $l_{m}<N_{m}\left( a\right) <l_{m+1}$, s.t. 
\begin{equation*}
\left\vert \left\langle f,D_{N_{m}\left( a\right) }\right\rangle
-a\right\vert \leq \frac{\pi }{l_{m}}\text{.}
\end{equation*}%
Fix odd integer $m\geq M\left( a\right) $. For any $N\geq l_{m}$ (use $%
c^{m}\leq \sum_{k=l_{m}}^{l_{m+1}-1}k^{-1}$, i.e.$\left( \ref{condition on
ln}\right) $),%
\begin{equation}
\left\vert a\right\vert +10\pi <\sum_{j=1}^{m-1}\left( -1\right)
^{j}s_{j}^{N}\leq 2c^{m}\leq 2\sum_{k=l_{m}}^{l_{m+1}-1}k^{-1}.
\label{inner bound M(a)}
\end{equation}%
Thus, when $N=l_{m+1}$ in $\left( \ref{inner bound M(a)}\right) $, we have 
\begin{align}
\left\langle f,D_{l_{m+1}}\right\rangle & =\sum_{j=1}^{m-1}\left( -1\right)
^{j}s_{j}^{l_{m+1}}-\sum_{k=l_{m}}^{l_{m+1}-1}\frac{\sin \left( 2\pi
2^{2k-2l_{m+1}}\right) }{k2^{2k-2l_{m+1}}}  \label{inner D_l_M+1} \\
& \leq \sum_{j=1}^{m-1}\left( -1\right)
^{j}s_{j}^{l_{m+1}}-4\sum_{k=l_{m}}^{l_{m+1}-1}k^{-1}  \notag \\
& \leq -2\sum_{k=l_{m}}^{l_{m+1}-1}k^{-1}<-\left\vert a\right\vert -10\pi 
\text{ .}  \notag
\end{align}%
While in $\left( \ref{inner bound M(a)}\right) $ let $N=l_{m}$, we have%
\begin{equation}
\left\langle f,D_{l_{m}}\right\rangle =\sum_{j=1}^{m-1}\left( -1\right)
^{j}s_{j}^{m}>\left\vert a\right\vert +10\pi \text{.}  \label{inner D_l_M}
\end{equation}%
Combine $\left( \ref{inner D_l_M+1}\right) $ with $\left( \ref{inner D_l_M}%
\right) $, if $\left\vert \left\langle f,D_{N}\right\rangle -\left\langle
f,D_{N+1}\right\rangle \right\vert $ is uniformly small when $l_{m}\leq
N\leq l_{m+1}-1$, then $\exists N_{m}\left( a\right) $, $l_{m}\leq
N_{1}\left( a\right) \leq l_{m+1}$, s.t. $\left\langle f,D_{N_{m}\left(
a\right) }\right\rangle $ is in the neighborhood of $a$.

Actually, for any $N\geq l_{1}+1$, 
\begin{eqnarray*}
&&\left\vert \left\langle f,D_{N+1}\right\rangle -\left\langle
f,D_{N}\right\rangle \right\vert \\
&\leq &\left\vert \sum_{k=l_{1}}^{N-1}\frac{1}{k2^{2k-2\left( N+1\right) }}%
\left( \frac{\sin \left( 2\pi 2^{2k-2N}\right) }{4}-\sin \left( 2\pi
2^{2k-2\left( N+1\right) }\right) \right) \right\vert +\frac{4}{N}.
\end{eqnarray*}%
For any $\theta \in \left[ 0,\frac{\pi }{2}\right] $, using $\sin \left(
2\theta \right) =2\sin \theta \cos \theta $, we have%
\begin{gather*}
\frac{1}{\theta }\left\vert \frac{\sin \left( 4\theta \right) }{4}-\sin
\theta \right\vert =\frac{\sin \theta }{\theta }\left\vert \cos \theta \cos
2\theta -1\right\vert \\
\leq \left\vert \left( 1-2\sin ^{2}\frac{\theta }{2}\right) \left( 1-2\sin
^{2}\theta \right) -1\right\vert \leq 14\sin ^{2}\frac{\theta }{2}\leq \frac{%
7}{2}\theta ^{2}.
\end{gather*}%
Thus let $\theta =2\pi 2^{2k-2\left( N+1\right) }$, we have%
\begin{equation*}
\frac{1}{2^{2k-2\left( N+1\right) }}\left\vert \frac{\sin \left( 2\pi
2^{2k-2N}\right) }{4}-\sin \left( 2\pi 2^{2k-2\left( N+1\right) }\right)
\right\vert \leq 28\pi ^{3}\left( \frac{1}{2^{2\left( N+1\right) -2k}}%
\right) ^{2}.
\end{equation*}%
Thus, when $l_{m}\leq N\leq l_{m+1}-1$, 
\begin{equation}
\left\vert \left\langle f,D_{N+1}\right\rangle -\left\langle
f,D_{N}\right\rangle \right\vert \leq 28\pi ^{3}\sum_{k=l_{1}}^{N-1}\frac{1}{%
k}\left( \frac{1}{2^{2\left( N+1\right) -2k}}\right) ^{2}+\frac{4}{N}.
\label{inner2}
\end{equation}%
While one can prove that for any $m\geq 2$, $\sum_{k=1}^{m-1}\frac{2^{4k}}{k}%
\leq \frac{2^{4m}}{m}$ by using mathematical induction, so for any $N\geq
l_{1}+1$, 
\begin{equation}
\sum_{k=l_{1}}^{N-1}\frac{1}{k}\left( \frac{1}{2^{2\left( N+1\right) -2k}}%
\right) ^{2}\leq \frac{1}{2^{4N+4}}\sum_{k=1}^{N-1}\frac{2^{4k}}{k}\leq 
\frac{1}{16N}.  \label{inner1}
\end{equation}

Then, combined $\left( \ref{inner2}\right) $ with $\left( \ref{inner1}%
\right) $, we get when $l_{m}\leq N\leq l_{m+1}-1$, 
\begin{equation*}
\left\vert \left\langle f,D_{N+1}\right\rangle -\left\langle
f,D_{N}\right\rangle \right\vert \leq \left( \frac{7}{4}\pi ^{3}+4\right) 
\frac{1}{N}<\frac{20\pi }{l_{m}}.
\end{equation*}%
Thus, combined with $\left( \ref{inner D_l_M+1}\right) $ and $\left( \ref%
{inner D_l_M}\right) $, there exists integer $N_{m}\left( a\right) $, $%
l_{m}\leq N_{m}\left( a\right) \leq l_{m+1}$, s.t.%
\begin{equation*}
\left\vert \left\langle f,D_{N_{m}\left( a\right) }\right\rangle
-a\right\vert <\frac{10\pi }{l_{m}}.
\end{equation*}%
Moreover, since $\left\langle f,D_{l_{m}}\right\rangle >\left\vert
a\right\vert +10\pi \geq \left\vert a\right\vert +\frac{10\pi }{l_{m}}$, $%
\left\langle f,D_{l_{m+1}}\right\rangle <-\left\vert a\right\vert -10\pi
\leq -\left\vert a\right\vert -\frac{10\pi }{l_{m}}$, so $l_{m}<N_{m}\left(
a\right) <l_{m+1}$.

Therefore, if let $D_{m}^{a}:=D_{N_{m}\left( a\right) }$, $m\geq 1$, then $%
\left\{ D_{m}^{a}\right\} _{m}$ is a sequence of finite partitions, whose
mesh tends to zero, but the limit of the corresponding Riemann sum is $a$.
\end{proof}

Next, we demonstrate that when the space of smooth paths is equipped with $2$%
-variation, the area operator is unbounded, and non-closable when the area
is equipped with $p$-variation, $p>1$.

\begin{lemma}
\label{triangle converge in p-variation}Suppose $\left\{ l_{n}\right\} _{n}$
is a sequence of strictly increasing integers. Then%
\begin{equation*}
\lim_{n\rightarrow \infty }\left\Vert \sum_{k=l_{n}}^{l_{n+1}-1}\frac{1}{%
k2^{2k}}\sin \left( 2\pi 2^{2k}\left( t-s\right) \right) \right\Vert _{p-var,%
\left[ 0,1\right] }=0\text{ for any }p>1\text{.}
\end{equation*}
\end{lemma}

\begin{proof}
We do estimation for fixed $p>1$ and fixed sufficiently large $n$.

For integer $m\geq l_{n}$, denote $I_{m}$ $:=(2^{-2p\left( m+1\right)
},2^{-2pm}]$, and denote $I_{l_{n}+}:=(2^{-2pl_{n}},1]$. Suppose $D=\left\{
t_{j}\right\} $ is a finite partition satisfying that $\left\{ \left\vert
t_{j+1}-t_{j}\right\vert \right\} _{j}\subset \cup _{i=1}^{s}I_{m_{i}}\cup
I_{l_{n}+}$ with $\min_{1\leq i\leq s}m_{i}\geq l_{n}$. Denote $%
J_{m_{i}}:=\left\{ j|t_{j+1}-t_{j}\in I_{m_{i}}\right\} $ and $%
J_{l_{n}+}:=\left\{ j|t_{j+1}-t_{j}\in I_{l_{n}+}\right\} $. We assume that $%
J_{m_{i}}$ is not empty for each $i$. For $J_{l_{n}+}$, since we can not
have more than $2^{2pl_{n}+1}\sum_{j,j\in J_{l_{n}+}}\left(
t_{j+1}-t_{j}\right) $ intervals in $J_{l_{n}+}$, so%
\begin{eqnarray}
&&\sum_{j,j\in J_{l_{n}+}}\left( \sum_{k=l_{n}}^{l_{n+1}-1}\frac{1}{k2^{2k}}%
\sin \left( 2\pi 2^{2k}\left( t_{j+1}-t_{j}\right) \right) \right) ^{p}
\label{estimation for large interval} \\
&\leq &2^{2pl_{n}+1}\left( \sum_{k=l_{n}}^{l_{n+1}-1}\frac{1}{k2^{2k}}%
\right) ^{p}\sum_{j,j\in J_{l_{n}+}}\left( t_{j+1}-t_{j}\right) \leq \frac{%
2^{2p+1}}{3^{p}l_{n}^{p}}\sum_{j,j\in J_{l_{n}+}}\left( t_{j+1}-t_{j}\right)
.  \notag
\end{eqnarray}%
Then we do estimation for fixed $i$, $i=1,2,\dots ,s$. Suppose $%
t_{j+1}-t_{j}\in I_{m_{i}}$, then%
\begin{eqnarray*}
&&\left( \sum_{k=l_{n}}^{l_{n+1}-1}\frac{1}{k2^{2k}}\sin \left( 2\pi
2^{2k}\left( t_{j+1}-t_{j}\right) \right) \right) ^{p} \\
&\leq &2^{p-1}\left( \left( 2\pi \sum_{k=l_{n}}^{m_{i}}\frac{1}{k}\right)
^{p}\left\vert t_{j+1}-t_{j}\right\vert ^{p}+\left( \sum_{k=m_{i}+1}^{\infty
}\frac{1}{k2^{2k}}\right) ^{p}\right) \\
&\leq &2^{p-1}\left( \left( 2\pi \left( 1+\ln m_{i}\right) \right)
^{p}\left( \frac{1}{2^{2p^{2}m_{i}}}\right) +\frac{1}{%
3^{p}m_{i}^{p}2^{2pm_{i}}}\right) .
\end{eqnarray*}%
Since there can not be more than $2\times 2^{2p\left( m_{i}+1\right)
}\sum_{j,j\in J_{m_{i}}}\left( t_{j+1}-t_{j}\right) $ many intervals whose
length fail into the category $I_{m_{i}}$, so 
\begin{eqnarray}
&&\sum_{j,j\in J_{m_{i}}}\left( \sum_{k=l_{n}}^{l_{n+1}-1}\frac{1}{k2^{2k}}%
\sin \left( 2\pi 2^{2k}\left( t_{j+1}-t_{j}\right) \right) \right) ^{p}
\label{estimation for small intervals} \\
&\leq &2^{p-1}\left( \left( 2\pi \right) ^{p}\frac{\left( 1+\ln m_{i}\right)
^{p}}{2^{2p^{2}m_{i}}}+\frac{1}{3^{p}m_{i}^{p}2^{2pm_{i}}}\right) \times
2^{2p\left( m_{i}+1\right) +1}\sum_{j,j\in J_{m_{i}}}\left(
t_{j+1}-t_{j}\right)  \notag \\
&\leq &2^{3p}\left( \left( 2\pi \right) ^{p}\frac{\left( 1+\ln m_{i}\right)
^{p}}{2^{2p\left( p-1\right) m_{i}}}+\frac{1}{3^{p}m_{i}^{p}}\right)
\sum_{j,j\in J_{m_{i}}}\left( t_{j+1}-t_{j}\right) .  \notag
\end{eqnarray}%
Since $\left\{ l_{n}\right\} $ are strictly increasing integers, so $%
\lim_{n\rightarrow \infty }l_{n}=+\infty $. Thus, for our fixed $p>1$, there
exists $N\left( p\right) \geq 1$, s.t. for any $n\geq N\left( p\right) $ and
any $m_{i}\geq l_{n}$, we have 
\begin{equation*}
\frac{\left( 1+\ln m_{i}\right) ^{p}}{2^{2p\left( p-1\right) m_{i}}}\leq 
\frac{1}{m_{i}^{p}}.
\end{equation*}%
Therefore, for any fixed finite partition $D=\left\{ t_{j}\right\} $ of $%
\left[ 0,1\right] $, when $n\geq N\left( p\right) $, we have (using $\left( %
\ref{estimation for large interval}\right) $, $\left( \ref{estimation for
small intervals}\right) $ and $\sum_{i=1}^{s}\sum_{j\in J_{m_{i}}}\left(
t_{j+1}-t_{j}\right) +\sum_{j\in J_{l_{n}+}}\left( t_{j+1}-t_{j}\right) =1$, 
$\min_{1\leq i\leq s}m_{i}\geq l_{n}$) 
\begin{eqnarray*}
&&\sum_{j,t_{j}\in D}\left( \sum_{k=l_{n}}^{l_{n+1}-1}\frac{1}{k2^{2k}}\sin
\left( 2\pi 2^{2k}\left( t_{j+1}-t_{j}\right) \right) \right) ^{p} \\
&\leq &2^{3p}\left( \left( 2\pi \right) ^{p}+\frac{1}{3^{p}}\right) \left(
\sum_{i=1}^{s}\frac{1}{m_{i}^{p}}\sum_{j,j\in J_{m_{i}}}\left(
t_{j+1}-t_{j}\right) \right) +\frac{2^{2p+1}}{3^{p}l_{n}^{p}}\sum_{j,j\in
J_{l_{n}+}}\left( t_{j+1}-t_{j}\right) \\
&\leq &2^{3p}\left( \left( 2\pi \right) ^{p}+\frac{1}{3^{p}}\right) \frac{1}{%
l_{n}^{p}}\text{. }
\end{eqnarray*}%
Hence, for any fixed $p>1$, there exists integer $N\left( p\right) $, s.t.
for any $n\geq N\left( p\right) $, 
\begin{equation*}
\left\Vert \sum_{k=l_{n}}^{l_{n+1}-1}\frac{1}{k2^{2k}}\sin \left( 2\pi
2^{2k}\left( t-s\right) \right) \right\Vert _{p-var,\left[ 0,1\right]
}^{p}\leq 2^{3p}\left( \left( 2\pi \right) ^{p}+\frac{1}{3^{p}}\right) \frac{%
1}{l_{n}^{p}}\text{.}
\end{equation*}%
Proof finishes.
\end{proof}

\begin{lemma}
\label{Lemma 1-var convergence}Suppose $\left\{ l_{n}\right\} _{n}$ is a
sequence of strictly increasing integers. Define%
\begin{equation*}
g_{n}\left( t\right) =\sum_{k=l_{n}}^{l_{n+1}-1}\frac{1}{k^{\frac{1}{2}}2^{k}%
}\exp \left( 2\pi i2^{2k}t\right) \text{, }t\in \left[ 0,1\right] .
\end{equation*}%
Then $\lim_{n\rightarrow \infty }\left\Vert g_{n}\right\Vert _{2-var}=0$,
and for any $p>1$,%
\begin{equation*}
\lim_{n\rightarrow \infty }\left\Vert A\left( g_{n}\right) \left( s,t\right)
-\left( \pi \sum_{k=l_{n}}^{l_{n+1}-1}\frac{1}{k}\right) \left( t-s\right)
\right\Vert _{p-var,\left[ 0,1\right] }=0.
\end{equation*}
\end{lemma}

\begin{proof}
Since trigonometric functions are Lipschitz and bounded, so according to $%
\left( \ref{inequality 2-var tends to zero}\right) $ in Lemma \ref{Lemma,
decay of 2-var} with $p=2$, $\lim_{n\rightarrow \infty }\left\Vert
g_{n}\right\Vert _{2-var,\left[ 0,1\right] }=0$.

According to the definition of area, if denote $x_{n}:=\func{Re}g_{n}$, $%
y_{n}:=\func{Im}g_{n}$, and%
\begin{eqnarray*}
p_{n}\left( s,t\right) &:&=\int_{s}^{t}x_{n}\left( u\right) dy_{n}\left(
u\right) -y_{n}\left( u\right) dx_{n}\left( u\right) , \\
\text{ }q_{n}\left( s,t\right) &:&=y_{n}\left( s\right) x_{n}\left( t\right)
-x_{n}\left( s\right) y_{n}\left( t\right) ,
\end{eqnarray*}%
we have 
\begin{equation*}
A\left( g_{n}\right) \left( s,t\right) =\frac{1}{2}\left( p_{n}\left(
s,t\right) +q_{n}\left( s,t\right) \right) .
\end{equation*}%
Firstly, for $p_{n}\left( s,t\right) $, 
\begin{eqnarray*}
p_{n}\left( s,t\right) &=&2\pi \int_{s}^{t}\left( \sum_{i=l_{n}}^{l_{n+1}-1}%
\frac{1}{i^{\frac{1}{2}}2^{i}}\cos \left( 2\pi 2^{2i}u\right) \right) \left(
\sum_{j=l_{n}}^{l_{n+1}-1}\frac{2^{j}}{j^{\frac{1}{2}}}\cos \left( 2\pi
2^{2j}u\right) \right) \\
&&\text{ \ \ \ \ \ \ \ \ \ \ \ \ }+\left( \sum_{j=l_{n}}^{l_{n+1}-1}\frac{1}{%
j^{\frac{1}{2}}2^{j}}\sin \left( 2\pi 2^{2j}u\right) \right) \left(
\sum_{j=l_{n}}^{l_{n+1}-1}\frac{2^{i}}{i^{\frac{1}{2}}}\sin \left( 2\pi
2^{2i}u\right) \right) du \\
&=&2\pi \sum_{i,j=l_{n}}^{l_{n+1}-1}\int_{s}^{t}\frac{2^{j-i}}{i^{\frac{1}{2}%
}j^{\frac{1}{2}}}\cos \left( 2\pi 2^{2i}u\right) \cos \left( 2\pi
2^{2j}u\right) +\frac{2^{i-j}}{i^{\frac{1}{2}}j^{\frac{1}{2}}}\sin \left(
2\pi 2^{2j}u\right) \sin \left( 2\pi 2^{2i}u\right) du \\
&=&\left( 2\pi \sum_{k=l_{n}}^{l_{n+1}-1}\frac{1}{k}\right) \left(
t-s\right) +2\pi \sum_{l_{n}\leq i<j\leq l_{n+1}-1}\left( \frac{2^{j-i}}{i^{%
\frac{1}{2}}j^{\frac{1}{2}}}+\frac{2^{i-j}}{i^{\frac{1}{2}}j^{\frac{1}{2}}}%
\right) \int_{s}^{t}\cos \left( 2\pi \left( 2^{2j}-2^{2i}\right) u\right) du
\\
&=&:\left( 2\pi \sum_{k=l_{n}}^{l_{n+1}-1}\frac{1}{k}\right) \left(
t-s\right) +\sum_{l_{n}\leq i<j\leq l_{n+1}-1}\frac{1}{i^{\frac{1}{2}}j^{%
\frac{1}{2}}2^{i+j}}p_{i,j}\left( s,t\right) ,
\end{eqnarray*}%
where%
\begin{equation*}
p_{i,j}\left( s,t\right) :=\left( \frac{2^{2j}+2^{2i}}{2^{2j}-2^{2i}}\right)
\left( \sin \left( 2\pi \left( 2^{2j}-2^{2i}\right) t\right) -\sin \left(
2\pi \left( 2^{2j}-2^{2i}\right) s\right) \right) .
\end{equation*}%
While, for $q_{n}\left( s,t\right) $, 
\begin{eqnarray*}
&&q_{n}\left( s,t\right) =y_{n}\left( s\right) x_{n}\left( t\right)
-x_{n}\left( s\right) y_{n}\left( t\right) \\
&=&\left( \sum_{i=l_{n}}^{l_{n+1}-1}\frac{1}{i^{\frac{1}{2}}2^{i}}\sin
\left( 2\pi 2^{2i}s\right) \right) \left( \sum_{j=l_{n}}^{l_{n+1}-1}\frac{1}{%
j^{\frac{1}{2}}2^{j}}\cos \left( 2\pi 2^{2j}t\right) \right) \\
&&-\left( \sum_{i=l_{n}}^{l_{n+1}-1}\frac{1}{i^{\frac{1}{2}}2^{i}}\cos
\left( 2\pi 2^{2i}s\right) \right) \left( \sum_{j=l_{n}}^{l_{n+1}-1}\frac{1}{%
j^{\frac{1}{2}}2^{j}}\sin \left( 2\pi 2^{2j}t\right) \right) \\
&=&\sum_{i,j=l_{n}}^{l_{n+1}-1}\frac{1}{i^{\frac{1}{2}}j^{\frac{1}{2}}2^{i+j}%
}\sin \left( 2\pi \left( 2^{2i}s-2^{2j}t\right) \right)
\end{eqnarray*}%
\begin{equation*}
=-\sum_{k=l_{n}}^{l_{n+1}-1}\frac{1}{k2^{2k}}\sin \left( 2\pi 2^{2k}\left(
t-s\right) \right) +\sum_{l_{n}\leq i<j\leq l_{n+1}-1}\frac{1}{i^{\frac{1}{2}%
}j^{\frac{1}{2}}2^{i+j}}q_{i,j}\left( s,t\right) ,
\end{equation*}%
where%
\begin{equation*}
q_{i,j}\left( s,t\right) =\sin \left( 2\pi \left( 2^{2i}s-2^{2j}t\right)
\right) +\sin \left( 2\pi \left( 2^{2j}s-2^{2i}t\right) \right) .
\end{equation*}%
Thus%
\begin{eqnarray*}
&&A\left( g_{n}\right) \left( s,t\right) -\left( \pi
\sum_{k=l_{n}}^{l_{n+1}-1}\frac{1}{k}\right) \left( t-s\right) \\
&=&\frac{1}{2}\left( -\sum_{k=l_{n}}^{l_{n+1}-1}\frac{1}{k2^{2k}}\sin \left(
2\pi 2^{2k}\left( t-s\right) \right) +\sum_{l_{n}\leq i<j\leq l_{n+1}-1}%
\frac{1}{i^{\frac{1}{2}}j^{\frac{1}{2}}2^{i+j}}\left( p_{i,j}\left(
s,t\right) +q_{i,j}\left( s,t\right) \right) \right) .
\end{eqnarray*}%
Based on Lemma \ref{triangle converge in p-variation}, $%
\sum_{k=l_{n}}^{l_{n+1}-1}k^{-1}2^{-2k}\sin \left( 2\pi 2^{2k}\left(
t-s\right) \right) $ converge to $0$ as $n$ tends to infinity in $p$%
-variation for any $p>1$, so we are left with 
\begin{equation*}
\sum_{l_{n}\leq i<j\leq l_{n+1}-1}\frac{1}{i^{\frac{1}{2}}j^{\frac{1}{2}%
}2^{i+j}}\left( p_{i,j}\left( s,t\right) +q_{i,j}\left( s,t\right) \right) .
\end{equation*}

While%
\begin{eqnarray*}
&&p_{i,j}\left( s,t\right) +q_{i,j}\left( s,t\right) \\
&=&\left( \frac{2^{2j}+2^{2i}}{2^{2j}-2^{2i}}\right) \left( \sin \left( 2\pi
\left( 2^{2j}-2^{2i}\right) t\right) -\sin \left( 2\pi \left(
2^{2j}-2^{2i}\right) s\right) \right) \\
&&+\sin \left( 2\pi \left( 2^{2i}s-2^{2j}t\right) \right) +\sin \left( 2\pi
\left( 2^{2j}s-2^{2i}t\right) \right) \\
&=&\left( \frac{2\times 2^{2i}}{2^{2j}-2^{2i}}\right) \left( \sin \left(
2\pi \left( 2^{2j}-2^{2i}\right) t\right) -\sin \left( 2\pi \left(
2^{2j}-2^{2i}\right) s\right) \right) \\
&&+\left( \sin \left( 2\pi \left( 2^{2j}-2^{2i}\right) t\right) +\sin \left(
2\pi \left( 2^{2i}s-2^{2j}t\right) \right) \right) +\left( \sin \left( 2\pi
\left( 2^{2j}s-2^{2i}t\right) \right) -\sin \left( 2\pi \left(
2^{2j}-2^{2i}\right) s\right) \right) \\
&=&\left( \frac{2\times 2^{2i}}{2^{2j}-2^{2i}}\right) \left( \sin \left(
2\pi \left( 2^{2j}-2^{2i}\right) t\right) -\sin \left( 2\pi \left(
2^{2j}-2^{2i}\right) s\right) \right) \\
&&-2\cos \left( 2\pi \left( 2^{2j}t-2^{2i}\frac{t+s}{2}\right) \right) \sin
\left( 2\pi 2^{2i}\frac{t-s}{2}\right) -2\cos \left( 2\pi \left(
2^{2j}s-2^{2i}\frac{t+s}{2}\right) \right) \sin \left( 2\pi 2^{2i}\frac{t-s}{%
2}\right) \\
&=&\left( \frac{4\times 2^{2i}}{2^{2j}-2^{2i}}\right) \cos \left( 2\pi
\left( 2^{2j}-2^{2i}\right) \frac{t+s}{2}\right) \sin \left( 2\pi \left(
2^{2j}-2^{2i}\right) \frac{t-s}{2}\right) \\
&&-4\cos \left( 2\pi \left( \left( 2^{2j}-2^{2i}\right) \frac{t+s}{2}\right)
\right) \cos \left( 2\pi 2^{2j}\frac{t-s}{2}\right) \sin \left( 2\pi 2^{2i}%
\frac{t-s}{2}\right) \\
&=&4\cos \left( 2\pi \left( \left( 2^{2j}-2^{2i}\right) \frac{t+s}{2}\right)
\right) \left( \left( \frac{2^{2i}}{2^{2j}-2^{2i}}\right) \sin \left( 2\pi
\left( 2^{2j}-2^{2i}\right) \frac{t-s}{2}\right) \right. \\
&&\left. -\cos \left( 2\pi 2^{2j}\frac{t-s}{2}\right) \sin \left( 2\pi 2^{2i}%
\frac{t-s}{2}\right) \right) .
\end{eqnarray*}%
Therefore,%
\begin{eqnarray}
&&\left\vert p_{i,j}\left( s,t\right) +q_{i,j}\left( s,t\right) \right\vert
\label{estimates for p_ij plus q_ij} \\
&\leq &4\left( \frac{2^{2i+2j}}{2^{2j}-2^{2i}}\right) \left\vert \frac{\sin
\left( 2\pi 2^{2j}\frac{t-s}{2}\right) }{2^{2j}}\cos \left( 2\pi 2^{2i}\frac{%
t-s}{2}\right) -\frac{\sin \left( 2\pi 2^{2i}\frac{t-s}{2}\right) }{2^{2i}}%
\cos \left( 2\pi 2^{2j}\frac{t-s}{2}\right) \right\vert .  \notag
\end{eqnarray}%
While, since for any $\theta $, and any integer $n\geq 1$, 
\begin{equation*}
\sin \theta \tprod\limits_{k=0}^{n-1}\cos \left( 2^{k}\theta \right) =\frac{%
\sin \left( 2^{n}\theta \right) }{2^{n}}\text{, }
\end{equation*}%
so, when $j>i$,%
\begin{equation*}
\frac{\sin \left( 2^{2j}\theta \right) }{2^{2j}}=\frac{\sin \left(
2^{2i}\theta \right) }{2^{2i}}\tprod\limits_{k=2i}^{2j-1}\cos \left(
2^{k}\theta \right) .
\end{equation*}%
Thus when $\theta =\pi \left( t-s\right) $, continue with $\left( \ref%
{estimates for p_ij plus q_ij}\right) $, we have%
\begin{eqnarray*}
&&\left\vert p_{i,j}\left( s,t\right) +q_{i,j}\left( s,t\right) \right\vert
\\
&\leq &4\left( \frac{2^{2i+2j}}{2^{2j}-2^{2i}}\right) \left\vert \frac{\sin
\left( 2\pi 2^{2j}\frac{t-s}{2}\right) }{2^{2j}}\cos \left( 2\pi 2^{2i}\frac{%
t-s}{2}\right) -\frac{\sin \left( 2\pi 2^{2i}\frac{t-s}{2}\right) }{2^{2i}}%
\cos \left( 2\pi 2^{2j}\frac{t-s}{2}\right) \right\vert \\
&=&4\left( \frac{2^{2i+2j}}{2^{2j}-2^{2i}}\right) \left\vert \frac{\sin
\left( 2\pi 2^{2i}\frac{t-s}{2}\right) }{2^{2i}}\right\vert \left\vert \cos
\left( 2\pi 2^{2i}\frac{t-s}{2}\right) \tprod\limits_{k=2i}^{2j-1}\cos
\left( 2\pi 2^{k}\frac{t-s}{2}\right) -\cos \left( 2\pi 2^{2j}\frac{t-s}{2}%
\right) \right\vert \\
&\leq &4\left( \frac{2^{2i+2j}}{2^{2j}-2^{2i}}\right) \frac{\left\vert \sin
\left( 2\pi 2^{2i}\frac{t-s}{2}\right) \right\vert }{2^{2i}}\times 2\leq 
\frac{32}{3}\left\vert \sin \left( 2\pi 2^{2i}\frac{t-s}{2}\right)
\right\vert \text{.\ \ \ (}\frac{2^{2j}}{2^{2j}-2^{2i}}\leq \frac{4}{3}\text{
when }j>i\text{)}
\end{eqnarray*}%
Therefore, for any $p\in \left( 1,2\right) $,%
\begin{eqnarray}
&&\sum_{l_{n}\leq i<j\leq l_{n+1}-1}\frac{1}{i^{\frac{1}{2}}j^{\frac{1}{2}%
}2^{i+j}}\left\vert p_{i,j}\left( s,t\right) +q_{i,j}\left( s,t\right)
\right\vert  \label{bound for p_ij plus q_ij} \\
&\leq &\frac{32}{3}\sum_{j=l_{n}+1}^{l_{n+1}-1}\frac{1}{j^{\frac{1}{2}}2^{j}}%
\left( \sum_{i=l_{n}}^{j-1}\frac{1}{i^{\frac{1}{2}}2^{i}}\left\vert \sin
\left( 2\pi 2^{2i}\frac{t-s}{2}\right) \right\vert \right)  \notag \\
&\leq &\frac{32}{3}\sum_{j=l_{n}+1}^{l_{n+1}-1}\frac{1}{j^{\frac{1}{2}%
}2^{\left( 2-\frac{2}{p}\right) j}}\left( \sum_{i=l_{n}}^{j-1}\frac{1}{i^{%
\frac{1}{2}}2^{\frac{2}{p}i}}\left\vert \sin \left( 2\pi 2^{2i}\frac{t-s}{2}%
\right) \right\vert \right) .  \notag
\end{eqnarray}%
While since $\left\vert \sin \left( t-s\right) \right\vert \leq 1\wedge
\left\vert t-s\right\vert $, based on $\left( \ref{inequality 2-var tends to
zero}\right) $ in Lemma \ref{Lemma, decay of 2-var}, for any $p>1$, there
exists a constant $\widetilde{C_{\frac{1}{2},p,1}}$, s.t. for any $l_{n}$
and~any $j>l_{n}$, we have, 
\begin{equation*}
\left\Vert \sum_{i=l_{n}}^{j-1}\frac{1}{i^{\frac{1}{2}}2^{\frac{2}{p}i}}%
\left\vert \sin \left( 2\pi 2^{2i}\frac{t-s}{2}\right) \right\vert
\right\Vert _{p-var}\leq \frac{\widetilde{C_{\frac{1}{2},p,1}}}{l_{n}^{\frac{%
1}{2}}}.
\end{equation*}%
Therefore, for any $p\in \left( 1,2\right) $, since $\left\Vert \cdot
\right\Vert _{p-var}$ is a norm, combined with $\left( \ref{bound for p_ij
plus q_ij}\right) $, 
\begin{eqnarray*}
&&\left\Vert \sum_{l_{n}\leq i<j\leq l_{n+1}-1}\frac{1}{i^{\frac{1}{2}}j^{%
\frac{1}{2}}2^{i+j}}\left( p_{i,j}\left( s,t\right) +q_{i,j}\left(
s,t\right) \right) \right\Vert _{p-var} \\
&\leq &\frac{32}{3}\sum_{j=l_{n}+1}^{l_{n+1}-1}\frac{1}{j^{\frac{1}{2}%
}2^{2\left( 1-\frac{1}{p}\right) j}}\left\Vert \sum_{i=l_{n}}^{j-1}\frac{1}{%
i^{\frac{1}{2}}2^{\frac{2}{p}i}}\left\vert \sin \left( 2\pi 2^{2i}\frac{t-s}{%
2}\right) \right\vert \right\Vert _{p-var} \\
&\leq &\frac{32\widetilde{C_{\frac{1}{2},p,1}}}{3l_{n}^{\frac{1}{2}}}%
\sum_{j=l_{n}+1}^{l_{n+1}-1}\frac{1}{j^{\frac{1}{2}}2^{2\left( 1-\frac{1}{p}%
\right) j}}\leq \frac{32\widetilde{C_{\frac{1}{2},p,1}}}{3(2^{2\left( 1-%
\frac{1}{p}\right) }-1)}\frac{1}{l_{n}2^{2\left( 1-\frac{1}{p}\right) l_{n}}}%
\rightarrow 0\text{ as }n\rightarrow \infty \text{.}
\end{eqnarray*}%
Thus, for any $p>1$ (since $p$-variation is non-increasing, so if converge
in $p$-variation, $p\in \left( 1,2\right) $, then converge in $p$-variation, 
$p>1$)%
\begin{equation*}
\lim_{n\rightarrow \infty }\left\Vert A\left( g_{n}\right) \left( s,t\right)
-\left( \pi \sum_{k=l_{n}}^{l_{n+1}-1}\frac{1}{k}\right) \left( t-s\right)
\right\Vert _{p-var}=0.
\end{equation*}
\end{proof}

\begin{example}
\label{Example divergence in p-variation,p>1}Suppose $\left\{ l_{n}\right\} $
is a sequence of increasing integers, satisfying that for any $n\geq 1$, $%
\sum_{k=l_{n}}^{l_{n+1}-1}k^{-1}\geq n$. Define 
\begin{equation}
f_{n}\left( t\right) =\sum_{k=l_{n}}^{l_{n+1}-1}\frac{1}{k^{\frac{1}{2}}2^{k}%
}\exp \left( 2\pi i2^{2k}t\right) \text{, }t\in \left[ 0,1\right] .
\label{Definition of gn}
\end{equation}%
Then $\lim_{n\rightarrow \infty }\left\Vert f_{n}\right\Vert _{2-var,\left[
0,1\right] }=0$, but for any $0\leq s<t\leq 1$, $\lim_{n\rightarrow \infty
}A\left( f_{n}\right) \left( s,t\right) =+\infty $.
\end{example}

\begin{proof}
Follows from Lemma \ref{Lemma 1-var convergence}:%
\begin{equation*}
\lim_{n\rightarrow \infty }\left\Vert A\left( f_{n}\right) \left( s,t\right)
-\left( \pi \sum_{k=l_{n}}^{l_{n+1}-1}\frac{1}{k}\right) \left( t-s\right)
\right\Vert _{p-var}=0\text{, for any }p>1\text{.}
\end{equation*}
\end{proof}

As a clear consequence of this example, when the space of smooth paths is
equipped with $2$-variation, the area operator is not continuous, nor
bounded.

\begin{example}
\label{Example, non-closability}Suppose $\left\{ l_{n}\right\} $ is a
sequence of increasing integers, satisfying that for any $n\geq 1$, $%
\sum_{k=l_{n}}^{l_{n+1}-1}k^{-1}\geq 1$. Define 
\begin{equation*}
g_{n}\left( t\right) =\left( \pi \sum_{k=l_{n}}^{l_{n+1}-1}\frac{1}{k}%
\right) ^{-\frac{1}{2}}\sum_{k=l_{n}}^{l_{n+1}-1}\frac{1}{k^{\frac{1}{2}%
}2^{k}}\exp \left( 2\pi i2^{2k}t\right) \text{, }t\in \left[ 0,1\right] .
\end{equation*}%
Then $\lim_{n\rightarrow \infty }\left\Vert g_{n}\right\Vert _{2-var,\left[
0,1\right] }=0$, and for any $p>1$, 
\begin{equation*}
\lim_{n\rightarrow \infty }\left\Vert A\left( g_{n}\right) \left( s,t\right)
-\left( t-s\right) \right\Vert _{p-var,\left[ 0,1\right] }=0.
\end{equation*}
\end{example}

\begin{proof}
Follows from Lemma \ref{Lemma 1-var convergence}.
\end{proof}

The convergence of $A\left( g_{n}\right) $ to $t-s$ can not hold in $1$%
-variation, because $h_{n}$ is a sequence of smooth paths, so the limit of $%
A\left( g_{n}\right) $ in $1$-variation is of vanishing $1$-variation, while 
$t-s$ is not. Actually, since $g_{n}$ converge to zero in $2$-variation, so
if $A\left( g_{n}\right) $ converge in $1$-variation then should converge to 
$0$ (closable when area equipped with $1$-variation).

Example \ref{Example, non-closability} demonstrates that when the space of
smooth paths is equipped with $2$-variation and their area with $p$%
-variation, $p>1$, the area operator is not closable.

Next, we extend Young integral \cite{L. C. Young} to the case $%
p^{-1}+q^{-1}=1$ by assigning a finer scale continuity (e.g. logarithmic).
Before that, we prove a lemma. Recall definition of $\omega _{p}\left(
\gamma ,\delta \right) $ at $\left( \ref{Definition vanishing p-variation}%
\right) $.

\begin{lemma}
\label{Lemma difference of 1-var of area}Suppose $\gamma _{1}\in
C^{p-var}\left( \left[ 0,1\right] ,\mathcal{V}\right) $, $\gamma _{2}\in
C^{q-var}\left( \left[ 0,1\right] ,\mathcal{V}\right) $, $p^{-1}+q^{-1}=1$. $%
D_{1}=\left\{ t_{k}\right\} _{k}$ and $D_{2}=\left\{ s_{j}\right\} _{j}$ are
two finite partitions of $\left[ 0,1\right] $, and $D_{2}$ is a refinement
of $D_{1}$, i.e. for any $k$, there exist integers $n_{k}<n_{k+1}$, s.t. $%
t_{k}=s_{n_{k}}<s_{n_{k}+1}<\dots <s_{n_{k+1}}=t_{k+1}$. Then if denote $%
I^{D}:=I\left( \gamma _{1}^{D},\gamma _{2}^{D}\right) $ (see definition at $%
\left( \ref{Definition area of two paths}\right) $) and suppose $\left\vert
D_{1}\right\vert \leq \delta $, we have 
\begin{eqnarray*}
\left\Vert I^{D_{1}}-I^{D_{2}}\right\Vert _{1-var,\left[ 0,1\right] } &\leq
&\sum_{k}\left\Vert I^{D_{1}}\right\Vert _{1-var,\left[ t_{k},t_{k+1}\right]
}+\sum_{k}\left\Vert I^{D_{2}}\right\Vert _{1-var,\left[ t_{k},t_{k+1}\right]
} \\
&&+2\omega _{p}\left( \gamma _{1},\delta \right) \left\Vert \gamma
_{2}\right\Vert _{q-var,\left[ 0,1\right] }+2\omega _{q}\left( \gamma
_{2},\delta \right) \left\Vert \gamma _{1}\right\Vert _{p-var,\left[ 0,1%
\right] }.
\end{eqnarray*}
\end{lemma}

\begin{proof}
Denote $\bigtriangleup \gamma _{i}:=\gamma _{i}^{D_{1}}-\gamma _{i}^{D_{2}}$%
, $i=1,2$, denote $\bigtriangleup I:=I^{D_{1}}-I^{D_{2}}$. For any $\left(
s,t\right) \in \bigtriangleup _{\left[ 0,T\right] }$, 
\begin{eqnarray*}
\bigtriangleup I\left( s,t\right) &=&\int_{s}^{t}\left( \gamma
_{1}^{D_{1}}\left( u\right) -\gamma _{1}^{D_{1}}\left( s\right) \right)
\otimes d\gamma _{2}^{D_{1}}\left( u\right) -\int_{s}^{t}\left( \gamma
_{1}^{D_{2}}\left( u\right) -\gamma _{1}^{D_{2}}\left( s\right) \right)
\otimes d\gamma _{2}^{D_{2}}\left( u\right) \\
&=&\int_{s}^{t}\left( \bigtriangleup \gamma _{1}\left( u\right)
-\bigtriangleup \gamma _{1}\left( s\right) \right) \otimes d\gamma
_{2}^{D_{1}}\left( u\right) +\int_{s}^{t}\left( \gamma _{1}^{D_{2}}\left(
u\right) -\gamma _{1}^{D_{2}}\left( s\right) \right) \otimes d\bigtriangleup
\!\gamma _{2}\left( u\right) \\
&=&:I\left( \bigtriangleup \gamma _{1},\gamma _{2}^{D_{1}}\right) \left(
s,t\right) +I\left( \gamma _{1}^{D_{2}},\bigtriangleup \gamma _{2}\right)
\left( s,t\right) \text{.}
\end{eqnarray*}%
Suppose $t_{k_{1}-1}<s\leq t_{k_{1}}\leq t_{k_{2}}\leq t<t_{k_{2}+1}$, then 
\begin{eqnarray*}
I\left( \bigtriangleup \gamma _{1},\gamma _{2}^{D_{1}}\right) \left(
s,t\right) &=&I\left( \bigtriangleup \gamma _{1},\gamma _{2}^{D_{1}}\right)
\left( s,t_{k_{1}}\right) +I\left( \bigtriangleup \gamma _{1},\gamma
_{2}^{D_{1}}\right) \left( t_{k_{1}},t_{k_{2}}\right) +I\left(
\bigtriangleup \gamma _{1},\gamma _{2}^{D_{1}}\right) \left(
t_{k_{2}},t\right) \\
&&+\left( \bigtriangleup \gamma _{1}\left( t_{k_{1}}\right) -\bigtriangleup
\gamma _{1}\left( s\right) \right) \otimes \left( \gamma _{2}^{D_{1}}\left(
t\right) -\gamma _{2}^{D_{1}}\left( t_{k_{1}}\right) \right) \\
&&+\left( \bigtriangleup \gamma _{1}\left( t_{k_{2}}\right) -\bigtriangleup
\gamma _{1}\left( t_{k_{1}}\right) \right) \otimes \left( \gamma
_{2}^{D_{1}}\left( t\right) -\gamma _{2}^{D_{1}}\left( t_{k_{2}}\right)
\right) .
\end{eqnarray*}%
where the last term vanishes, because $\bigtriangleup \gamma _{1}\left(
t_{k_{1}}\right) =\bigtriangleup \gamma _{1}\left( t_{k_{2}}\right) $.
Similar result holds for $I\left( \gamma _{2},\bigtriangleup \gamma \right)
\left( s,t\right) $:%
\begin{eqnarray*}
I\left( \gamma _{1}^{D_{2}},\bigtriangleup \gamma _{2}\right) \left(
s,t\right) &=&I\left( \gamma _{1}^{D_{2}},\bigtriangleup \gamma _{2}\right)
\left( s,t_{k_{1}}\right) +I\left( \gamma _{1}^{D_{2}},\bigtriangleup \gamma
_{2}\right) \left( t_{k_{1}},t_{k_{2}}\right) +I\left( \gamma
_{1}^{D_{2}},\bigtriangleup \gamma _{2}\right) \left( t_{k_{2}},t\right) \\
&&+\left( \gamma _{1}^{D_{2}}\left( t_{k_{2}}\right) -\gamma
_{1}^{D_{2}}\left( s\right) \right) \otimes \left( \bigtriangleup \gamma
_{2}\left( t\right) -\bigtriangleup \gamma _{2}\left( t_{k_{2}}\right)
\right) .
\end{eqnarray*}%
Thus (since $\bigtriangleup I=I\left( \bigtriangleup \gamma ,\gamma
_{1}\right) +I\left( \gamma _{2},\bigtriangleup \gamma \right) $, $%
\left\Vert u\otimes v\right\Vert \leq \left\Vert u\right\Vert \left\Vert
v\right\Vert $)%
\begin{eqnarray}
\left\Vert \bigtriangleup I\left( s,t\right) \right\Vert &\leq &\left\Vert
\bigtriangleup I\left( s,t_{k_{1}}\right) \right\Vert +\left\Vert
\bigtriangleup I\left( t_{k_{1}},t_{k_{2}}\right) \right\Vert +\left\Vert
\bigtriangleup I\left( t_{k_{2}},t\right) \right\Vert
\label{inner estimation for delta A} \\
&&+\left\Vert \bigtriangleup \gamma _{1}\left( t_{k_{1}}\right)
-\bigtriangleup \gamma _{1}\left( s\right) \right\Vert \left\Vert \gamma
_{2}^{D_{1}}\left( t\right) -\gamma _{2}^{D_{1}}\left( t_{k_{1}}\right)
\right\Vert  \notag \\
&&+\left\Vert \gamma _{1}^{D_{2}}\left( t_{k_{2}}\right) -\gamma
_{1}^{D_{2}}\left( s\right) \right\Vert \left\Vert \bigtriangleup \gamma
_{2}\left( t\right) -\bigtriangleup \gamma _{2}\left( t_{k_{2}}\right)
\right\Vert .  \notag
\end{eqnarray}%
For $\bigtriangleup I\left( t_{k_{1}},t_{k_{2}}\right) $, by using
multiplicativity and $\bigtriangleup \gamma _{i}\left( t_{k}\right) =0$,$%
\forall k$, $i=1,2$, we get 
\begin{equation}
\bigtriangleup I\left( t_{k_{1}},t_{k_{2}}\right)
=\sum_{j=k_{1}}^{k_{2}-1}\bigtriangleup I\left( t_{j},t_{j+1}\right) .
\label{inner estimation for delta  A2}
\end{equation}%
Thus, combine $\left( \ref{inner estimation for delta A}\right) $ with $%
\left( \ref{inner estimation for delta A2}\right) $, we decompose $\left[ s,t%
\right] $ into the union of three kinds of subintervals: $\left[ s,t_{k_{1}}%
\right] $, $\left[ t_{j},t_{j+1}\right] $ and $\left[ t_{k_{2}},t\right] $,
and each of them is a subinterval of some $\left[ t_{k},t_{k+1}\right] $.
Thus, for any finite partition, applying our estimates to each subinterval,
summing them together, and taking supremum over all finite partitions. By
using H\"{o}lder inequality, we get%
\begin{eqnarray}
\left\Vert \bigtriangleup I\right\Vert _{1-var} &\leq &\sum_{k}\left\Vert
\bigtriangleup I\right\Vert _{1-var,\left[ t_{k},t_{k+1}\right] }
\label{inner estimation1} \\
&&+\left( \sum_{k}\left\Vert \bigtriangleup \gamma _{1}\right\Vert _{p-var, 
\left[ t_{k},t_{k+1}\right] }^{p}\right) ^{\frac{1}{p}}\left\Vert \gamma
_{2}^{D_{1}}\right\Vert _{q-var,\left[ 0,1\right] }  \notag \\
&&+\left( \sum_{k}\left\Vert \bigtriangleup \gamma _{2}\right\Vert _{q-var, 
\left[ t_{k},t_{k+1}\right] }^{q}\right) ^{\frac{1}{q}}\left\Vert \gamma
_{1}^{D_{2}}\right\Vert _{p-var,\left[ 0,1\right] }  \notag
\end{eqnarray}

On the other hand, when $i=1,2$,%
\begin{equation}
\sup_{D}\left\Vert \gamma _{i}^{D}\right\Vert _{p-var,\left[ 0,1\right]
}\leq \left\Vert \gamma _{i}\right\Vert _{p-var,\left[ 0,1\right] },
\label{inner estimation2}
\end{equation}%
and since $\bigtriangleup \gamma _{i}:=\gamma _{i}^{D_{1}}-\gamma
_{i}^{D_{2}}$, 
\begin{eqnarray}
\left\Vert \bigtriangleup \gamma _{i}\right\Vert _{p-var,\left[ t_{k},t_{k+1}%
\right] } &\leq &\left\Vert \gamma _{i}^{D_{1}}\right\Vert _{p-var,\left[
t_{k},t_{k+1}\right] }+\left\Vert \gamma _{i}^{D_{2}}\right\Vert _{p-var,%
\left[ t_{k},t_{k+1}\right] }  \label{inner estimation3} \\
&\leq &2\left\Vert \gamma _{i}\right\Vert _{p-var,\left[ t_{k},t_{k+1}\right]
}.  \notag
\end{eqnarray}%
Therefore, combine $\left( \ref{inner estimation1}\right) $, $\left( \ref%
{inner estimation2}\right) $ with $\left( \ref{inner estimation3}\right) $, 
\begin{eqnarray*}
\left\Vert \bigtriangleup I\right\Vert _{1-var,\left[ 0,1\right] } &\leq
&\sum_{k}\left\Vert \bigtriangleup I\right\Vert _{1-var,\left[ t_{k},t_{k+1}%
\right] }+2\left( \sum_{k}\left\Vert \gamma _{1}\right\Vert _{p-var,\left[
t_{k},t_{k+1}\right] }^{p}\right) ^{\frac{1}{p}}\left\Vert \gamma
_{2}\right\Vert _{q-var,\left[ 0,1\right] } \\
&&+2\left( \sum_{k}\left\Vert \gamma _{2}\right\Vert _{q-var,\left[
t_{k},t_{k+1}\right] }^{q}\right) ^{\frac{1}{q}}\left\Vert \gamma
_{1}\right\Vert _{p-var,\left[ 0,1\right] }.
\end{eqnarray*}%
Since $\left\Vert \bigtriangleup I\right\Vert _{1-var,\left[ t_{k},t_{k+1}%
\right] }\leq \left\Vert I^{D_{1}}\right\Vert _{1-var,\left[ t_{k},t_{k+1}%
\right] }+\left\Vert I^{D_{2}}\right\Vert _{1-var,\left[ t_{k},t_{k+1}\right]
}$ and $\left\vert D_{1}\right\vert \leq \delta $, recall definition of $%
\omega _{p}\left( \gamma ,\delta \right) $ at $\left( \ref{Definition
vanishing p-variation}\right) $, proof finishes.
\end{proof}

The following lemma will be used in the proof of Theorem \ref{Theorem
generalized Young integral}.

\begin{lemma}
\label{Lemma estimation of area bisecting interval}Suppose $\gamma _{i}:%
\left[ 0,T\right] \rightarrow \mathcal{V}$, $i=1,2$, are two continuous
piecewise linear paths obtained by interpolating on the same finite
partition of $\left[ 0,T\right] $. Then for any $p>1$, $q>1$, $%
p^{-1}+q^{-1}=1$, there exists finite partition $D=\left\{ t_{k}\right\} $
of $\left[ 0,T\right] $, $\left\vert D\right\vert \leq 2^{-1}T$, s.t. 
\begin{equation*}
\left\Vert I\left( \gamma _{1},\gamma _{2}\right) \right\Vert _{1-var,\left[
0,T\right] }\leq \sum_{k,t_{k}\in D}\left\Vert I\left( \gamma _{1},\gamma
_{2}\right) \right\Vert _{1-var,\left[ t_{k},t_{k+1}\right] }+2\left\Vert
\gamma _{1}\right\Vert _{p-var,\left[ 0,T\right] }\left\Vert \gamma
_{2}\right\Vert _{q-var,\left[ 0,T\right] }.
\end{equation*}%
If $\gamma _{i}$ are linear on $\left[ 0,T\right] $ then%
\begin{equation*}
\left\Vert I\left( \gamma _{1},\gamma _{2}\right) \right\Vert _{1-var,\left[
0,T\right] }\leq \left\Vert \gamma _{1}\right\Vert _{p-var,\left[ 0,T\right]
}\left\Vert \gamma _{2}\right\Vert _{q-var,\left[ 0,T\right] }.
\end{equation*}
\end{lemma}

\begin{proof}
Denote $I:=I\left( \gamma _{1},\gamma _{2}\right) $ and denote $D^{\prime
}=\left\{ t_{j}\right\} _{j=0}^{n}$ as the finite partition on which $\gamma
_{i}$, $i=1,2$, are interpolated.

When $n=1$, $\left\{ t_{j}\right\} _{j=0}^{n}=\left\{ 0,T\right\} $, then $%
\gamma _{i}$ are linear on $\left[ 0,T\right] $, $i=1,2$. After computation,
one gets (assume $\left\Vert u\otimes v\right\Vert \leq \left\Vert
u\right\Vert \left\Vert v\right\Vert $)%
\begin{eqnarray}
\left\Vert I\right\Vert _{1-var,\left[ 0,T\right] } &=&\left\Vert \left(
\gamma _{1}\left( T\right) -\gamma _{1}\left( 0\right) \right) \otimes
\left( \gamma _{2}\left( T\right) -\gamma _{2}\left( 0\right) \right)
\right\Vert  \label{estimation of area on a line segment} \\
&\leq &\left\Vert \gamma _{1}\left( T\right) -\gamma _{1}\left( 0\right)
\right\Vert \left\Vert \gamma _{2}\left( T\right) -\gamma _{2}\left(
0\right) \right\Vert  \notag \\
&\leq &\left\Vert \gamma _{1}\right\Vert _{p-var,\left[ 0,T\right]
}\left\Vert \gamma _{2}\right\Vert _{q-var,\left[ 0,T\right] }.  \notag
\end{eqnarray}

When $n\geq 2$, denote $t_{j_{1}}:=\min_{j}\left\{ t_{j}|t_{j}\leq
2^{-1}T\right\} $.

If $t_{j_{1}}=0$, then $j_{1}=0$, and $t_{j_{1}+1}=t_{1}>2^{-1}T$. Thus%
\begin{eqnarray*}
\left\Vert I\right\Vert _{1-var,\left[ 0,T\right] } &\leq &\left\Vert
I\right\Vert _{1-var,\left[ 0,t_{1}\right] }+\left\Vert I\right\Vert _{1-var,%
\left[ t_{1},T\right] } \\
&&+\left\Vert \gamma _{1}\right\Vert _{p-var,\left[ 0,t_{1}\right]
}\left\Vert \gamma _{2}\right\Vert _{q-var,\left[ t_{1},T\right] }.
\end{eqnarray*}%
Use $\left( \ref{estimation of area on a line segment}\right) $ for $%
\left\Vert I\right\Vert _{1-var,\left[ 0,t_{1}\right] }$, 
\begin{eqnarray*}
\left\Vert I\right\Vert _{1-var,\left[ 0,T\right] } &\leq &\left\Vert
I\right\Vert _{1-var,\left[ t_{1},T\right] }+\left\Vert \gamma
_{1}\right\Vert _{p-var,\left[ 0,t_{1}\right] }\left\Vert \gamma
_{2}\right\Vert _{q-var,\left[ 0,t_{1}\right] } \\
&&+\left\Vert \gamma _{1}\right\Vert _{p-var,\left[ 0,t_{1}\right]
}\left\Vert \gamma _{2}\right\Vert _{q-var,\left[ t_{1},T\right] }. \\
&\leq &\left\Vert I\right\Vert _{1-var,\left[ t_{1},T\right] }+2^{\frac{1}{p}%
}\left\Vert \gamma _{1}\right\Vert _{p-var,\left[ 0,T\right] }\left\Vert
\gamma _{2}\right\Vert _{q-var,\left[ 0,T\right] }.
\end{eqnarray*}%
Since $T-t_{1}<2^{-1}T$, lemma holds.

If $t_{j_{1}}>0$, then%
\begin{eqnarray}
\left\Vert I\right\Vert _{1-var,\left[ 0,T\right] } &\leq &\left\Vert
I\right\Vert _{1-var,\left[ 0,t_{j_{1}}\right] }+\left\Vert I\right\Vert
_{1-var,\left[ t_{j_{1}},T\right] }  \label{inner area 1} \\
&&+\left\Vert \gamma _{1}\right\Vert _{p-var,\left[ 0,t_{j_{1}}\right]
}\left\Vert \gamma _{2}\right\Vert _{q-var,\left[ t_{j_{1}},T\right] .} 
\notag
\end{eqnarray}%
Then if $t_{j_{1}+1}=T$, $\gamma _{i}$ are linear on $\left[ t_{j_{1}},T%
\right] $, $i=1,2$, so similar as above,%
\begin{equation*}
\left\Vert I\right\Vert _{1-var,\left[ 0,T\right] }\leq \left\Vert
I\right\Vert _{1-var,\left[ 0,t_{j_{1}}\right] }+2^{\frac{1}{q}}\left\Vert
\gamma _{1}\right\Vert _{p-var,\left[ 0,T\right] }\left\Vert \gamma
_{2}\right\Vert _{q-var,\left[ 0,T\right] }.
\end{equation*}%
Since $t_{j_{1}}\leq 2^{-1}T$, lemma holds.

If $t_{j_{1}+1}<T$, then $0<t_{j_{1}}\leq 2^{-1}T<t_{j_{1}+1}<T$, continue
with $\left( \ref{inner area 1}\right) $,%
\begin{eqnarray}
\left\Vert I\right\Vert _{1-var,\left[ 0,t_{j_{1}}\right] } &\leq
&\left\Vert I\right\Vert _{1-var,\left[ 0,t_{j_{1}}\right] }+\left\Vert
I\right\Vert _{1-var,\left[ t_{j_{1}},t_{j_{1}+1}\right] }
\label{inner area 2} \\
&&+\left\Vert \gamma _{1}\right\Vert _{p-var,\left[ 0,t_{j_{1}}\right]
}\left\Vert \gamma _{2}\right\Vert _{q-var,\left[ t_{j_{1}},t_{j_{1}+1}%
\right] .}  \notag
\end{eqnarray}%
While $\gamma _{i}$ are linear on $\left[ t_{j_{1}},t_{j_{1}+1}\right] $, so 
\begin{equation}
\left\Vert I\right\Vert _{1-var,\left[ t_{j_{1}},t_{j_{1}+1}\right] }\leq
\left\Vert \gamma _{1}\right\Vert _{p-var,\left[ t_{j_{1}},t_{j_{1}+1}\right]
}\left\Vert \gamma _{2}\right\Vert _{q-var,\left[ t_{j_{1}},t_{j_{1}+1}%
\right] }.  \label{inner area 3}
\end{equation}%
Thus, combine $\left( \ref{inner area 1}\right) $, $\left( \ref{inner area 2}%
\right) $ with $\left( \ref{inner area 3}\right) $, using H\"{o}lder
inequality, we get%
\begin{eqnarray*}
\left\Vert I\right\Vert _{1-var,\left[ 0,T\right] } &\leq &\left\Vert
I\right\Vert _{1-var,\left[ 0,t_{j_{1}}\right] }+\left\Vert I\right\Vert
_{1-var,\left[ t_{j_{1}+1},T\right] } \\
&&+2\left\Vert \gamma _{1}\right\Vert _{p-var,\left[ 0,T\right] }\left\Vert
\gamma _{2}\right\Vert _{q-var,\left[ 0,T\right] .}
\end{eqnarray*}%
Since $t_{j_{1}}\leq 2^{-1}T$ and $t_{j_{1}+1}>2^{-1}T$, so lemma holds.
Proof finishes.
\end{proof}

\medskip \medskip

\noindent \textbf{Theorem \ref{Theorem generalized Young integral}} \textit{%
Let }$\gamma _{i}:\left[ 0,1\right] \rightarrow \mathcal{V}_{i}$\textit{, }$%
i=1,2$, \textit{be two continuous paths. If there exist }$p>1$\textit{, }$%
q>1 $\textit{, }$p^{-1}+q^{-1}=1$\textit{, and two non-decreasing functions }%
$m_{i}:\left[ 0,1\right] \rightarrow \overline{%
\mathbb{R}
^{+}}$\textit{, }$i=1,2$\textit{, satisfying}%
\begin{equation*}
\lim_{t\rightarrow 0}m_{i}\left( t\right) =0\text{, }m_{i}\left( 1\right)
\leq 1\text{, }i=1,2\text{, and }\int_{0}^{1}\frac{m_{1}\left( t\right)
m_{2}\left( t\right) }{t}dt<\infty .
\end{equation*}%
\textit{such that }%
\begin{equation*}
C_{1}:=\sup_{0\leq s<t\leq 1}\frac{\left\Vert \gamma _{1}\left( t\right)
-\gamma _{1}\left( s\right) \right\Vert }{\left\vert t-s\right\vert ^{\frac{1%
}{p}}m_{1}\left( t-s\right) }<\infty \text{, }C_{2}:=\sup_{0\leq s<t\leq 1}%
\frac{\left\Vert \gamma _{2}\left( t\right) -\gamma _{2}\left( s\right)
\right\Vert }{\left\vert t-s\right\vert ^{\frac{1}{q}}m_{2}\left( t-s\right) 
}<\infty .
\end{equation*}%
\textit{Then the Riemann-Stieltjes integral }$\int_{0}^{t}\gamma _{1}\left(
t\right) \otimes d\gamma _{2}\left( t\right) $\textit{, }$t\in \left[ 0,1%
\right] $\textit{, exists, and }%
\begin{equation*}
\left\Vert \int_{0}^{\cdot }\gamma _{1}\left( t\right) \otimes d\gamma
_{2}\left( t\right) \right\Vert _{q-var}\leq 8C_{1}C_{2}\left( 2+\int_{0}^{1}%
\frac{m_{1}\left( t\right) m_{2}\left( t\right) }{t}dt\right) .
\end{equation*}

\bigskip

\begin{proof}
\label{Proof of theorem generalized Young integral}Recall the definition of $%
I\left( \gamma _{1}^{D},\gamma _{2}^{D}\right) $ at $\left( \ref{Definition
area of two paths}\right) $:%
\begin{equation*}
I\left( \gamma _{1}^{D},\gamma _{2}^{D}\right) \left( s,t\right)
=\int_{s}^{t}\left( \gamma _{1}^{D}\left( u\right) -\gamma _{1}^{D}\left(
s\right) \right) \otimes d\gamma _{2}^{D}\left( u\right) \text{, \ }0\leq
s<t\leq 1\text{.}
\end{equation*}%
Denote $I^{D_{i}}:=I\left( \gamma _{1}^{D_{i}},\gamma _{2}^{D_{i}}\right) $, 
$i=1,2\,$. Firstly, we prove that $I^{D}$ converge in $1$-variation as $%
\left\vert D\right\vert \rightarrow 0$.

Since $m_{i}$ are non-decreasing, so ($\omega _{p}$ defined at $\left( \ref%
{Definition vanishing p-variation}\right) $)%
\begin{gather}
\omega _{p}\left( \gamma _{1},\delta \right) \leq C_{1}m_{1}\left( \delta
\right) \text{, }\omega _{q}\left( \gamma _{2},\delta \right) \leq
C_{2}m_{2}\left( \delta \right) \text{;}  \label{generalized Young 1} \\
\text{since }\left\vert m_{i}\right\vert \leq 1\text{ so }\left\Vert \gamma
_{1}\right\Vert _{p-var,\left[ 0,T\right] }\leq C_{1}\text{, }\left\Vert
\gamma _{2}\right\Vert _{q-var,\left[ 0,T\right] }\leq C_{2}\text{.}  \notag
\end{gather}%
Based on Lemma \ref{Lemma difference of 1-var of area}, for any finite
partition $D_{1}\subset D_{2}\subset \left[ 0,1\right] $, if $\left\vert
D_{1}\right\vert \leq \delta $ then%
\begin{eqnarray}
&&\left\Vert I^{D_{1}}-I^{D_{2}}\right\Vert _{1-var}
\label{generlized Young formula} \\
&\leq &\sum_{k}\left\Vert I^{D_{1}}\right\Vert _{1-var,\left[ t_{k},t_{k+1}%
\right] }+\sum_{k}\left\Vert I^{D_{2}}\right\Vert _{1-var,\left[
t_{k},t_{k+1}\right] }  \notag \\
&&+2\omega _{p}\left( \gamma _{1},\delta \right) \left\Vert \gamma
_{2}\right\Vert _{q-var,\left[ 0,1\right] }+2\omega _{q}\left( \gamma
_{2},\delta \right) \left\Vert \gamma _{1}\right\Vert _{p-var,\left[ 0,1%
\right] }.  \notag
\end{eqnarray}%
Combined with $\left( \ref{generalized Young 1}\right) $, we get%
\begin{eqnarray}
&&2\omega _{p}\left( \gamma _{1},\delta \right) \left\Vert \gamma
_{2}\right\Vert _{q-var,\left[ 0,1\right] }+2\omega _{q}\left( \gamma
_{2},\delta \right) \left\Vert \gamma _{1}\right\Vert _{p-var,\left[ 0,1%
\right] }  \label{generalized Young 2} \\
&\leq &2C_{1}C_{2}\left( m_{1}\left( \delta \right) +m_{2}\left( \delta
\right) \right)  \notag
\end{eqnarray}

For $\sum_{k}\left\Vert I^{D_{1}}\right\Vert _{1-var,\left[ t_{k},t_{k+1}%
\right] }$. Since $D_{1}$ is linear on $\left[ t_{k},t_{k+1}\right] $, so%
\begin{eqnarray*}
\left\Vert I^{D_{1}}\right\Vert _{1-var,\left[ t_{k},t_{k+1}\right] } &\leq
&\left\Vert \gamma _{1}\left( t_{k+1}\right) -\gamma _{1}\left( t_{k}\right)
\right\Vert \left\Vert \gamma _{2}\left( t_{k+1}\right) -\gamma _{2}\left(
t_{k}\right) \right\Vert \\
&\leq &\left\Vert \gamma _{1}\right\Vert _{p-var,\left[ t_{k},t_{k+1}\right]
}\left\Vert \gamma _{2}\right\Vert _{q-var,\left[ t_{k},t_{k+1}\right] }%
\text{.}
\end{eqnarray*}%
Therefore, using H\"{o}lder inequality, 
\begin{eqnarray}
&&\sum_{k}\left\Vert I^{D_{1}}\right\Vert _{1-var,\left[ t_{k},t_{k+1}\right]
}  \label{generalized Young 3} \\
&\leq &\left( \sum_{k}\left\Vert \gamma _{1}\right\Vert _{p-var,\left[
t_{k},t_{k+1}\right] }^{p}\right) ^{\frac{1}{p}}\left( \sum_{k}\left\Vert
\gamma _{2}\right\Vert _{q-var,\left[ t_{k},t_{k+1}\right] }^{q}\right) ^{%
\frac{1}{q}}  \notag \\
&\leq &m_{p}\left( \gamma _{1},\delta \right) m_{q}\left( \gamma _{2},\delta
\right) \leq C_{1}C_{2}m_{1}\left( \delta \right) m_{2}\left( \delta \right)
.  \notag
\end{eqnarray}

For $\sum_{k}\left\Vert I^{D_{2}}\right\Vert _{1-var,\left[ t_{k},t_{k+1}%
\right] }$. Applying Lemma \ref{Lemma estimation of area bisecting interval}
to $\left\Vert I^{D_{2}}\right\Vert _{1-var,\left[ t_{k},t_{k+1}\right] }$, $%
\forall k$, then there exists a finite partition $D^{\left( 1\right)
}=\left\{ u_{j}^{1}\right\} _{j}$, $\left\vert D^{\left( 1\right)
}\right\vert \leq 2^{-1}\delta $, s.t.%
\begin{eqnarray*}
&&\sum_{k}\left\Vert I^{D_{2}}\right\Vert _{1-var,\left[ t_{k},t_{k+1}\right]
} \\
&\leq &\sum_{j,u_{j}^{1}\in D^{\left( 1\right) }}\left\Vert
I^{D_{2}}\right\Vert _{1-var,\left[ u_{j}^{1},u_{j+1}^{1}\right]
}+2\sum_{k}\left\Vert \gamma _{1}\right\Vert _{p-var,\left[ t_{k},t_{k+1}%
\right] }\left\Vert \gamma _{2}\right\Vert _{q-var,\left[ t_{k},t_{k+1}%
\right] } \\
&\leq &\sum_{j,u_{j}^{1}\in D^{\left( 1\right) }}\left\Vert
I^{D_{2}}\right\Vert _{1-var,\left[ u_{j}^{1},u_{j+1}^{1}\right] }+2\left(
\sum_{k}\left\Vert \gamma _{1}\right\Vert _{p-var,\left[ t_{k},t_{k+1}\right]
}^{p}\right) ^{\frac{1}{p}}\left( \sum_{k}\left\Vert \gamma _{2}\right\Vert
_{q-var,\left[ t_{k},t_{k+1}\right] }^{q}\right) ^{\frac{1}{q}} \\
&\leq &\sum_{j,u_{j}^{1}\in D^{\left( 1\right) }}\left\Vert
I^{D_{2}}\right\Vert _{1-var,\left[ u_{j}^{1},u_{j+1}^{1}\right]
}+2C_{1}C_{2}m_{1}\left( \delta \right) m_{2}\left( \delta \right) .
\end{eqnarray*}%
Continue the process: applying Lemma \ref{Lemma estimation of area bisecting
interval} to $\left\Vert I^{D_{2}}\right\Vert _{1-var,\left[
u_{j}^{1},u_{j+1}^{1}\right] }$, $\forall j$, then there exists a finite
partition $D^{\left( 2\right) }=\left\{ u_{j}^{2}\right\} $, $\left\vert
D^{\left( 2\right) }\right\vert \leq 2^{-2}\delta $, s.t.%
\begin{equation*}
\sum_{j,u_{j}^{1}\in D^{\left( 1\right) }}\left\Vert I^{D_{2}}\right\Vert
_{1-var,\left[ u_{j}^{1},u_{j+1}^{1}\right] }\leq \sum_{j,u_{j}^{2}\in
D^{\left( 2\right) }}\left\Vert I^{D_{2}}\right\Vert _{1-var,\left[
u_{j}^{2},u_{j+1}^{2}\right] }+2C_{1}C_{2}m_{1}\left( \frac{\delta }{2}%
\right) m_{2}\left( \frac{\delta }{2}\right) .
\end{equation*}%
So on and so forth, and we get\ (for fixed $D_{2}$, $I^{D_{2}}$ is of
vanishing $1$-variation)%
\begin{equation}
\sum_{k}\left\Vert I^{D_{2}}\right\Vert _{1-var,\left[ t_{k},t_{k+1}\right]
}\leq 2C_{1}C_{2}\sum_{n=0}^{\infty }m_{1}\left( \frac{\delta }{2^{n}}%
\right) m_{2}\left( \frac{\delta }{2^{n}}\right) .
\label{generalized Young 4}
\end{equation}%
Since $m_{1}$ and $m_{2}$ are non-decreasing, so when $n\geq 1$,%
\begin{gather*}
m_{1}\left( \frac{\delta }{2^{n}}\right) m_{2}\left( \frac{\delta }{2^{n}}%
\right) \leq \left( \frac{\delta }{2^{n}}\right) ^{-1}\int_{\frac{\delta }{%
2^{n}}}^{\frac{\delta }{2^{n-1}}}m_{1}\left( t\right) m_{2}\left( t\right)
dt\leq 2\int_{\frac{\delta }{2^{n}}}^{\frac{\delta }{2^{n-1}}}\frac{%
m_{1}\left( t\right) m_{2}\left( t\right) }{t}dt. \\
\text{Thus }\sum_{n=0}^{\infty }m_{1}\left( \frac{\delta }{2^{n}}\right)
m_{2}\left( \frac{\delta }{2^{n}}\right) \leq m_{1}\left( \delta \right)
m_{2}\left( \delta \right) +2\int_{0}^{\delta }\frac{m_{1}\left( t\right)
m_{2}\left( t\right) }{t}dt.
\end{gather*}%
Combined with $\left( \ref{generalized Young 4}\right) $, 
\begin{equation}
\sum_{k}\left\Vert I^{D_{2}}\right\Vert _{1-var,\left[ t_{k},t_{k+1}\right] }%
\hspace{-0.02in}\leq \hspace{-0.02in}2C_{1}C_{2}\left( m_{1}\left( \delta
\right) m_{2}\left( \delta \right) \hspace{-0.02in}+\hspace{-0.02in}%
2\int_{0}^{\delta }\frac{m_{1}\left( t\right) m_{2}\left( t\right) }{t}%
dt\right) .  \label{generalized Young 5}
\end{equation}%
Therefore, combine $\left( \ref{generlized Young formula}\right) $, $\left( %
\ref{generalized Young 2}\right) $, $\left( \ref{generalized Young 3}\right) 
$ with $\left( \ref{generalized Young 5}\right) $, we get%
\begin{eqnarray*}
&&\left\Vert I^{D_{1}}-I^{D_{2}}\right\Vert _{1-var} \\
&\leq &C_{1}C_{2}\left( 2\left( m_{1}\left( \delta \right) +m_{2}\left(
\delta \right) \right) +3m_{1}\left( \delta \right) m_{2}\left( \delta
\right) +4\int_{0}^{\delta }\frac{m_{1}\left( t\right) m_{2}\left( t\right) 
}{t}dt\right) .
\end{eqnarray*}

In the above we assume $D_{2}\subset D_{1}$. For two general finite
partitions $D$ and $D^{\prime }$, $\left\vert D\right\vert \vee \left\vert
D^{\prime }\right\vert \leq \delta $, denote $D^{\prime \prime }:=D\cup
D^{\prime }$, apply our estimates to $D$, $D^{\prime \prime }$ and $%
D^{\prime }$, $D^{\prime \prime }$, we get%
\begin{eqnarray*}
&&\left\Vert I^{D}-I^{D^{\prime }}\right\Vert _{1-var} \\
&\leq &2C_{1}C_{2}\left( 2\left( m_{1}\left( \delta \right) +m_{2}\left(
\delta \right) \right) +3m_{1}\left( \delta \right) m_{2}\left( \delta
\right) +4\int_{0}^{\delta }\frac{m_{1}\left( t\right) m_{2}\left( t\right) 
}{t}dt\right) .
\end{eqnarray*}%
Because we assumed that $\lim_{t\rightarrow 0}m_{i}\left( t\right) =0$ and $%
\int_{0}^{1}\frac{m_{1}\left( t\right) m_{2}\left( t\right) }{t}dt<\infty $,
so the Riemann-Stieltjes integral $I\left( \gamma _{1},\gamma _{2}\right) $
exists, $I\left( \gamma _{1}^{D},\gamma _{2}^{D}\right) $ converge in $1$%
-variation to $I\left( \gamma _{1},\gamma _{2}\right) $ as $\left\vert
D\right\vert \rightarrow 0$, and ($\left\vert m_{i}\right\vert \leq 1$, $%
i=1,2$)%
\begin{equation*}
\sup_{D}\left\Vert I\left( \gamma _{1},\gamma _{2}\right) -I\left( \gamma
_{1}^{D},\gamma _{2}^{D}\right) \right\Vert _{1-var}\leq 2C_{1}C_{2}\left(
7+4\int_{0}^{1}\frac{m_{1}\left( t\right) m_{2}\left( t\right) }{t}dt\right)
.
\end{equation*}%
Moreover, if denote finite partition $D_{0}:=\left\{ 0,1\right\} $ then%
\begin{equation*}
\left\Vert I^{D_{0}}\right\Vert _{1-var}\leq \left\Vert \left( \gamma
_{1}\left( 1\right) -\gamma _{1}\left( 0\right) \right) \otimes \left(
\gamma _{2}\left( 1\right) -\gamma _{2}\left( 0\right) \right) \right\Vert
\leq C_{1}C_{2}.
\end{equation*}%
\begin{equation}
\text{Thus, }\left\Vert I\left( \gamma _{1},\gamma _{2}\right) \right\Vert
_{1-var}\leq C_{1}C_{2}\left( 15+8\int_{0}^{1}\frac{m_{1}\left( t\right)
m_{2}\left( t\right) }{t}dt\right)
\label{estimation of 1-var of I(gamma1, gamma2)}
\end{equation}

Then we work out $\left\Vert \int_{0}^{\cdot }\gamma _{1}\left( u\right)
\otimes d\gamma _{2}\left( u\right) \right\Vert _{q-var}$ from $\left\Vert
I\left( \gamma _{1},\gamma _{2}\right) \right\Vert _{1-var}$. Since%
\begin{eqnarray*}
I\left( \gamma _{1},\gamma _{2}\right) \left( s,t\right)
&:&=\int_{s}^{t}\left( \gamma _{1}\left( u\right) -\gamma _{1}\left(
s\right) \right) \otimes d\gamma _{2}\left( u\right) \\
&=&\int_{s}^{t}\gamma _{1}\left( u\right) \otimes d\gamma _{2}\left(
u\right) -\gamma _{1}\left( s\right) \otimes \left( \gamma _{2}\left(
t\right) -\gamma _{2}\left( s\right) \right)
\end{eqnarray*}%
Therefore, if define function $\beta :\bigtriangleup _{\left[ 0,1\right]
}\rightarrow \mathcal{V}_{1}\otimes \mathcal{V}_{2}$ by setting 
\begin{equation*}
\beta \left( s,t\right) :=\gamma _{1}\left( s\right) \otimes \left( \gamma
_{2}\left( t\right) -\gamma _{2}\left( s\right) \right) \text{, }\forall
\left( s,t\right) \in \bigtriangleup _{\left[ 0,1\right] }\text{.}
\end{equation*}%
Then 
\begin{equation*}
\left\Vert \beta \right\Vert _{q-var}\leq \left\Vert \gamma _{1}\right\Vert
_{\infty -var}\left\Vert \gamma _{2}\right\Vert _{q-var}\leq C_{1}C_{2}.
\end{equation*}%
Thus, combined with $\left( \ref{estimation of 1-var of I(gamma1, gamma2)}%
\right) $, we get 
\begin{eqnarray*}
\left\Vert \int_{0}^{\cdot }\gamma _{1}\left( u\right) \otimes d\gamma
_{2}\left( u\right) \right\Vert _{q-var} &\leq &\left\Vert I\left( \gamma
_{1},\gamma _{2}\right) \right\Vert _{1-var}+\left\Vert \beta \right\Vert
_{q-var} \\
&\leq &8C_{1}C_{2}\left( 2+\int_{0}^{1}\frac{m_{1}\left( t\right)
m_{2}\left( t\right) }{t}dt\right) .
\end{eqnarray*}%
Proof finishes.
\end{proof}

When $m_{1}\left( t\right) =t^{a}$, $m_{2}\left( t\right) =t^{b}$, $a>0$, $%
b>0$, we get Young integral.

The condition $\int_{0}^{1}\frac{m_{1}\left( t\right) m_{2}\left( t\right) }{%
t}dt<\infty $ is necessary in the sense of following example.

\bigskip

\noindent \textbf{Example} \textbf{\ref{Example integration finite is
necessary}}\ \ \ \textit{Suppose }$m_{i}:\left[ 0,1\right] \rightarrow 
\overline{%
\mathbb{R}
^{+}}$\textit{\ are two non-decreasing functions, satisfying }$%
\lim_{t\rightarrow 0}m_{i}\left( t\right) =0$\textit{, }$\left\vert
m_{i}\right\vert \leq 1$\textit{, }$i=1,2$\textit{, and }$\int_{0}^{1}\frac{%
m_{1}\left( t\right) m_{2}\left( t\right) }{t}dt=\infty $\textit{. Then for
any }$p>1$\textit{, }$q>1$\textit{, }$p^{-1}+q^{-1}=1$\textit{, there exist
two continuous real-valued paths }$\gamma _{i}:\left[ 0,1\right] \rightarrow 
\mathbb{R}
$\textit{, }$i=1,2$\textit{, s.t. }%
\begin{equation}
C_{1}:=\sup_{0\leq s<t\leq 1}\frac{\left\vert \gamma _{1}\left( t\right)
-\gamma _{1}\left( s\right) \right\vert }{\left\vert t-s\right\vert ^{\frac{1%
}{p}}m_{1}\left( t-s\right) }<\infty \text{, }C_{2}:=\sup_{0\leq s<t\leq 1}%
\frac{\left\vert \gamma _{2}\left( t\right) -\gamma _{2}\left( s\right)
\right\vert }{\left\vert t-s\right\vert ^{\frac{1}{q}}m_{2}\left( t-s\right) 
}<\infty ,  \label{Condition on gammai}
\end{equation}%
\textit{but the Riemann-Stieltjes integral }$\int_{0}^{1}\gamma _{1}\left(
t\right) d\gamma _{2}\left( t\right) $\textit{\ does not exist.}

\bigskip

\begin{proof}
\label{Proof of Example integration finite is necessary}Let $\epsilon _{k}=1$
or $-1$, $\forall k$, and define%
\begin{eqnarray*}
\gamma _{1}\left( t\right) &=&\sum_{k=1}^{\infty }\frac{m_{1}\left(
2^{-2k}\right) }{2^{\frac{2k}{p}}}\cos \left( 2\pi 2^{2k}t\right) \text{, }%
t\in \left[ 0,1\right] \text{,} \\
\gamma _{2}\left( t\right) &=&\sum_{k=1}^{\infty }\epsilon _{k}\frac{%
m_{2}\left( 2^{-2k}\right) }{2^{\frac{2k}{q}}}\sin \left( 2\pi
2^{2k}t\right) \text{, }t\in \left[ 0,1\right] \text{.}
\end{eqnarray*}%
Then $\gamma _{i}$ satisfy $\left( \ref{Condition on gammai}\right) $. Take $%
\gamma _{1}$ as an example. For $0\leq s<t\leq 1$, let $n=\left[ \log _{4}%
\frac{4}{\left\vert t-s\right\vert }\right] $, we have%
\begin{eqnarray}
&&\left\vert \gamma _{1}\left( t\right) -\gamma _{1}\left( s\right)
\right\vert  \label{inner modulus of continuity} \\
&\leq &2\pi \left( \sum_{k=1}^{n}m_{1}\left( 2^{-2k}\right) 2^{2\left( 1-%
\frac{1}{p}\right) k}\right) \left\vert t-s\right\vert
+2\sum_{k=n+1}^{\infty }\frac{m_{1}\left( 2^{-2k}\right) }{2^{\frac{2k}{p}}}.
\notag
\end{eqnarray}%
Since $\lim_{t\rightarrow 0}\frac{m_{1}\left( 4t\right) }{m_{1}\left(
t\right) }=1$ ($\int_{0}^{1}\frac{m_{i}\left( t\right) }{t}dt\geq
\int_{0}^{1}\frac{m_{1}\left( t\right) m_{2}\left( t\right) }{t}dt=\infty $
so $\lim_{t\rightarrow 0}\frac{m_{i}\left( t\right) }{\left( \ln \frac{1}{t}%
\right) ^{-2}}=\infty $), so using L'Hospital's rule, 
\begin{eqnarray*}
&&\lim_{n\rightarrow \infty }\frac{m_{1}\left( 2^{-2n}\right) 2^{2\left( 1-%
\frac{1}{p}\right) n}}{\sum_{k=1}^{n}m_{1}\left( 2^{-2k}\right) 2^{2\left( 1-%
\frac{1}{p}\right) k}} \\
&=&\lim_{n\rightarrow \infty }\frac{m_{1}\left( 2^{-2n}\right) 2^{2\left( 1-%
\frac{1}{p}\right) n}-m_{1}\left( 2^{-2\left( n-1\right) }\right) 2^{2\left(
1-\frac{1}{p}\right) \left( n-1\right) }}{m_{1}\left( 2^{-2n}\right)
2^{2\left( 1-\frac{1}{p}\right) n}} \\
&=&\lim_{t\rightarrow 0}\frac{m_{1}\left( t\right) 2^{2\left( 1-\frac{1}{p}%
\right) }-m_{1}\left( 4t\right) }{m_{1}\left( t\right) 2^{2\left( 1-\frac{1}{%
p}\right) }}=\frac{2^{2\left( 1-\frac{1}{p}\right) }-1}{2^{2\left( 1-\frac{1%
}{p}\right) }}.
\end{eqnarray*}%
Therefore, there exists constant $C_{1}$, s.t. for any $n\geq 1$, 
\begin{equation*}
\sum_{k=1}^{n}m_{1}\left( 2^{-2k}\right) 2^{2\left( 1-\frac{1}{p}\right)
k}\leq C_{1}m_{1}\left( 2^{-2n}\right) 2^{2\left( 1-\frac{1}{p}\right) n},
\end{equation*}%
Continue with $\left( \ref{inner modulus of continuity}\right) $, since $%
m_{1}$ is non-decreasing ($n=\left[ \log _{4}\frac{4}{t-s}\right] $ so $%
2^{-2n}<\left\vert t-s\right\vert \leq 2^{-2\left( n-1\right) }$)%
\begin{eqnarray*}
\left\vert \gamma _{1}\left( t\right) -\gamma _{1}\left( s\right)
\right\vert &\leq &2\pi C_{1}m_{1}\left( 2^{-2n}\right) 2^{2\left( 1-\frac{1%
}{p}\right) n}\left\vert t-s\right\vert +\frac{2}{2^{\frac{2}{p}}-1}%
m_{1}\left( 2^{-2n}\right) 2^{-\frac{2}{p}n} \\
&\leq &\left( 8\pi C_{1}+\frac{2}{2^{\frac{2}{p}}-1}\right) \left\vert
t-s\right\vert ^{\frac{1}{p}}m_{1}\left( t-s\right) .
\end{eqnarray*}

Then we prove the Riemann-Stieltjes integral $\int_{0}^{1}\gamma _{1}\left(
t\right) d\gamma _{2}\left( t\right) $ does not exist. First, the limit of
Riemann sum as $\left\vert D\right\vert \rightarrow 0$ does not depend on
the selection of representative points, because $\gamma _{1}\in C^{0,p-var}$%
, $\gamma _{2}\in C^{0,q-var}$. Actually, since $\gamma _{i}$ satisfy $%
\left( \ref{Condition on gammai}\right) $ and $m_{i}$ are non-decreasing, so 
$\omega _{p}\left( \gamma _{1},\delta \right) \leq C_{1}m_{1}\left( \delta
\right) $ and $\omega _{q}\left( \gamma _{2},\delta \right) \leq
C_{2}m_{2}\left( \delta \right) $. Suppose $D=\left\{ t_{k}\right\} $ is a
finite partition of $\left[ 0,1\right] $, then the error occurred to the
Riemann sum of $\int_{0}^{1}\gamma _{1}\left( t\right) d\gamma _{2}\left(
t\right) $ w.r.t. $D$ from selecting different representative points is
bounded by%
\begin{eqnarray*}
&&\sum_{k}\left\vert \gamma _{1}\left( t_{k+1}\right) -\gamma _{1}\left(
t_{k}\right) \right\vert \left\vert \gamma _{2}\left( t_{k+1}\right) -\gamma
_{2}\left( t_{k}\right) \right\vert \\
&\leq &\left( \sum_{k}\left\vert \gamma _{1}\left( t_{k+1}\right) -\gamma
_{1}\left( t_{k}\right) \right\vert ^{p}\right) ^{\frac{1}{p}}\left(
\sum_{k}\left\vert \gamma _{2}\left( t_{k+1}\right) -\gamma _{2}\left(
t_{k}\right) \right\vert ^{q}\right) ^{\frac{1}{q}} \\
&\leq &C_{1}C_{2}m_{1}\left( \left\vert D\right\vert \right) m_{2}\left(
\left\vert D\right\vert \right) ,
\end{eqnarray*}%
which tends to zero as $\left\vert D\right\vert \rightarrow 0$. On the other
hand, since%
\begin{eqnarray*}
&&\sum_{k,}\frac{1}{2}\left( \gamma _{1}\left( t_{k+1}\right) +\gamma
_{1}\left( t_{k}\right) \right) \left( \gamma _{2}\left( t_{k+1}\right)
-\gamma _{2}\left( t_{k}\right) \right) \\
&=&\frac{1}{2}\sum_{k}\left( \gamma _{1}\left( t_{k}\right) \gamma
_{2}\left( t_{k+1}\right) -\gamma _{2}\left( t_{k}\right) \gamma _{1}\left(
t_{k+1}\right) \right) +\frac{1}{2}\gamma _{1}\left( 1\right) \gamma
_{2}\left( 1\right) -\frac{1}{2}\gamma _{1}\left( 0\right) \gamma _{2}\left(
0\right) \text{,}
\end{eqnarray*}%
so the existence of Riemann-Stieltjes integral $\int_{0}^{1}\gamma
_{1}\left( t\right) d\gamma _{2}\left( t\right) $ is equivalent to the
existence of%
\begin{equation*}
\lim_{\left\vert D\right\vert \rightarrow 0}\sum_{k,t_{k}\in D}\left( \gamma
_{1}\left( t_{k}\right) \gamma _{2}\left( t_{k+1}\right) -\gamma _{2}\left(
t_{k}\right) \gamma _{1}\left( t_{k+1}\right) \right) .
\end{equation*}%
Similar as the estimates in Example \ref{Example non-existence of Reimann
integral}, if denote finite partition $D_{2N}=\left\{ t_{l}^{N}\right\} $
where $t_{l}^{N}:=l2^{-2N}$, $l=0,1,\dots ,2^{2N}$, we get%
\begin{align}
\left\langle \int \gamma _{1}d\gamma _{2},D_{2N}\right\rangle &
:=\sum_{l=0}^{2^{2N}-1}\left( \gamma _{1}\left( t_{l}^{N}\right) \gamma
_{2}\left( t_{l+1}^{N}\right) -\gamma _{2}\left( t_{l}^{N}\right) \gamma
_{1}\left( t_{l+1}^{N}\right) \right)  \label{inner expression 1} \\
& =\sum_{k=1}^{N-1}\epsilon _{k}\frac{m_{1}\left( 2^{-2k}\right) m_{2}\left(
2^{-2k}\right) }{2^{2k-2N}}\sin \left( 2\pi 2^{2k-2N}\right) .  \notag
\end{align}%
While since $m_{i}$ are non-decreasing, so for any $k\geq 1$, 
\begin{equation*}
m_{1}\left( 2^{-2k}\right) m_{2}\left( 2^{-2k}\right) \geq \frac{1}{3}%
\int_{2^{-2\left( k+1\right) }}^{2^{-2k}}\frac{m_{1}\left( t\right)
m_{2}\left( t\right) }{t}dt,
\end{equation*}%
so based on our assumption $\int_{0}^{1}\frac{m_{1}\left( t\right)
m_{2}\left( t\right) }{t}dt=\infty $, we have 
\begin{equation*}
\sum_{k=1}^{\infty }m_{1}\left( 2^{-2k}\right) m_{2}\left( 2^{-2k}\right)
=\infty \text{.}
\end{equation*}%
Thus, since $m_{i}$ are non-decreasing and $\lim_{t\rightarrow 0}m_{i}\left(
t\right) =0$, so using exactly the same estimates as in Example \ref{Example
non-existence of Reimann integral}, for any sequence of strictly increasing
integers $\left\{ l_{n}\right\} $ satisfying for some $c>\pi $ 
\begin{gather*}
c^{n}\leq \sum_{k=l_{n}}^{l_{n+1}-1}m_{1}\left( 2^{-2k}\right) m_{2}\left(
2^{-2k}\right) \leq c^{n}+1\text{, }\forall n\geq 1, \\
\text{we let }\epsilon _{k}=\left( -1\right) ^{n}\text{, when }l_{n}\leq
k\leq l_{n+1}-1\text{.}
\end{gather*}%
Then, for any $a\in \left[ -\infty ,\infty \right] $, there exists a finite
partition $\left\{ D_{n}^{a}\right\} _{n}\subset \left\{ D_{2N}\right\} _{N}$%
, $\lim_{n\rightarrow \infty }\left\vert D_{n}^{a}\right\vert =0$, but $%
\lim_{n\rightarrow \infty }\left\langle \int \gamma _{1}d\gamma
_{2},D_{n}^{a}\right\rangle =a$.
\end{proof}

Next, we want to prove that a vanishing $2$-variation path $\gamma $ can be
enhanced into a geometric $2$-rough path, if and only if $A\left( \gamma
^{D}\right) $ (the areas of piecewisely linear approximation) converge in $1$%
-variation as $\left\vert D\right\vert \rightarrow 0$.

\begin{lemma}
\label{Lemma difference of 1-var of area copy(1)}Suppose $\gamma \in
C^{0,2-var}\left( \left[ 0,T\right] ,\mathcal{V}\right) $. $D_{1}=\left\{
t_{k}\right\} _{k}$ and $D_{2}=\left\{ s_{j}\right\} _{j}$ are two finite
partitions of $\left[ 0,T\right] $, and $D_{2}$ is a refinement of $D_{1}$,
i.e. for any $k$, there exist integers $n_{k}<n_{k+1}$, s.t. $%
t_{k}=s_{n_{k}}<s_{n_{k}+1}<\dots <s_{n_{k+1}}=t_{k+1}$. Then if $\left\vert
D_{1}\right\vert \leq \delta $, we have 
\begin{equation*}
\left\Vert A\left( \gamma ^{D_{1}}\right) -A\left( \gamma ^{D_{2}}\right)
\right\Vert _{1-var}\leq \sum_{k}\left\Vert A\left( \gamma ^{D_{2}}\right)
\right\Vert _{1-var,\left[ t_{k},t_{k+1}\right] }+4\left\Vert \gamma
\right\Vert _{2-var,\left[ 0,T\right] }\omega _{2}\left( \gamma ,\delta
\right) .
\end{equation*}
\end{lemma}

\begin{proof}
Almost the same as that of Lemma \ref{Lemma difference of 1-var of area}
when $p=q=2$, by using $\left\Vert \left[ u,v\right] \right\Vert \leq
2\left\Vert u\right\Vert \left\Vert v\right\Vert $. $\sum_{k}\left\Vert
A^{D_{1}}\right\Vert _{1-var,\left[ t_{k},t_{k+1}\right] }=0$ because $%
\gamma ^{D_{1}}$ is linear on $\left[ t_{k},t_{k+1}\right] $, $\forall k$.
\end{proof}

\begin{lemma}
\label{Lemma relations path and area}Suppose $\left( \gamma ,\alpha \right) $
is a weak geometric $2$-rough path, and $D=\left\{ t_{k}\right\} _{k=0}^{n}$
is a finite partition of $\left[ 0,T\right] $. Then%
\begin{gather*}
A\left( \gamma ^{D}\right) \left( 0,T\right) =\frac{1}{2}\sum_{k=0}^{n-1}%
\left[ \gamma \left( t_{k}\right) ,\gamma \left( t_{k+1}\right) \right] -%
\frac{1}{2}\left[ \gamma \left( 0\right) ,\gamma \left( T\right) \right] , \\
\alpha \left( 0,T\right) =\sum_{k=0}^{n-1}\alpha \left( t_{k},t_{k+1}\right)
+A\left( \gamma ^{D}\right) \left( 0,T\right) .
\end{gather*}
\end{lemma}

\begin{proof}
The first is obtained from directly computation, the second is got by using
multiplicativity of $\left( \gamma ,\alpha \right) $ (i.e.$\left( \ref%
{property of multiplicativity}\right) $).
\end{proof}

\medskip \medskip

\noindent \textbf{Theorem} \textbf{\ref{Theorem condition to be enhancible}}%
\ \ \ \textit{Suppose }$\gamma \in C^{0,2-var}\left( \left[ 0,T\right] ,%
\mathcal{V}\right) $\textit{. Then }$\gamma \in \mathcal{G}_{2}\left( 
\mathcal{V}\right) $\textit{\ if and only if }$A\left( \gamma ^{D}\right) $%
\textit{\ converges in }$1$\textit{-variation as }$\left\vert D\right\vert
\rightarrow 0$\textit{.}

\bigskip

\begin{proof}
\label{Proof of Theorem condition to be enhancible}$\Leftarrow $ is clear;
we prove $\Rightarrow $. \ Suppose $\left( \gamma ,\alpha \right) $ is a
geometric $2$-rough path, so $\gamma $ is of vanishing $2$-variation, $%
\alpha $ is of vanishing $1$-variation. Thus, for any $\epsilon >0$, there
exists $\delta >0$, s.t. for any finite partition $D$ of $\left[ 0,T\right] $
satisfying $\left\vert D\right\vert \leq \delta $, $\sum_{k,t_{k}\in
D}\left\Vert \gamma \right\Vert _{2-var,\left[ t_{k},t_{k+1}\right]
}^{2}<\epsilon $ and $\sum_{k,t_{k}\in D}\left\Vert \alpha \left(
t_{k},t_{k+1}\right) \right\Vert <\epsilon $.

Suppose $D_{1}=\left\{ t_{k}\right\} _{k}$ and $D_{2}=\left\{ s_{j}\right\}
_{j}$ are two finite partitions of $\left[ 0,T\right] $ satisfying $%
\left\vert D_{1}\right\vert \leq \delta $, $\left\vert D_{2}\right\vert \leq
\delta $, $D_{2}$ is a refinement of $D_{1}$. Based on Lemma \ref{Lemma
difference of 1-var of area copy(1)}, 
\begin{equation*}
\left\Vert A\left( \gamma ^{D_{1}}\right) -A\left( \gamma ^{D_{2}}\right)
\right\Vert _{1-var}\leq \sum_{k}\left\Vert A\left( \gamma ^{D_{2}}\right)
\right\Vert _{1-var,\left[ t_{k},t_{k+1}\right] }+4\left\Vert \gamma
\right\Vert _{2-var,\left[ 0,T\right] }\epsilon ^{\frac{1}{2}}\text{.}
\end{equation*}%
For $\sum_{k}\left\Vert A\left( \gamma ^{D_{2}}\right) \right\Vert _{1-var,%
\left[ t_{k},t_{k+1}\right] }$. Since $\gamma ^{D_{2}}$ is a piecewisely
linear path on each $\left[ t_{k},t_{k+1}\right] $, so we only consider
finite partitions, whose points are all "corner" points. Suppose $D_{3}$ is
a finite partition satisfying $D_{1}\subset D_{3}=\left\{ u_{i}\right\}
\subset D_{2}$. Suppose $u_{i}=s_{m_{i}}<s_{m_{i}+1}<\dots
<s_{m_{i+1}}=u_{i+1}$, then based on Lemma \ref{Lemma relations path and
area}, for each $i$,%
\begin{equation*}
\left\Vert A\left( \gamma ^{D_{2}}\right) \left( u_{i},u_{i+1}\right)
\right\Vert \leq \left\Vert \alpha \left( u_{i},u_{i+1}\right) \right\Vert
+\sum_{j=m_{i}}^{m_{i+1}-1}\left\Vert \alpha \left( s_{j},s_{j+1}\right)
\right\Vert
\end{equation*}%
Sum over $i$, then%
\begin{equation*}
\sum_{i,u_{i}\in D_{3}}\left\Vert A\left( \gamma ^{D_{2}}\right) \left(
u_{i},u_{i+1}\right) \right\Vert \leq \sum_{i,u_{i}\in D_{3}}\left\Vert
\alpha \left( u_{i},u_{i+1}\right) \right\Vert
+\sum_{i}\sum_{j=m_{i}}^{m_{i+1}-1}\left\Vert \alpha \left(
s_{j},s_{j+1}\right) \right\Vert .
\end{equation*}%
Since $\left\vert D_{2}\right\vert \leq \left\vert D_{3}\right\vert \leq
\left\vert D_{1}\right\vert \leq \delta $, so as we assumed, $%
\sum_{i,u_{i}\in D_{3}}\left\Vert \alpha \left( u_{i},u_{i+1}\right)
\right\Vert <\epsilon $, $\sum_{i}\sum_{j=m_{i}}^{m_{i+1}-1}\left\Vert
\alpha \left( s_{j},s_{j+1}\right) \right\Vert =\sum_{j,s_{j}\in
D_{2}}\left\Vert \alpha \left( s_{j},s_{j+1}\right) \right\Vert <\epsilon $.
Thus%
\begin{equation*}
\sum_{i,u_{i}\in D_{3}}\left\Vert A\left( \gamma ^{D_{2}}\right) \left(
u_{i},u_{i+1}\right) \right\Vert <2\epsilon .
\end{equation*}%
Therefore, taking supremum over all possible $D_{3}$, we get 
\begin{equation*}
\sum_{k}\left\Vert A\left( \gamma ^{D_{2}}\right) \right\Vert _{1-var,\left[
t_{k},t_{k+1}\right] }\leq 2\epsilon \text{.}
\end{equation*}%
Thus%
\begin{equation*}
\left\Vert A\left( \gamma ^{D_{1}}\right) -A\left( \gamma ^{D_{2}}\right)
\right\Vert _{1-var}\leq 2\epsilon +4\left\Vert \gamma \right\Vert _{2-var, 
\left[ 0,T\right] }\epsilon ^{\frac{1}{2}}\text{.}
\end{equation*}

For any finite partition $D$ and $D^{\prime }$, denote $D^{\prime \prime
}=D\cup D^{\prime }$, and use the above estimates for $D$, $D^{\prime \prime
}$ and $D^{\prime }$, $D^{\prime \prime }$. Proof finishes.
\end{proof}

Therefore, if a vanishing $2$-variation path $\gamma $ can be enhanced into
a geometric weak geometric $2$-rough path, then $A\left( \gamma ^{D}\right) $
converge in $1$-variation as $\left\vert D\right\vert \rightarrow 0$, so
converge pointwisely to the Riemann-Stieltjes integral $2^{-1}\int_{s}^{t}%
\left[ \gamma \left( u\right) -\gamma \left( s\right) ,d\gamma \left(
u\right) \right] $.

\bigskip

\noindent \textbf{Theorem \ref{Theorem condition to be enhancible into a
2-rough path}} \ \ \ \ Suppose $\gamma \in C^{2-var}\left( \left[ 0,T\right]
,\mathcal{V}\right) $. Then $\gamma $ can be enhanced into a weak geometric $%
2$-rough path if and only if 
\begin{equation*}
\sup_{D}\left\Vert A\left( \gamma ^{D}\right) \right\Vert _{1-var,\left[ 0,T%
\right] }<\infty \text{ and }\left\{ A\left( \gamma ^{D}\right) \right\} _{D}%
\text{ are equicontinuous.}
\end{equation*}

\medskip

\begin{proof}
\label{Proof of Theorem condition to be enhancible into a 2-rough path}$%
\Leftarrow $ Suppose $\left\{ D_{n}\right\} _{n}$ is a sequence of finite
partitions of $\left[ 0,T\right] $ satisfying $\lim_{n\rightarrow \infty
}\left\vert D_{n}\right\vert =0$. Since $\left\{ A\left( \gamma
^{D_{n}}\right) \right\} _{n}$ are uniformly bounded and equicontinuous, so
based on Arzel\`{a}-Ascoli theorem, there exists a subsequence $\left\{
A\left( \gamma ^{D_{n_{k}}}\right) \right\} _{k}$\ which converge in uniform
norm. Denote the limit as $\alpha $.

$\gamma $ is continuous, so $\gamma ^{D_{n_{k}}}$ converge to $\gamma $ in
uniform norm as $k$ tends to infinity. Since multiplicativity is preserved
under pointwise convergence, $\left( \gamma ,\alpha \right) $ is
multiplicative. On the other hand, use the lower semi-continuity of $p$%
-variation, 
\begin{equation*}
\left\Vert \alpha \right\Vert _{1-var,\left[ 0,T\right] }\leq \underline{%
\lim }_{n\rightarrow \infty }\left\Vert A\left( \gamma ^{D_{n}}\right)
\right\Vert _{1-var,\left[ 0,T\right] }\leq \sup_{D}\left\Vert A\left(
\gamma ^{D}\right) \right\Vert _{1-var,\left[ 0,T\right] }<\infty .
\end{equation*}%
Thus, $\left( \gamma ,\alpha \right) $ is a weak geometric $2$-rough path.

$\Rightarrow $ Suppose $\left( \gamma ,\alpha \right) $ is a weak geometric $%
2$-rough path. Fix finite partition $D=\left\{ t_{k}\right\} $ and $\left(
s,t\right) \in \bigtriangleup _{\left[ 0,T\right] }$. Suppose $%
t_{k_{1}-1}<s\leq t_{k_{1}}\leq t_{k_{2}}\leq t<t_{k_{2}+1}$, then based on
Lemma \ref{Lemma relations path and area}%
\begin{eqnarray}
&&\left\Vert A\left( \gamma ^{D}\right) \left( s,t\right) \right\Vert
\label{inner inequality equicontinuity of A(gamma^D)} \\
&\leq &\left\Vert \alpha \left( s,t\right) \right\Vert +\left\Vert \alpha
\left( s,t_{k_{1}}\right) \right\Vert +\sum_{k=k_{1}}^{k_{2}-1}\left\Vert
\alpha \left( t_{k},t_{k+1}\right) \right\Vert +\left\Vert \alpha \left(
t_{k_{2}},t\right) \right\Vert  \notag \\
&\leq &2\left\Vert \alpha \right\Vert _{1-var,\left[ s,t\right] }.  \notag
\end{eqnarray}%
Thus $\left\{ A\left( \gamma ^{D}\right) \right\} _{D}$ are equicontinuous,
and (based on $\left( \ref{inner inequality equicontinuity of A(gamma^D)}%
\right) $),%
\begin{equation*}
\sup_{D}\left\Vert A\left( \gamma ^{D}\right) \right\Vert _{1-var,\left[ 0,T%
\right] }\leq 2\left\Vert \alpha \right\Vert _{1-var,\left[ 0,T\right]
}<\infty .
\end{equation*}
\end{proof}

\begin{lemma}
\label{Lemma estimation of area bisecting interval copy(1)}Suppose $\gamma :%
\left[ 0,T\right] \rightarrow \mathcal{V}$, is a continuous finitely
piecewise linear path. Then for any $p>1$, $q>1$, $p^{-1}+q^{-1}=1$, there
exists finite partition $D=\left\{ t_{k}\right\} $ s.t. $\left\vert
D\right\vert \leq 2^{-1}T$ and%
\begin{equation*}
\left\Vert A\left( \gamma \right) \right\Vert _{1-var,\left[ 0,T\right]
}\leq \sum_{k,t_{k}\in D}\left\Vert A\left( \gamma \right) \right\Vert
_{1-var,\left[ t_{k},t_{k+1}\right] }+2\left\Vert \gamma _{1}\right\Vert
_{2-var,\left[ 0,T\right] }^{2}.
\end{equation*}
\end{lemma}

\begin{proof}
Almost the same as that of Lemma \ref{Lemma estimation of area bisecting
interval} when $p=q=2$, by using $\left\Vert \left[ u,v\right] \right\Vert
\leq 2\left\Vert u\right\Vert \left\Vert v\right\Vert $.
\end{proof}

\medskip \medskip

\noindent \textbf{Theorem \ref{Theorem larger space of enhancible paths} \ }%
\textit{Let }$\gamma :\left[ 0,1\right] \rightarrow \mathcal{V}$\textit{\ be
a continuous paths. Then if there exists an non-decreasing function }$m:%
\left[ 0,1\right] \rightarrow \overline{%
\mathbb{R}
^{+}}$\textit{\ satisfying}%
\begin{equation*}
\lim_{t\rightarrow 0}m\left( t\right) =0\text{, }m\left( 1\right) \leq 1%
\text{, and }\int_{0}^{1}\frac{m^{2}\left( t\right) }{t}dt<\infty ,
\end{equation*}%
\textit{such that }%
\begin{equation*}
\sup_{0\leq s<t\leq 1}\frac{\left\Vert \gamma \left( t\right) -\gamma \left(
s\right) \right\Vert }{\left\vert t-s\right\vert ^{\frac{1}{2}}m\left(
t-s\right) }<\infty .
\end{equation*}%
\textit{Then }$\gamma \in \mathcal{G}_{2}\left( \mathcal{V}\right) $.

\bigskip

\begin{proof}
\label{Proof of Theorem larger space of enhancible paths}Denote%
\begin{equation*}
C:=\sup_{0\leq s<t\leq 1}\frac{\left\Vert \gamma \left( t\right) -\gamma
\left( s\right) \right\Vert }{\left\vert t-s\right\vert ^{\frac{1}{2}%
}m\left( t-s\right) }<\infty \text{.}
\end{equation*}%
Then $\left\Vert \gamma \right\Vert _{2-var,\left[ 0,T\right] }\leq C$, $%
\omega _{2}\left( \gamma ,\delta \right) \leq Cm\left( \delta \right) $.
Using Lemma \ref{Lemma difference of 1-var of area copy(1)},%
\begin{eqnarray*}
\left\Vert A\left( \gamma ^{D_{1}}\right) -A\left( \gamma ^{D_{2}}\right)
\right\Vert _{1-var} &\leq &\sum_{k}\left\Vert A\left( \gamma
^{D_{2}}\right) \right\Vert _{1-var,\left[ t_{k},t_{k+1}\right]
}+4\left\Vert \gamma \right\Vert _{2-var,\left[ 0,T\right] }\omega
_{2}\left( \gamma ,\delta \right) . \\
&\leq &\sum_{k}\left\Vert A\left( \gamma ^{D_{2}}\right) \right\Vert _{1-var,
\left[ t_{k},t_{k+1}\right] }+4C^{2}m\left( \delta \right) .
\end{eqnarray*}%
While, apply Lemma \ref{Lemma estimation of area bisecting interval copy(1)}
to bisect intervals, and use similar reasoning as that lead to $\left( \ref%
{generalized Young 4}\right) $ in proof of Theorem \ref{Theorem generalized
Young integral} (starting from page \pageref{Proof of theorem generalized
Young integral}), we get%
\begin{gather*}
\sum_{k}\left\Vert A^{D_{2}}\right\Vert _{1-var,\left[ t_{k},t_{k+1}\right]
}\leq 2C^{2}\sum_{n=0}^{\infty }m^{2}\left( \frac{\delta }{2^{n}}\right)
\leq 2C^{2}\left( m^{2}\left( \delta \right) +2\int_{0}^{\delta }\frac{%
m^{2}\left( t\right) }{t}dt\right) . \\
\text{Thus, }\left\Vert A\left( \gamma ^{D_{1}}\right) -A\left( \gamma
^{D_{2}}\right) \right\Vert _{1-var}\leq 2C^{2}\left( 2m\left( \delta
\right) +m^{2}\left( \delta \right) +2\int_{0}^{\delta }\frac{m^{2}\left(
t\right) }{t}dt\right) .
\end{gather*}%
Since $\lim_{\delta \rightarrow 0}m\left( \delta \right) =0$ and $%
\int_{0}^{1}\frac{m^{2}\left( t\right) }{t}dt<\infty $, so $A\left( \gamma
^{D}\right) $ converge in $1$-variation as $\left\vert D\right\vert
\rightarrow 0$. Based on Theorem \ref{Theorem condition to be enhancible}, $%
\gamma $ is in $\mathcal{G}_{2}\left( \mathcal{V}\right) $.
\end{proof}

\bigskip

\textbf{Acknowledgement} \ \ Thanks are due to the extremely patient
guidance from and enlightening discussions with my supervisor Prof. Terry J.
Lyons.

\end{document}